\documentclass[twoside]{article}

\usepackage{geometry}
 \usepackage{fancyhdr}
\usepackage[square,numbers,super,sort&compress]{natbib}

\makeatletter
\renewcommand\NAT@citesuper[3]{\ifNAT@swa
\if*#2*\else#2\NAT@spacechar\fi
\unskip\kern\p@\textsuperscript{\NAT@@open#1\if*#3*\else,\NAT@spacechar#3\fi\NAT@@close}%
   \else #1\fi\endgroup}
\makeatother

\usepackage{amsfonts, dsfont}
\usepackage{amsmath, amssymb}
\usepackage{mathtools}
\usepackage{amsthm}
\usepackage{stmaryrd}

\usepackage{multicol}

\usepackage{xcolor}
\usepackage{array}
\usepackage{booktabs}
\usepackage[final]{graphicx}
\usepackage{enumitem}
\usepackage{varwidth}

\usepackage[
  normal,
  font={small,color=black},
  labelfont=bf,
  margin=2em,
  labelsep=period]{caption}[2008/04/01]

\renewcommand{\div}{\text{div}}
\newcommand{\sfrei}{}
\newcommand{\red}{\color{red}}
\newcommand{\bl}{\color{blue}}

\renewcommand{\kappa}{\varkappa}

 \newcommand{\jump}[1]{\llbracket #1 \rrbracket }

\newtheorem{lemma}{Lemma}[section]
\newtheorem{theorem}[lemma]{Theorem}
\newtheorem{corollary}[lemma]{Corollary}
\newtheorem{remark}[lemma]{Remark}
\newtheorem{assumption}[lemma]{Assumption}
\newtheorem*{acknowledgement}{Acknowledgement}

\usepackage{bm}
\renewcommand{\u}{\bm{u}}
\newcommand{\w}{\bm{w}}
\newcommand{\e}{\bm{e}}
\newcommand{\z}{\bm{z}}
\newcommand{\n}{\bm{n}}
\renewcommand{\v}{\bm{v}}
\newcommand{\f}{\bm{f}}
\newcommand{\T}{\bm{T}}
\newcommand{\etab}{\bm{\eta}}
\newcommand{\xib}{\bm{\xi}}
\newcommand{\psib}{\bm{\psi}}
\newcommand{\phib}{\bm{\phi}}

\begin{document}

\title{Eulerian time-stepping schemes for the non-stationary Stokes equations on time-dependent domains}

\author{Erik Burman\thanks{Department of Mathematics, University College London, Gower Street, WC1E 6BT, London, UK
(e.burman@ucl.ac.uk)}
\and Stefan Frei\thanks{Department of Mathematics and Statistics, University of Konstanz,
Universit\"atsstra\ss e 10, 78457 Konstanz, Germany}
 \and
Andre Massing\thanks{
Department of Mathematical Sciences, Norwegian University of Science and Technology, NO-7491 Trondheim, Norway, (andre.massing@ntnu.no)}
\thanks{Department of Mathematics and Mathematical Statistics, Ume{\aa} University, SE-90187 Ume{\aa}, Sweden,
  (andre.masssing@umu.se)
}
}

\date{\today}

\maketitle

\begin{abstract}
 This article is concerned with the discretisation of the Stokes equations on time-dependent domains in an Eulerian coordinate framework.
 Our work can be seen as an extension of a recent paper by Lehrenfeld \& Olshanskii [ESAIM: M2AN, 53(2):\,585-614, 2019], where BDF-type 
 time-stepping schemes are studied for a parabolic equation on
 moving domains. For space discretisation, a geometrically unfitted finite element discretisation is applied in combination with Nitsche's method
 to impose boundary conditions.
 Physically undefined values of the solution at previous time-steps are extended implicitly by means of 
 so-called \textit{ghost penalty} stabilisations. We derive a complete \textit{a priori} error analysis of the discretisation error
 in space and time, including optimal $L^2(L^2)$-norm error bounds for the velocities. Finally, the theoretical results are substantiated
 with numerical examples.
\end{abstract}

\section{Introduction}

Flows on moving domains $\Omega(t)\subset \mathbb{R}^d$ ($d=2,3$) need to be considered in many different applications. Examples include 
particulate flows 
or flows around moving objects like biological or mechanical valves, wind turbines or parachutes. Strongly related problems are fluid-structure interactions or multi-phase flows.

There exists a vast literature on time discretisation of the 
non-stationary Stokes or Navier-Stokes equations on fixed domains, 
see for example the classical works of Girault \& Raviart~\cite{GiraultRaviart79}, 
Baker et al~\cite{BakerDougalisKarakashian} and Rannacher \& Heywood~\cite{HeywoodRannacherIV}, 
or more recently Bochev et al.~\cite{Bochevetal2007} and 
Burman \& Fern\'andez~\cite{BurmanFernandez2007, BurmanFernandez2008} in the context 
of stabilised finite element methods. If the computational domain remains unchanged in each time-step, the same 
spatial discretisation can be used (unless adaptive mesh 
refinement is considered) and finite difference schemes based on the method of lines can be applied for time discretisation.

In the case of moderate domain movements, these techniques can be transferred to the moving framework by using the \textit{Arbitrary Lagrangian Eulerian} (ALE) approach~\cite{Noh1964, FrankLazarus1964, DoneaSurvey2004}.
Here, the idea is to  formulate an equivalent system of equations on a fixed reference 
configuration $\hat{\Omega}$, for example the initial configuration $\Omega(0)$, by means of a time-dependent map $\T(t):\hat{\Omega} \to \Omega(t)$. This technique  
has been used widely for flows on moving domains, see e.g,.~\cite{DoneaSurvey2004, CauchaFreiRubio2018} and fluid-structure interactions~\cite{BazilevsBook, RichterBuch, FSISammelband}. 
The analysis of the time discretisation error is then very similar to the fixed framework, as all
quantities and equations are formulated on the same reference domain $\hat{\Omega}$, see e.g.~\cite{RichterWick2015}. {\sfrei For a detailed stability analysis of ALE formulations, we refer to Nobile \& Formaggia~\cite{NobileFormaggia99} and Boffi \& Gastaldi~\cite{BoffiGastaldi2004}.}

On the other hand, it is well-known that the ALE method is less practical in the case of large domain 
deformations~\cite{RichterBuch, FreiPhD}. This is due to the degeneration of mesh 
elements both in a finite element and a finite difference context. A re-meshing of the domain $\Omega(t)$ becomes necessary. Moreover, topology changes, for example due to contact of particles within the flow or of a particle 
with an outer wall~\cite{BurmanFernandezFrei2018}, 
are not allowed, as the map between $\hat{\Omega}$ and $\Omega(t)$ can not have the required regularity in this situation.

In such cases an Eulerian formulation of the problem formulated on the moving domains $\Omega(t)$ is preferable. This is also the standard coordinate framework for the simulation of multi-phase flows.
In the last years a variety of space discretisation techniques have been designed to resolve curved or moving boundaries accurately. Examples include 
the cut finite element method~\cite{HansboHansbo2002, Massing2012, CutFEM2015, MassingLarsonLoggEtAl2013a,BurmanClausMassing2015,
MassingSchottWall2017,GuzmanOlshanskii2016} 
within a 
fictitious domain approach, extended finite elements~\cite{XFEM,GrossReusken2007,HansboLarsonZahedi2014a,ChessaBelytschko2003} or locally fitted finite element techniques~\cite{FreiRichter14}, to name such a few of the approaches. 

Much less analysed 
is a proper time discretisation of the problem.
In the case of moving domains, standard time discretisation based on the method of lines is not applicable in a straight-forward way. 
The reason is that the domain of definition of the variables changes from time step to time step.

As an example consider the \textit{backward Euler} discretisation of the time derivative within a variational formulation
\begin{align*}
 (\partial_t u_h(t_n), \phi_h^n)_{\Omega(t_n)} \approx \frac{1}{\Delta t} (u_h(t_n)- u_h(t_{n-1}), \phi_h^n)_{\Omega(t_n)}.
\end{align*}
Note that $u_h(t_{n-1})$ is only well-defined on $\Omega(t_{n-1})$, but is needed on $\Omega(t_n)$.

One solution to this dilemma are so-called characteristic-based approaches~\cite{HechtPironneau2016}. Similar time-stepping schemes result when applying the ALE method only locally within 
one time-step and projecting the system back to the original reference frame after each step~\cite{Codina2009}, or based on Galerkin time discretisations with modified Galerkin spaces~\cite{FreiRichter17}.
The disadvantage of these approaches is a 
projection between the domains $\Omega(t_{n-1})$ and $\Omega(t_n)$ that needs to be computed within each or after a certain number of steps.

Another possibility consists of space-time approaches~\cite{HansboLarsonZahedi2016, Lehrenfeld2015}, where a $d+1$-dimensional domain is discretised if $\Omega(t)\subset \mathbb{R}^d$. 
The computational requirements of these approaches 
might, however, exceed the available computational resources, in particular within complex three-dimensional applications. 
Moreover, the implementation of higher-dimensional discretisations and accurate quadrature formulas pose additional challenges. 

A simpler approach has been proposed recently in the dissertation of Schott~\cite{SchottDiss} and 
{\sfrei by Lehrenfeld \& Olshanskii~\cite{LehrenfeldOlshanskii}}. Here, the idea is to define
extensions of the solution $\u_h(t_{n-1})$ from the previous time-step to a domain that spans at least $\Omega(t_n)$. On the finite element level these extensions can be incorporated implicitly in the 
time-stepping scheme by using so-called \textit{ghost penalty} stabilisations~\cite{Burman2010} to a sufficiently large domain $\Omega_\delta(t_{n-1}) \supset \Omega(t_n)$. These techniques have originally been proposed to extend 
the coercivity of elliptic bilinear forms from the physical to the computational domain in the context of CutFEM or fictitious domain approaches~\cite{Burman2010}.

{\sfrei While Schott used such an extension explicitly after each time step to define values for $\u_h(t_{n-1})$ in mesh nodes lying in $\Omega(t_n)\setminus \Omega(t_{n-1})$, Lehrenfeld \& Olshanskii included the extension operator implicitly within each time step by solving a combined discrete system including the extension operator on the larger computational domain $\Omega_\delta(t_n)$. For the latter approach a complete analysis could be given 
for the
corresponding backward Euler time discretisation, showing first-order convergence in time in the spatial energy norm~\cite{LehrenfeldOlshanskii}.} 
Moreover, the authors gave hints on how to transfer the argumentation 
to a backward difference scheme (BDF(2)), which results in second-order convergence. We should also mention that similar time discretisation 
techniques have been used previously in the context of surface PDEs~\cite{OlshanskiiXu2017, LehrenfeldOlshanskiiXu2018},
and mixed-dimensional surface-bulk problems~\cite{HansboLarsonZahedi2016}
on moving domains.

In this work, we apply such an approach to the discretisation of the non-stationary Stokes equations on a moving domain, including a complete analysis of the space and time discretisation errors. Particular problems are related to 
the approximation of the pressure variable.
It is well-known that stability of the pressure is lost in the case of fixed domains, when the discretisation changes from one time-step to another. This can already be observed, when the finite element mesh is refined 
or coarsened globally at some instant of time, see Besier \& Wollner~\cite{BesierWollner} and is due to the fact that the old solution 
$\u_h^{n-1}:=\u_h(t_{n-1})$ is not discrete divergence free with 
respect to the new mesh. Possible remedies include the use of Stokes or Darcy projections~\cite{BesierWollner, BraackTaschenberger2013} to pass $\u_h^{n-1}$
to the new mesh. Our analysis will reveal that similar issues hold true for the case of moving domains, even if the same discretisation is used on $\Omega(t_n) \cap \Omega(t_{n-1})$. The 
reason is that $\u_h^{n-1}$ is discrete divergence-free with respect to $\Omega(t_{n-1})$, but not with respect to $\Omega(t_n)$
\begin{align*}
 ({\rm div }\, \u_h^{n-1}, \phi_h)_{\Omega(t_{n-1})} = 0, \quad \text{but } ({\rm div }\, \u_h^{n-1}, \phi_h)_{\Omega(t_n)} \neq 0
\end{align*}
for certain $\phi_h \in {\cal V}_h$.

For space discretisation, we will use the Cut Finite Element framework~\cite{HansboHansbo2002}. The idea is to discretise a larger domain of simple structure in the spirit of the 
Fictitious Domain approach. The active degrees of freedom consist of all degrees of freedom in mesh elements with non-empty intersection with $\Omega_\delta(t_n)$. 
Dirichlet boundary conditions are incorporated by means of Nitsche's method~\cite{Nitsche70}.

We will consider both the BDF(1)/backward Euler and the BDF(2) variant of the approach. {\sfrei To simplify the presentation of the analysis, we will neglect quadrature errors related to the approximation of curved boundaries and, moreover, focus on the BDF(1) variant. The necessary modifications for the BDF(2) variant will be sketched within remarks.} Finally, we will use a duality technique 
to prove an optimal $L^2(L^2)$-norm estimate for the velocities.

The structure of this article is as follows: In Section~\ref{sec.equations} we introduce the equations and sketch how to prove 
the well-posedness of the system. Then we introduce time and space 
discretisation in Section~\ref{sec.disc}, including the extension operators and assumptions, that will be needed in the stability analysis of Section~\ref{sec.stab} and the error 
analysis in Section~\ref{sec.error}.
Then, we give some three-dimensional numerical results in Section~\ref{sec.num}. We conclude 
in Section~\ref{sec.concl}.

\section{Equations}
\label{sec.equations}

We consider the non-stationary Stokes equations {\sfrei with homogeneous Dirichlet boundary conditions} on a moving domain $\Omega(t)\subset \mathbb{R}^d, d=2,3$ for $t\in I=[0,t_{\text{fin}}]$ 
{\sfrei
\begin{eqnarray}
 \begin{aligned}\label{StokesNormalised}
\partial_ t \u -  \Delta \u + \nabla p &= \f,\qquad & {\rm div } \, \u &=0\qquad \text{ in } \Omega(t),\\
 \u=0 \quad \text{ on } &\partial\Omega(t), \qquad &
\u(x, 0)&= \u^0(x) \text{ in } \Omega(0). 
\end{aligned}
\end{eqnarray}}
We assume that the domain motion can be described by a $W^{1,\infty}$-diffeomorphism 
\begin{align}\label{map}
 \T(t): \Omega(0) \to \Omega(t).
\end{align}
with the additional regularity
\begin{align}\label{mapreg}
\T\in W^{1,\infty}(I,{\sfrei W^{2,\infty}(\Omega(0))}).
\end{align}

and that the initial domain $\Omega(0)$ is piecewise smooth and Lipschitz.
In order to formulate the variational formulation we define the spaces
\begin{eqnarray}
\begin{aligned}\label{spaces}
 {\cal V}(t) &:= H^1_0(\Omega(t))^d, \qquad\; {\cal L}(t) := L^2(\Omega(t)), \qquad\; {\cal L}_0(t) := L^2_0(\Omega(t)), \\
 {\cal V}_I &:= \{ \u\in L^2(I, {\cal V}(t))^d, \; \partial_t \u \in L^2(I, {\sfrei {\cal L}(t)}^d)\}, \quad {\cal L}_{0,I} := L^2(I, {\cal L}_0(t)).
 \end{aligned}
 \end{eqnarray}
 and 
 consider the variational formulation: \textit{Find } $\u\in {\cal V}_I, p \in {\cal L}_{0,I}$ such that
 \begin{eqnarray}
\begin{aligned}\label{Stokes}
 (\partial_t \u, \v)_{\Omega(t)} + {\cal A}_S(\u,p;\v,q) &= (\f, \v)_{\Omega(t)} \quad \forall \v \in {\cal V}(t), \; q \in {\cal L}(t) \quad \text{ a.e. in } t \in I,\\
 \u(x,0) &= \u^0(x) \qquad \text{a.e. in } \Omega(0),
\end{aligned}
\end{eqnarray}
where 
\begin{align}
  {\cal A}_S(\u,p;\v,q):=   (\nabla \u, \nabla \v)_{\Omega(t)} - (p,\div \,\v)_{\Omega(t)} 
+ (\div \, \u, q)_{\Omega(t)}.
\end{align}
We assume that {\sfrei $\f \in L^{\infty}(I,{\cal L}(t)^d)$} a.e.$\,$in $t\in I$ and {\sfrei $\u^0 \in H^1(\Omega(0))^d$}.

{\sfrei 
\begin{remark}{(Boundary conditions)}
It might seem unnatural at first sight to use homogeneous Dirichlet boundary conditions for a Stokes problem on a moving domain $\Omega(t)$. In fact the assumption that the flow follows the domain motion on $\partial\Omega(t)$ would be a more realistic boundary condition, i.e.
\begin{eqnarray}
\begin{aligned}\label{StokesNonhom}
\partial_ t \tilde{\u} -  \Delta \tilde{\u} + \nabla p &= \tilde{\f},\qquad
 {\rm div } \, \tilde{\u} =0 \qquad \text{in } \Omega(t)\\
 \tilde{\u}&=\partial_t \T{\sfrei (\T^{-1})} \text{ on } \partial\Omega(t).
 \end{aligned}
 \end{eqnarray}
 Note, however, that for a sufficiently smooth map $\T(t)$ that fulfils ${\rm div} (\partial_t \T{\sfrei (\T^{-1})}) =0$, one obtains \eqref{StokesNormalised} from~\eqref{StokesNonhom} for $\u=\tilde{\u}-\partial_t \T{\sfrei (\T^{-1})}$ and
$\f:=\tilde{\f} +\partial_t^2 \T{\sfrei (\T^{-1})} - \Delta (\partial_t \T{\sfrei (\T^{-1})})$. For this reason and in order to simplify the presentation of the error analysis, we will consider homogeneous Dirichlet conditions in the remainder of this article.
\end{remark}}

\subsection{Well-posedness}
\label{sec.wellp}

As the spaces in \eqref{spaces} are lacking a tensor product structure, {\sfrei the proof of well-posedness of~\eqref{Stokes} is more complicated than on a fixed domain. In the case of a fixed domain existence and uniqueness of solutions can be shown under weaker regularity assumptions on the data $f$ and $u^0$ and the domain $\Omega(t)$ in the velocity space
\begin{align*}
\tilde{\cal V}_I &:= \{ \u\in L^2(I, {\cal V})^d, \; \partial_t \u \in L^2(I, ({\cal V}^*)^d)\},
\end{align*}
where ${\cal V}^*$ is the dual space to ${\cal V}$. Low regularity is, however, not of interest for the present paper, as we will require additional regularity of the solution in the error estimates. On the other hand, working with the space ${\cal V}_I$ under the additional regularity assumptions made above simplifies the proof of well-posedness of~\eqref{Stokes} significantly. 

The well-posedness of the Navier-Stokes problem on time-dependent domains, including an additional nonlinear convective term has in fact been the subject of a number of papers in literature~\cite{Bock1977, Salvi1988}. In order to deal with the additional non-linearity, additional assumptions on the regularity of the domains are typically made. For completeness, we give a proof of the following Lemma~\eqref{Stokes} in the appendix under the regularity assumptions made above. }

{\sfrei
\begin{lemma}\label{lem.wellposed}
Let $\Omega(0)$ be piecewise smooth and Lipschitz and $\T(t)$ a $W^{1,\infty}(\Omega(0))$ diffeomorphism with regularity $\T\in W^{1,\infty}(I,W^{2,\infty}(\Omega(0))$. For $\f \in L^{\infty}(I,{\cal L}(t)^d)$ and $\u^0 \in H^1(\Omega(0))^d$, Problem~\eqref{Stokes} has a unique solution $\u\in {\cal V}_I, p \in {\cal L}_{0,I}$.
\end{lemma}
\begin{proof}
A proof is given in the appendix.
\end{proof}}

\section{Discretisation}
\label{sec.disc}

For discretisation in time, we split the time interval of interest $I=[0,t_{\text{fin}}]$ into time intervals $I_n=(t_{n-1}, t_n]$ of 
uniform step size $\Delta t=t_n - t_{n-1}$
\begin{align*}
 0=t_0 < t_1 < ...< t_N = t_{\text{fin}}.
\end{align*}
We follow the work of Lehrenfeld \& Olshanskii~\cite{LehrenfeldOlshanskii} for parabolic problems on moving domains and 
use BDF(s) discretisation for $s=1,2$, where $s=1$ corresponds to a backward Euler time discretisation. 
Higher-order BDF formulae are not considered here, due to their lack of $A$-stability~\cite{HairerNorsettWanner1991}.
Following Lehrenfeld \& Olshanskii~\cite{LehrenfeldOlshanskii} we extend the domain $\Omega^n:= \Omega(t_n)$ in each time point $t_n$ 
by a strip of size $\delta$
to a domain $\Omega_\delta^n$, which is chosen large enough such that
\begin{align}\label{Omegani}
 \bigcup_{i=0}^s \Omega^{n+i} \subset \Omega_{\delta}^n,
\end{align}
{\sfrei see also the left part of Figure~\ref{fig.disc}.}
{\sfrei In particular, we will allow
\begin{equation}\label{defDelta}
s w_{\rm max} \Delta t \,\leq\, \delta \leq c_{\delta} s w_{\rm max} \Delta t,
\end{equation}
where 
\begin{align*}
 w_{\rm max} := \max_{t\in I, x \in \partial\Omega(0)} \|\partial_t T(x,t)\cdot \n\|
\end{align*}
is the maximum velocity of the boundary movement in normal direction in the Euclidean norm $\|\cdot\|$ and $c_\delta>1$ is a constant. 
If we assume that the domain map $\T$ lies in $W^{1,\infty}(I, L^{\infty}(\Omega(0)))$ (see Assumption~\ref{ass.T} below), the lower bound on $\delta$ in \eqref{defDelta} guarantees \eqref{Omegani}.}

{\sfrei 
The space-time slabs defined by the time discretisation and the space-time domain are denoted by
\begin{align*}
Q^n := \underset{t\in I_n}{\cup} \{t\}\times \Omega(t), \qquad Q_\delta^n := \underset{t\in I_n}{\cup} \{t\}\times \Omega_\delta(t), \qquad Q := \underset{t\in I}{\cup} \{t\}\times \Omega(t) . 
\end{align*}}
{\sfrei In what follows we denote by $c$ generic positive constants. These are in particular independent of space and time discretisation ($\Delta t, N$ and $h$) and of domain velocity $w_{\max}$ and $\delta$, unless such a dependence is explicitly mentioned.}

\begin{figure}
\resizebox*{0.3\textwidth}{!}{
\begin{picture}(0,0)%
\includegraphics{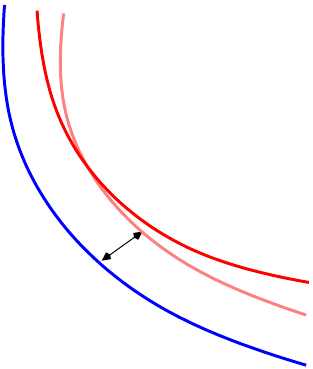}%
\end{picture}%
\setlength{\unitlength}{1243sp}%
\begingroup\makeatletter\ifx\SetFigFont\undefined%
\gdef\SetFigFont#1#2{%
  \fontsize{#1}{#2pt}%
  \selectfont}%
\fi\endgroup%
\begin{picture}(4749,5578)(2541,-7350)
\put(3826,-2716){\makebox(0,0)[lb]{\smash{{\SetFigFont{5}{6.0}{\color[rgb]{1,.5,.5}$\Omega^{n}$}%
}}}}
\put(4006,-5416){\makebox(0,0)[lb]{\smash{{\SetFigFont{5}{6.0}{\color[rgb]{0,0,0}$\delta$}%
}}}}
\put(5086,-6316){\makebox(0,0)[lb]{\smash{{\SetFigFont{5}{6.0}{\color[rgb]{0,0,1}$\Omega_{\delta}^n$}%
}}}}
\put(6256,-5416){\makebox(0,0)[lb]{\smash{{\SetFigFont{5}{6.0}{\color[rgb]{1,0,0}$\Omega^{n+1}$}%
}}}}
\end{picture}%
}\hspace{0.5cm}
\resizebox*{0.5\textwidth}{!}{
\begin{picture}(0,0)%
\includegraphics{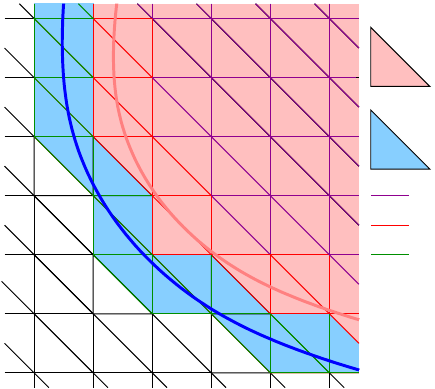}%
\end{picture}%
\setlength{\unitlength}{1243sp}%
\begingroup\makeatletter\ifx\SetFigFont\undefined%
\gdef\SetFigFont#1#2{%
  \fontsize{#1}{#2pt}%
  \selectfont}%
\fi\endgroup%
\begin{picture}(6569,5916)(1284,-7508)
\put(7651,-5056){\makebox(0,0)[lb]{\smash{{\SetFigFont{5}{6.0}{\color[rgb]{1,0,0}$e\in {\cal F}_h^{n,cut}$}%
}}}}
\put(7651,-5506){\makebox(0,0)[lb]{\smash{{\SetFigFont{5}{6.0}{\color[rgb]{0,.56,0}$e\in {\cal F}_{h,\delta}^{n,ext}$}%
}}}}
\put(7651,-4606){\makebox(0,0)[lb]{\smash{{\SetFigFont{5}{6.0}{\color[rgb]{.56,0,.56}$e\in {\cal F}_h^{n,int}$}%
}}}}
\put(7561,-2496){\makebox(0,0)[lb]{\smash{{\SetFigFont{5}{6.0}{\color[rgb]{0,0,0}$K\in {\cal T}_{h,\delta}^n$}%
}}}}
\put(7561,-2086){\makebox(0,0)[lb]{\smash{{\SetFigFont{5}{6.0}{\color[rgb]{0,0,0}$K\in{\cal T}_h^n$ and}%
}}}}
\put(7561,-3571){\makebox(0,0)[lb]{\smash{{\SetFigFont{5}{6.0}{\color[rgb]{0,0,0}$K\in {\cal T}_{h,\delta}^n$}%
}}}}
\put(6841,-6541){\makebox(0,0)[lb]{\smash{{\SetFigFont{5}{6.0}{\color[rgb]{1,.5,.5}$\partial\Omega^n$}%
}}}}
\put(6841,-7261){\makebox(0,0)[lb]{\smash{{\SetFigFont{5}{6.0}{\color[rgb]{0,0,1}$\partial\Omega_{\delta}^n$}%
}}}}
\end{picture}%
}
\caption{\textit{Left:} Illustration of $\Omega_\delta^n$ for $s=1$. \textit{Right:} Illustration of the discretisation and faces.\label{fig.disc}}
\end{figure}
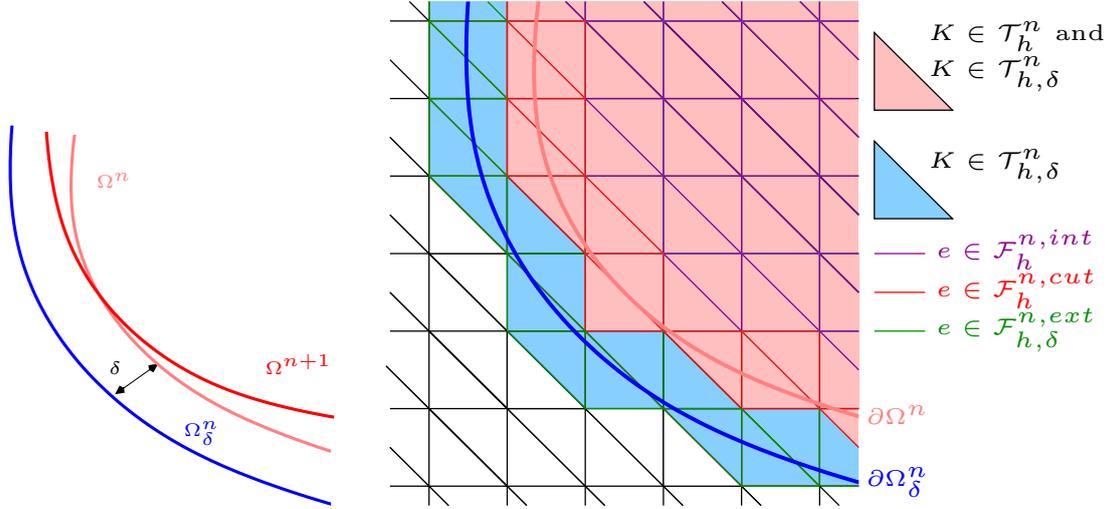

\subsection{Space discretisation}

Let ${\cal T}_{h,\delta}^n$ be a family of (possibly unfitted) quasi-uniform spatial discretisations of $\Omega_{\delta}^n$ into simplices with maximum cell size $h$.
We assume that ${\cal T}_{h,\delta}^n$ is based on a common background triangulation ${\cal T}_h$ for all $n$ and may differ 
only in the elements outside $\Omega(t_n)$ that are not present in ${\cal T}_{h,\delta}^k$ for $k\neq n$.
Further, we assume that ${\cal T}_{h,\delta}^n$ consists only of elements {\sfrei $K$} with non-empty intersection with 
$\Omega_{\delta}^n$, i.e.$\,{\sfrei K}\cap \Omega_\delta^n\neq \emptyset$.
The subset of cells with non-empty intersection with $\Omega$ is denoted by ${\cal T}_h^n$. {\sfrei An illustration is given in Figure~\ref{fig.disc}}.
By $\Omega_{h,\delta}^n$ we denote the domain spanned by all cells $K \in {\cal T}_{h,\delta}^n$ and by $\Omega_h^n$ the domain 
spanned by all cells $K \in {\cal T}_{h}^n$. 

Further, let ${\cal F}_{h,\delta}^n$ denote the set of interior faces $e$
of ${\cal T}_{h,\delta}^n$. { We split the faces into three parts: By ${\cal F}_h^{n,int}$, we denote the faces that belong exclusively to elements $K\in {\cal T}_{h,\delta}^n$ that lie in the interior 
of $\Omega^n$. By ${\cal F}_{h}^{n,cut}$ we denote the set of faces that belong to some element $K\in {\cal T}_{h,\delta}^n$
with $K\cap \partial\Omega^n\neq\emptyset$ and by
${\cal F}_{h,\delta}^{n,ext}$ the set of the remaining faces in ${\cal F}_{h,\delta}^n$, see Figure~\ref{fig.disc}.
Finally, we write 
${\cal F}_{h,\delta}^{n,g}$ for the union of ${\cal F}_{h}^{n,cut}$ and ${\cal F}_{h,\delta}^{n,ext}$, which will be used to define the ghost penalty extensions.
}
%


For spatial discretisation, we use continuous equal-order finite elements of degree $m\geq 1$ for all variables
\begin{align*}
 {\cal V}_h^n &:= \{\v \in C(\Omega_{h,\delta}^n), \, {\sfrei \v|_K \in P_m(K)} \, \forall {\sfrei K} \in {\cal T}_{h,\delta}^n \} \\
 {\cal L}_h^n &:= \{q \in C(\Omega_h^n), \, {\sfrei q|_K \in P_m(K)} \, \forall {\sfrei K} \in {\cal T}_h^n \}, \quad {\cal L}_{h,0}^n := {\cal L}_{h}^n \cap {\cal L}_0(t_n).
\end{align*}
Note that for the pressure space ${\cal L}_h^n$ an extension beyond $\Omega_h^n$ is not required. 

To deal with the \textit{inf-sup} stability, we will 
add a pressure stabilisation term $s_h^n$ to the variational formulation. {\sfrei In order to simplify the presentation, we will restrict ourselves to the Continuous Interior Penalty method (CIP~\cite{BurmanHansbo2006}) in this work, although different pressure stabilisations are possible. We define the CIP pressure stabilisation as}
\begin{align*}
  s_h^n(p;q) := 
 \sum_{e\in {\cal F}_h^{n,int}} h^3 (\jump{\partial_n p}, \jump{\partial_n q})_e
 { + \sum_{e\in {\cal F}_h^{n,cut}} \sum_{k=1}^{m} h^{2k+1} (\jump{\partial_n^k p}, \jump{\partial_n^k q})_e.}
\end{align*}
The higher derivatives in the boundary elements are necessary to control the derivatives $\nabla p_h$ on the extended computational domain $\Omega_h^n\setminus\Omega^n$ 
in the spirit 
of the ghost penalty stabilisation~\cite{Burman2010}.

We summarise the properties of the pressure stabilisation that will be needed in the following: There exists
an operator 
$C_h^n: {\cal V}_h^n \cup L^2(\Omega^n) \to {\cal V}_h^n$, such that the following properties are fulfilled for $n=1,\dots,N$
\begin{eqnarray}
\begin{aligned}\label{stabass}
  s_h^n(q,r) \leq s_h^n(q,q)^{1/2} s_h^n(r,r)^{1/2}&, \quad s_h^n(q,q) \leq ch^2 \|q\|_{H^1(\Omega^n)}^2
  \quad &&\forall q,r \in {\cal V}_h^n \cup H^2(\Omega^n)\\
  h^2\|\nabla p_h - C_h^n \nabla p_h\|_{\Omega_h^n}^2 &\leq s_h^n(p_h^n,p_h^n)\quad &&\forall p_h \in {\cal V}_h^n,\\ 
  { \|C_h^n \nabla p_h\|_{\Omega_h^n}^2} &\leq  { c \left( \|\nabla p_h\|_{\Omega^n}^2 + h^{-2} s_h^n(p_h^n, p_h^n)\right)} \quad &&\forall p_h \in {\cal V}_h^n.
\end{aligned}
\end{eqnarray}
A suitable projector $C_h^n$ for the CIP stabilisation is given by the the Oswald or Cl\'ement interpolation~\cite{BurmanHansbo2006}. {\sfrei For $m\geq 2$, we have additionally the consistency property
\begin{align}\label{stabass2}
 s_h^n(p,p) &=0 \qquad \forall p\in H^{m}(\Omega^n).
 \end{align}}

{\sfrei 
\begin{remark}{(Pressure stabilisation)}
In general any pressure stabilisation operator that leads to a well-posed discrete problem and that fulfils the assumptions \eqref{stabass}-\eqref{stabass2} can be used. The consistency condition $s_h^n(p,p)=0$ can be relaxed to a weak consistency of order $m_s>0$
\begin{eqnarray*}
  s_h^n(p,p) \leq ch^{2m_s} \|p\|_{H^{m_s}(\Omega^n)}^2.
\end{eqnarray*}
which will limit the spatial convergence order in the error estimates. 
One possibility is the Brezzi-Pitk\"aranta stabilisation~\cite{BrezziPitkaeranta} with order $m_s=1$. 
We refer to Burman \& Fern\'andez for a review of further possibilities for pressure stabilisation~\cite{BurmanFernandez2008}.
\end{remark}}

\subsection{Variational formulation}

To cope with the evolving geometry from one time-step to another, we extend the velocity variable $\u_h^n$
to $\Omega_{\delta}^n$, which will be needed in the following time-step, by using so-called \textit{ghost penalty terms} $g_h^n$. We will describe different possibilities to define $g_h^n$ in the next subsection.
For $k=s,\dots,n$ we define the following time-stepping scheme:
\textit{Find $\u_h^n \in {\cal V}_h^n, p_h^n \in {\cal L}_{h,0}^n$} such that
\begin{eqnarray}\label{DiscSystem}
 \begin{aligned}
  (D_t^{(s)} \u_h^n, \v_h)_{\Omega^n} &+{\cal A}_h^n(\u_h^n,p_h^n;\v_h,q_h)
     = (\f,\v_h)_{\Omega^n} \quad \forall \v_h \in {\cal V}_h^n, q_h \in {\cal L}_h^n, 
  \end{aligned}
\end{eqnarray}
where $D_t^{(s)}$ is an approximation of the time derivative $\partial_t$ by the BDF(s) backward difference formula, i.e.
\begin{align*}
 D_t^{(1)} \u_h^n &:= \frac{1}{\Delta t} (\u_h^n - \u_h^{n-1}), \\
 D_t^{(2)} \u_h^n &:=  D_t^{(1)} \u_h^n + \frac{1}{2} (D_t^{(1)} \u_h^n - D_t^{(1)} \u_h^{n-1}) = \frac{1}{2\Delta t} (3\u_h^n - 4\u_h^{n-1} + \u_h^{n-2}).
 \end{align*}

The bilinear form ${\cal A}_h$ is defined by
\begin{align}\label{defAh}
 {\cal A}_h^n(\u_h^n,p_h^n;\v_h,q_h) :=  {\cal A}_S^n(\u_h^n,p_h^n;\v_h,q_h) + a_D^n(\u_h^n, p_h^n; \v_h, q_h)
 + \gamma_g g_h^n(\u_h^n, \v_h) + \gamma_p s_h^n(p_h^n, q_h).
 \end{align}
It includes the Stokes part
 \begin{align}\label{AStokes}
  {\cal A}_S^n(\u,p;\v,q):=   (\nabla \u, \nabla \v)_{\Omega^n} - (p,\div \,\v)_{\Omega^n} 
+ (\div \, \u, q)_{\Omega^n}
\end{align}
and Nitsche terms to weakly impose the Dirichlet boundary conditions
\begin{align}\label{Nitsche}
a_D^n(\u_h^n, p_h^n; \v_h, q_h) := - (\partial_n \u_h^n -p_h^n \n, \v_h)_{\partial\Omega^n} - (\u_h^n, \partial_n \v_h +q_h \n)_{\partial\Omega^n}
 + \frac{\gamma_D}{h} \left(\u_h^n, \v_h\right)_{\partial\Omega^n}.
 \end{align}
{\sfrei In \eqref{Nitsche} the last term can be seen as a penalty term to weakly impose the homogeneous Dirichlet condition for the velocities. The first term 
on the right-hand side makes the variational formulation consistent (in space). Finally, the second term, which vanishes for $u_h^n=0$, yields a formulation, which is
}  
symmetric for the velocities, but skew-symmetric for the pressure.
{\sfrei The skew-symmetry in the pressure variable leads to a stable variational formulation, as the pressure terms cancel out by diagonal testing ($\v_h^n=\u_h^n, q_h^n=p_h^n$), see for example \cite{BurmanFernandezFSI}.} 
  The parameters $\gamma_D, \gamma_p$ and $\gamma_g$ are positive constants.
  
To include the initial condition, we set $\u_h^0 := \pi_h^1 E \u^0$, where $\pi_h^n$ denotes the $L^2(\Omega^n)$-projection onto ${\cal T}_{h}^n$ and $E$ denotes an $L^2$-stable extension operator, which is introduced in the next section.
Summing over $k=1,\dots,n$ in time, the complete system reads for $s=1$
 \begin{eqnarray}\label{completeSpaceTime}
 \begin{aligned}
 \sum_{k=1}^n \Big\{\frac{1}{\Delta t}
    &\left(\u_h^k-\u_h^{k-1}, \v_h^k\right)_{\Omega^k} + {\cal A}_h^k(\u_h^k,p_h^k;\v_h^k,q_h^k) 
   \Big\}
 + \frac{1}{\Delta t}\left(\u_h^0, \v_h^1\right)_{\Omega^1} \\
 &\qquad\qquad=  \frac{1}{\Delta t}\left(E\u^0, \v_h^1\right)_{\Omega^1} + \sum_{k=1}^n (\f,\v_h^k)_{\Omega^k} \quad \forall \v_h^k \in {\cal V}_h^k, q_h^k \in {\cal L}_h^k, \quad k=1,\dots n. 
  \end{aligned}
\end{eqnarray}


{\sfrei In order to simplify the presentation, we will assume that the integrals in \eqref{completeSpaceTime}
are evaluated exactly. While for complex domains $\Omega^k$ an exact evaluation might be difficult in practise, it is in principle possible to push the quadrature error below every prescribed tolerance, for example by adding sufficient quadrature points within a summed quadrature formula. If the integrals are only roughly approximated, for example due to a discrete level-set function $\phi_h^n\in{\cal V}_h^n$ which is only an approximation of a continuous function $\phi^n$, an additional quadrature error needs to be considered.
We refer to the work of Lehrenfeld \& Olshanskii\cite{LehrenfeldOlshanskii}, where these additional error contributions have been analysed in detail for parabolic problems on moving domains. An advantage of the CutFEM methodology compared to standard finite elements is that besides the quadrature error no additional discretisation errors related to the approximation of curved boundaries within the finite element spaces need to be considered.}

{ To initialise the BDF(2) scheme the value $\u_h^1$ needs to be computed with sufficient accuracy before the first full BDF(2) 
step can be made. 
We will comment on the specific requirements and on different possibilities below
in Remark~\ref{rem.initBDF2}. }

\subsection{Extension operators}
\label{sec.ext}

Due to the evolution of the domain, we will frequently need to extend variables defined on smaller domains to larger ones. 
Therefore, we will use $W^{k,p}$-stable extension operators {\sfrei $E^n: \Omega^n \to \Omega_\delta^n$ to extend functions $\u(t_n)\in W^{k,p}(\Omega^n)$}. We make the following assumption for the regularity of the domains $\Omega(t)$ and the domain movement, depending on the polynomial degree $m$ of the finite element spaces.

{\sfrei 
\begin{assumption}\label{ass.T}
We assume that the boundary of the initial domain $\Omega(0)$ is piecewise smooth and Lipschitz, and that the domain motion $\T(t)$ is a $W^{1,\infty}$-diffeomorphism for each $t$ and smooth in the sense that $\T \in L^{\infty}(I, W^{m+1,\infty}(\Omega(0))) \cap W^{1,\infty}(I,W^{m,\infty}(\Omega(0)))$.
\end{assumption}
}

{\sfrei If Assumption~\ref{ass.T} is fulfilled for $m\in\mathbb{N}$, suitable extension operators $E^n:  \Omega^n \to \Omega_\delta^n$ exist with the properties}
\begin{align}
\|E^n \u - \u\|_{\sfrei W^{m+1,p}(\Omega)} &=0,\qquad \|E^n \u\|_{\sfrei W^{m+1,p}(\Omega_\delta^n)} \leq c \|\u\|_{\sfrei W^{m+1,p}(\Omega^n)}, \label{extension}\\
 \|\partial_t (E^n \u)\|_{H^m(\Omega_\delta^n)} &\leq c \left(\|\u\|_{H^{m+1}(\Omega^n)} + \|\partial_t \u\|_{H^m(\Omega^n)}\right) \label{Extdt},\\
 {\sfrei \|\partial_t^2 (E^n \u)\|_{Q_\delta^n}} &{\sfrei \leq c \|\u\|_{H^2(Q^n)}.}\label{Extdt2}
\end{align}
{\sfrei
For a proof of \eqref{extension} we refer to Stein~\cite{Stein70}, Theorem 6 in Chapter VI. The estimate \eqref{Extdt} has been shown in~\cite{LehrenfeldOlshanskii}, Lemma 3.3. The estimate \eqref{Extdt2} follows by the same argumentation.}
In order to alleviate the notation we will in the following skip the operator $E^n$ frequently and denote the extension also by $\u(t_n)$.

\subsubsection{Ghost penalty extension}

The discrete quantities are extended implicitly by adding so-called ghost penalty terms to the variational formulation.
We will consider three variants for the ghost penalty stabilisation,
and refer to~\cite{BurmanHansbo2013,GuerkanMassing2019}
for a more abstract approach on how to design suitable ghost penalties
for a PDE problem at hand.
The first ``classical'' variant~\cite{Burman2010,BurmanHansbo2013, MassingLarsonLoggEtAl2013a} is to penalise jumps of derivatives over element edges 
\begin{align*}
 g_h^{n, \rm jump}(\u, \v) :=  \sum_{e\in {\cal F}_{h,\delta}^{n,g}} \sum_{k=1}^{m}  h^{2k-1} (\jump{\partial_n^k \u}, \jump{\partial_n^k \v})_e.
\end{align*}
This variant has the advantage that it is fully consistent, {\sfrei i.e.\,it vanishes 
for $\u\in H^{m+1}(\Omega_\delta^n)^d$, which implies $\jump{\partial_n^k \u}|_e = 0$ for $k\leq m$.} A disadvantage is that higher derivatives need to be computed for polynomial degrees $m>1$.

To define two further variants, let us introduce the notation {\sfrei $K_{e,1}$ and $K_{e,2}$} for the two cells surrounding a face $e\in {\cal F}_{h,\delta}^{n,g}$, such that 
\begin{align*}
 e=\overline{K}_{e,1} \cap \overline{K}_{e,2}.
\end{align*}
We denote the union of both cells by $w_e:= K_{e,1}\cup K_{e,2}$ and use the $L^2$-projection $\pi_{w_e}: L^2(\Omega_\delta^n) \to P_m(w_e)$, which is
defined by
\begin{align*}
 (\u - \pi_{w,e} \u, \v)_{w_e} = 0 \quad \forall \v \in P_m(w_e).
\end{align*}
We define the ``projection variant'' of the ghost penalty stabilisation~\cite{Burman2010}
\begin{align*}
 g_h^{n, \rm proj}(\u_h, \v_h) :=  \frac{1}{h^2} \sum_{e\in {\cal F}_{h,\delta}^{n,g}}  \left(\u_h-\pi_{w_e} \u_h, \v_h- \pi_{w_e} \v_h\right)_{w_e} 
 =  \frac{1}{h^2} \sum_{e\in {\cal F}_{h,\delta}^{n,g}}  \left(\u_h-\pi_{w_e} \u_h, \v_h\right)_{w_e}.
\end{align*}
The last equality is a direct consequence of the definition of the $L^2$-projection.

The third variant, which has first been used in~\cite{Preuss2018}, uses canonical extensions of polynomials to the neighbouring cell instead of the 
projection $\pi_{w_e} \u$. Let us therefore denote the polynomials that define a function $\u\in {\cal V}_h^n$ in a 
cell $K_{e,i}$ by $\u_{e,i} = \u|_{K_{e,i}}$. We use the same notation for the canonical extension to the neighbouring cell, such that 
$\u_{e,i} \in P_m(w_e)$. Using this notation, we define the so-called ``direct method'' of the ghost penalty stabilisation
\begin{align*}
  g_h^{n,\rm dir}(\u, \v) :=  \frac{1}{h^2} \sum_{e\in {\cal F}_{h,\delta}^{n,g}} \left(\u_{e,1}-\u_{e,2}, \v_{e,1}-\v_{e,2}\right)_{w_e}.
\end{align*}
 For the analysis, we extend the definition of the stabilisation to functions
$\u,\v\in L^2(\Omega_\delta^n)$. Here, we set $\u_{e,i} := \pi_{K_{e,i}} \u|_{K_{e,i}}$ for $i=1,2$, where $\pi_{K_{e,i}}$ denotes the $L^2$-projection to $P_m(K_{e,i})$ and extend this polynomial canonically to the neighbouring cell. {\sfrei In contrast to the classical variant, $g_h^{n, \rm proj}$ and $g_h^{n, \rm dir}$ are only weakly consistent, i.e.\,they fulfil the estimate 
\begin{align*}
 g_h^n(\u, \u) \leq ch^{2m} \|\u\|_{H^{m+1}(\Omega_\delta^n)}, \qquad\text{for } \u \in H^{m+1}(\Omega_\delta^n)^d.
\end{align*}}
We will summarise the properties of these stabilisation terms, that we will need below, in the following lemma. Therefore, we assume 
that from each cell {\sfrei $K\in {\cal T}_{h,\delta}^n$} with $K\cap \Omega^n=\emptyset${, there 
exists a path of cells $K_i, i=1,\dots, m$, such that two subsequent cells share one common face $e=\overline{K}_i\cap \overline{K}_{i+1}$, 
and the final element lies in the interior of $\Omega^n$, i.e.$\,K_m\subset\Omega^n$. In addition the path shall fulfil the following properties. 
Let {\sfrei ${\cal K}$} be the maximum number of cells
needed in the path among all cells $K\in {\cal T}_{h,\delta}^n$. We assume that
\begin{align}\label{assK}
 {\sfrei {\cal K}}\leq (1+ \delta/h) \leq 1 + \frac{c_{\delta} s w_{\rm max} \Delta t}{h},
\end{align}
{\sfrei where the second inequality follows from~\eqref{defDelta}}.
Moreover, we assume that the number of cases in which a specific interior element $K_m\subset\Omega^n$ is used as a final element 
among all the paths is bounded independently of $\Delta t$ and $h$.
These assumptions are reasonable, as one can choose for example the final elements by a projection of distance $\delta$ towards the interior. For a detailed justification, we refer to
Lehrenfeld \& Olshanskii~\cite{LehrenfeldOlshanskii}, Remark 5.2.

\begin{lemma}\label{lem.ghost}
For $\v_h \in {\cal V}_h^n$ and the three variants $g_h^n \in \{ g_h^{n,\rm jump}, g_h^{n,\rm proj}, g_h^{n,\rm dir}\}$ it holds that
\begin{align*}
 \|\v_h\|_{\Omega_\delta^n}^2 \leq c \|\v_h\|_{\Omega^n}^2 + {\cal K} h^2 g_h^n(\v_h, \v_h), \qquad
 \|\nabla \v_h\|_{\Omega_\delta^n}^2 \leq c \|\nabla \v_h\|_{\Omega^n}^2 + {\cal K} g_h^n(\v_h, \v_h)
\end{align*}
Further, it holds for $\u,\v \in H^{m+1}(\Omega_\delta^n)$ for $m\geq 1$ and $\v_h\in {\cal V}_h^n$ that
\begin{align}\label{ghbound}
 g_h^n(\u,\v) \leq g_h^n(\u,\u)^{1/2} g_h^n(\v,\v)^{1/2}, \,\,
 g_h^n(\u, \u) \leq ch^{2m} \|\u\|_{H^{m+1}(\Omega_\delta^n)}^2, \,\, g_h^n(\v_h, \v_h) \leq c \|\nabla \v_h\|_{\Omega_\delta^n}^2.
\end{align}
\end{lemma}
\begin{proof}
 The first four properties have been proven for the three possibilities introduced above by Lehrenfeld \& Olshanskii~\cite{LehrenfeldOlshanskii}. The last inequality in \eqref{ghbound}  follows similarly.
\end{proof}

\subsection{Properties of the bilinear form}

{\sfrei We start with a continuity result for the combined bilinear form including the Nitsche terms in the functional spaces.}
 
\begin{lemma}{\sfrei (Continuity in the functional spaces)}\label{lem.contCont}
For functions $\u,\v\in {\cal V}(t_n) \cap H^2(\Omega^n)^d$ and $p,q \in {\cal L}(t_n)\cap H^1(\Omega^n)$, we have
\begin{eqnarray*}
\begin{aligned}
 \big|({\cal A}_S^n+&a_D^n)(\u, p; \v, q)\big| + \big| ({\cal A}_S^n+a_D^n)(\v, q; \u, p)\big|  \\
 &\leq c \left(\|\nabla \u\|_{\Omega^n} +h^{-1/2} \|\u\|_{\partial\Omega^n} + h\|\nabla p\|_{\Omega^n} + h^{1/2}\|\partial_n \u\|_{\partial\Omega^n}  \right) \\
 &\qquad\cdot\left( \|\nabla \v\|_{\Omega^n} + h^{-1} \|\v\|_{\Omega^n}  +  \|q\|_{\Omega^n} + h^{-1/2} \|\v\|_{\partial\Omega^n} + h^{1/2} \left(\|\partial_n \v\|_{\partial\Omega^n} +\|q\|_{\partial\Omega^n} \right) \right). 
\end{aligned}
\end{eqnarray*}
\end{lemma}
\begin{proof}
We apply integration by parts in \eqref{AStokes}
 \begin{align*}
   {\cal A}_S^n(\u, p; \v, q) &= (\nabla \u, \nabla \v)_{\Omega^n} + (\nabla p, \v)_{\Omega^n} - (p, \v\cdot \n)_{\partial\Omega^n} + ({\rm div}\, \u, q)_{\Omega^n}\\
 &\leq c \left(\|\nabla \u\|_{\Omega^n} + h\|\nabla p\|_{\Omega^n}\right) \left(\|\nabla \v\|_{\Omega^n} + h^{-1}\|\v\|_{\Omega^n} + \|q\|_{\Omega^n}\right) - (p, \v\cdot \n)_{\partial\Omega^n}.
 \end{align*}
  For the Nitsche terms standard estimates result in
 \begin{eqnarray}
 \begin{aligned}
  a_D^n(\u, p; \v, q)
 &\leq c \left(h^{-1/2} \|\u\|_{\partial\Omega^n} + h^{1/2}\|\partial_n \u\|_{\partial\Omega^n} \right) 
 \left(h^{-1/2} \|\v\|_{\partial\Omega^n} + h^{1/2}\left(\|\partial_n \v\|_{\partial\Omega^n} + \|q\|_{\partial\Omega^n}\right) \right)\\
 &\qquad+ (p, \v\cdot \n)_{\partial\Omega^n}.
\end{aligned}
\end{eqnarray}
The estimate for $({\cal A}_S^n+a_D^n)(\v, q; \u, p)$ can be shown in exactly the same way by inverting the role of test and trial functions.

\end{proof}

Next, we show continuity and coercivity of the {\sfrei discrete bilinear form}. To this end, we introduce the triple norm
\begin{align*}
 |||\u_h|||_{h,n} := \left(  \|\nabla \u_h^n\|_{\Omega^n}^2 + \gamma_g g_h^n(\u_h^n, \u_h^n) + \frac{\gamma_D}{h} \|\u_h^n\|_{\partial\Omega^n}^2  \right)^{1/2}
\end{align*}
\begin{lemma}{\sfrei (Coercivity \& Continuity in the discrete setting)}\label{lem.CoercCont}
For the bilinear form ${\cal A}_h$ defined in \eqref{defAh} and $\u_h\in {\cal V}_h^n$ and $p_h\in {\cal L}_h^n$, it holds for 
$\gamma_D$ sufficiently large
\begin{align}\label{coerc}
{\cal A}_h^n(\u_h, p_h; \u_h, p_h) \geq \frac{1}{2}\left( |||\u_h|||_{h,n}^2 +  \gamma_p s_h^n(p_h, p_h)\right).
 \end{align}
 Moreover, we have for $\u_h,\v_h\in {\cal V}_h^n$ and $p_h, q_h\in {\cal L}_h^n$
 \begin{align}\label{contDisc}
  {\cal A}_h^n(\u_h, p_h; \v_h, q_h) \leq \left(|||\u_h|||_{h,n} + \|p_h\|_{\Omega^n} + s_h^n(p_h, p_h)^{1/2}\right) \left( |||\v_h|||_{h,n} + \|q_h\|_{\Omega^n} + s_h^n(q_h, q_h)^{1/2}\right).
 \end{align}
 \end{lemma}
\begin{proof}
{\sfrei To show coercivity \eqref{coerc}, we note that 
\begin{align*}
{\cal A}_h^n(\u_h, p_h; \u_h, p_h) = |||u_h|||_{h,n}^2 + \gamma_p s_h(p_h, p_h) - 2(\u_h^n,\partial_n \u_h^n)_{\partial\Omega^n}.
\end{align*}
To estimate the term $-2(\u_h^n,\partial_n \u_h^n)_{\partial\Omega^n}$, we apply a Cauchy-Schwarz and Young's inequality for $\epsilon>0$, followed by an inverse inequality on $\Omega_h^n$
\begin{align*}
-2(\u_h,\partial_n \u_h)_{\partial\Omega^n} 
\geq -\frac{1}{\epsilon h} \|\u_h\|_{\partial\Omega^n}^2 - \epsilon h \| \nabla \u_h\|_{\partial\Omega^n}^2
 \geq -\frac{1}{\epsilon h} \|\u_h\|_{\partial\Omega^n}^2 - c\epsilon \| \nabla \u_h\|_{\Omega_h^n}^2.
\end{align*}
Using Lemma~\ref{lem.ghost}, we obtain
\begin{align*}
2(\u_h,\partial_n \u_h)_{\partial\Omega^n} \geq -\frac{1}{\epsilon h} \|\u_h\|_{\partial\Omega^n}^2 - c\epsilon \left(\| \nabla \u_h\|_{\Omega^n}^2 + {\cal K} g_h^n(\u_h, \u_h)\right) \\
\geq -\frac{\gamma_D}{2 h} \|\u_h\|_{\partial\Omega^n}^2 - \frac{1}{2} \left(\| \nabla \u_h\|_{\Omega^n}^2 + \gamma_g g_h^n(\u_h, \u_h)\right)
\end{align*}
for $\gamma_D$ sufficiently large.}
Concerning continuity, we estimate
\begin{eqnarray}
\begin{aligned}\label{ASDisc}
 {\cal A}_S^n(\u_h, p_h; \v_h, q_h) 
 &\leq c \left(\|\nabla \u_h\|_{\Omega^n} + \|p_h\|_{\Omega^n}\right) \left(\|\nabla \v_h\|_{\Omega^n} + \|q_h\|_{\Omega^n}\right)
 \end{aligned}
 \end{eqnarray}
 For the Nitsche terms, we have using inverse inequalities and Lemma~\ref{lem.ghost}
 \begin{eqnarray}
 \begin{aligned}\label{Nitsche4p}
  a_D^n(\u_h, &p_h; \v_h, q_h) = \frac{\gamma_D}{h} (\u_h, \v_h)_{\partial\Omega^n} - ( \partial_n \u_h - p_h  n, \v_h)_{\partial\Omega^n} 
 -(\u_h, \partial_n \v_h + q_h \n)_{\partial\Omega^n}\\
 &\leq c \left(\frac{\gamma_D^{1/2}}{h^{1/2}} \|\u_h\|_{\partial\Omega^n} + \|\nabla \u_h\|_{\Omega_h^n}^2 +  \|p_h\|_{\Omega^n} \right) \left(\frac{\gamma_D^{1/2}}{h^{1/2}} \|\v_h\|_{\partial\Omega^n} + \|\nabla \v_h\|_{\Omega_h^n}^2 + \|q_h\|_{\Omega^n} \right)\\
  &\leq c \left(|||\u_h|||_{h,n} +  \|p_h\|_{\Omega^n} \right) \left(|||\v_h|||_{h,n} + \|q_h\|_{\Omega^n} \right).
\end{aligned}
\end{eqnarray}
 Finally, Lemma~\ref{lem.ghost} and the assumption \eqref{stabass} for the pressure stabilisation yield
 \begin{align*}
 g_h^n(\u_h, \v_h) &\leq g_h^n(\u_h, \u_h)^{1/2} g_h^n(\v_h, \v_h)^{1/2},\qquad
 s_h^n(p_h, q_h) \leq s_h^n(p_h, p_h)^{1/2}  s_h^n(q_h, q_h)^{1/2}.
\end{align*}
\end{proof}

Moreover, we have the following modified \textit{inf-sup} condition for the discrete spaces.
\begin{lemma}\label{lem.modinfsup}
 Let $p_h^n\in {\cal L}_h^n$. There exists a constant $\beta>0$, such that
 \begin{align}\label{infsup1}
  \beta \|p_h^n\|_{\Omega^n} \leq \sup_{\v_h^n \in {\cal V}_h^n} \frac{({\rm div}\, \v_h^n, p_h^n)_{\Omega^n} - (\v_h^n \cdot n, p_h^n)_{\partial\Omega^n}}{|||\v_h^n|||_{h,n}} + h\|\nabla p_h^n\|_{\Omega^n}.
 \end{align}
\end{lemma}
\begin{proof}
We follow Burman \& Hansbo~\cite{BurmanHansbo2006} and define $\v_p^n\in H^1_0(\Omega^n)^d$ as solution to
\begin{align}\label{vpn}
 {\rm div}\, \v_p^n =-\frac{p_h^n}{\|p_h^n\|_{\Omega^n}}\quad \text{on }\Omega^n.
\end{align}
Such a solution exists, see Temam~\cite{Temam2000}, and fulfils $\|\v_p^n\|_{H^1(\Omega^n)} \leq c$.
{\sfrei We introduce an $L^2$-stable interpolation $i_h^n \v_p^n$ (for example the Cl\'ement interpolation) to get}
\begin{align}\label{phn1}
 \|p_h^n\|_{\Omega^n} = -(p_h^n, {\rm div}\, \v_p^n)_{\Omega^n} = -(p_h^n, {\rm div}\, (\v_p^n-i_h^n \v_p^n))_{\Omega^n} - (p_h^n, {\rm div}\, (i_h^n \v_p^n))_{\Omega^n}.
\end{align}
We apply integration by parts in the first term 
and use that $\v_p^n$ vanishes in $\partial\Omega^n$
\begin{eqnarray}
\begin{aligned}\label{phn2}
 -(p_h^n, {\rm div} (\v_p^n-i_h^n \v_p^n))_{\Omega^n} &= (\nabla p_h^n, \v_p^n-i_h^n \v_p^n)_{\Omega^n} -  (p_h^n \n, \v_p^n-i_h^n \v_p^n)_{\partial\Omega^n}\\
 &\leq ch \|\nabla p_h^n\|_{\Omega^n}  + (p_h^n \n, i_h^n \v_p^n)_{\partial\Omega^n}
\end{aligned}
\end{eqnarray}
The statement follows by noting that
\begin{align*}
 |||i_h^n \v_p^n|||_{h,n}^2 &=  \|i_h^n \nabla \v_p^n\|_{\Omega^n}^2 + \frac{\gamma_D}{h} \|i_h^n \v_p^n\|_{\partial\Omega^n}^2 + g_h^n(i_h^n \v_p^n, i_h^n \v_p^n) \\
 &\leq c\left(  \|\nabla \v_p^n\|_{\Omega^n}^2 + \frac{\gamma_D}{h} \|i_h^n \v_p^n-\v_p^n\|_{\partial\Omega^n}^2 
 + \|\nabla i_h^n \v_p^n\|_{\Omega_{h,\delta}^n}^2\right)
 \leq c \|\nabla \v_p^n\|_{\Omega^n}^2 \leq c.
\end{align*}
\end{proof}


The well-posedness of the discrete system \eqref{DiscSystem} for sufficiently large $\gamma_p, \gamma_g, \gamma_D$ and given $\u_h^{n-1}$ (and $\u_h^{n-2}$ for BDF(2)) follows by standard arguments, see for example~\cite{BurmanHansbo2006}.

\section{Stability analysis}
\label{sec.stab}

In order to simplify the analysis, we restrict ourselves in this and the next section to the case $s=1$ first, i.e., the backward Euler variant of the time discretisation and comment on the case $s=2$ in remarks.
In order to abbreviate the notation, we write for the space-time Bochner norms
\begin{align*}
 \|\u\|_{\infty,m,I_k} := \|\u\|_{L^{\infty}(I_k,H^m(\Omega(t)))}, \qquad \|\u\|_{\infty,m} := \|\u\|_{\infty,m,I},
\end{align*}
where $m\in \mathbb{Z}$ and $H^0(\Omega(t)) := L^2(\Omega(t))$.

We start with a preliminary result concerning the extension of discrete functions to $\Omega_\delta^n$.

\begin{lemma}
 \label{lem.domainExt}
 Let $\v \in {\cal V}, \,\delta \leq c_\delta s w_{\rm max} \Delta t$ and $S_\delta^n := \Omega_\delta^n \setminus \Omega^n$. It holds for arbitrary $\epsilon>0$
 \begin{align}\label{deltaeps}
\|\v\|_{S_{\delta}^n}^2 \leq c \delta \left( (\epsilon+\epsilon^{-1}) \|\v\|_{\Omega_\delta^n}^2 + \epsilon \|\nabla \v\|_{\Omega_\delta^n}^2 \right). 
\end{align} 
For $\v_h \in {\cal V}_h^n$, we have further for $h$ sufficiently small
\begin{eqnarray}
 \begin{aligned}\label{secondStatement}
 \|\v_h\|_{\Omega_{\delta}^n}^2 &\leq \left(1 + {\sfrei c_1(w_{\max})} \Delta t \right)  \|\v_h\|_{\Omega^{n}}^2 + \frac{ \Delta t}{2} \|\nabla \v_h\|_{\Omega^{n}}^2
 + 
  {\sfrei c_2(w_{\max})} \Delta t {\cal K} g_h^n(\v_h, \v_h) 
 \end{aligned}
 \end{eqnarray}
 with constants ${\sfrei c_1(w_{\max})}:= 1/2 +c s^2 w_{\rm max}^2$, ${\sfrei c_2(w_{\max})} := c w_{\rm max}^2 h^2 +  1$ and $c>0$.
\end{lemma}
\begin{proof}
These results follow {\sfrei similarly to Lemmas 3.4 and 5.3} in~\cite{LehrenfeldOlshanskii}. 
Nevertheless, we give here a sketch of the proof due to the importance of the Lemma in 
the following estimates. We define 
\begin{align*}
 \Omega_r^n &:= \Omega^n \cup \{ x\in S_\delta^n, {\rm dist}(x,\partial\Omega^n) < r\}, \quad
  \Gamma_r^n := \{ x\in S_\delta^n, {\rm dist}(x,\partial\Omega^n) = r\} = \partial\Omega_r^n.
\end{align*}
We apply a {\sfrei multiplicative trace inequality
and Young's inequality} for arbitrary $\epsilon>0$
\begin{align}\label{multTrace}
 \| v\|_{\Gamma_r^n}^2 \leq c \|\v\|_{\Omega_r^n} \|\v\|_{H^1(\Omega_r^n)} 
 {\sfrei\leq c_0 \left( \epsilon^{-1}  \|\v\|_{\Omega_r^n}^2 + \epsilon \|\v\|_{H^1(\Omega_r^n)}^2\right)}
 = c_0 \left( (\epsilon+\epsilon^{-1}) \|\v\|_{\Omega_r^n}^2 + \epsilon \|\nabla \v\|_{\Omega_r^n}^2 \right)
\end{align}
with a constant $c_0$ depending on the curvature of $\partial\Omega^n$.
Integration over $r\in (0,\delta)$ yields \eqref{deltaeps}.
For a discrete function $\v_h\in {\cal V}_h^n$ we use Lemma~\ref{lem.ghost}
to obtain
\begin{align*}
 \|\v_h\|_{S_{\delta}^n}^2 &\leq c_0 \delta (\epsilon+\epsilon^{-1}) \|\v_h\|_{\Omega_\delta^n}^2 + c_0\delta \epsilon \|\nabla \v_h\|_{\Omega_\delta^n}^2\\
 &\leq c_0 \delta (\epsilon+\epsilon^{-1}) \|\v_h\|_{\Omega^n}^2 + c_0\delta \epsilon \|\nabla \v_h\|_{\Omega^n}^2
 + c_0\delta {\cal K} \left((\epsilon+\epsilon^{-1}) h^2 + \epsilon \right) g_h^n(\v_h, \v_h).
\end{align*}
Using {\sfrei \eqref{defDelta}} and choosing $\epsilon=\frac{1}{2c_0 c_\delta s w_{\rm max}}$, we have $c_0 \delta \epsilon \leq \frac{\Delta t}{2}$ and
\begin{align}\label{next2final}
 \|\v_h\|_{S_{\delta}^n}^2 &\leq {\sfrei c_1(w_{\max})}\Delta t \|\v_h\|_{\Omega^n}^2 + \frac{\Delta t}{2} \|\nabla \v_h\|_{\Omega^n}^2
 +{\sfrei c_2(w_{\max})} \Delta t {\cal K} g_h^n(\v_h, \v_h)
\end{align}
for $h<1$ with the constants {\sfrei $c_1(w_{\max}), c_2(w_{\max})$} given in the statement.
The inequality \eqref{secondStatement} follows by combining \eqref{next2final} with the equality 
\begin{align}\label{firstLine}
 \|\v_h\|_{\Omega_{\delta}^n}^2 = \|\v_h\|_{\Omega^n}^2 + \|\v_h\|_{S_{\delta}^n}^2.
\end{align}
\end{proof}

Now we are ready to show a stability result for the discrete formulation \eqref{DiscSystem}. 

\begin{theorem}\label{theo.stab}
 Let $\u_h = (\u_h^k)_{k=1}^N, p_h = (p_h^k)_{k=1}^N$ be the solution of \eqref{DiscSystem} for $s=1$,
$\gamma_g \geq {\sfrei c_2(w_{\max})} {\cal K}$,
where ${\sfrei c_2(w_{\max})}$ denotes the constant from Lemma~\ref{lem.domainExt} and $\gamma_D$ sufficiently large.
 Under the regularity assumptions stated above, it holds for $n\geq 1$ that
 \begin{eqnarray}
 \begin{aligned}\label{StabBDF1}
  \|\u_h^n\|_{\Omega^n}^2  + \sum_{k=1}^n \|\u_h^k-\u_h^{k-1}\|_{\Omega^k}^2 
  + \Delta t  &\left( |||\u_h^k|||_{h,k}^2 + \gamma_p s_h^k(p_h^k,p_h^k) \right) 
  \\
  &\hspace{1.5cm}\leq c\exp( {\sfrei c_1(w_{\max})} t_n) \left(\|\u^0\|_{\Omega^0}^2  +  t_n\|\f\|_{\infty,0}^2 \right), 
 \end{aligned}
 \end{eqnarray}
 with {\sfrei ${\sfrei c_1(w_{\max})}$ given in Lemma~\ref{lem.domainExt}} and $\u_h^0 := \pi_h^1 E \u^0$.
\end{theorem}
\begin{proof}
 Testing \eqref{DiscSystem} with $\v_h = 2\Delta t \u_h^n, q_h = 2\Delta t p_h^n$, using the coercivity \eqref{coerc} and the relation 
 \begin{align}\label{telescope}
{\sfrei -2(\u_h^n, \u_h^{n-1})_{\Omega^n} =  \|\u_h^n - \u_h^{n-1}\|_{\Omega^n}^2 -\|\u_h^n\|_{\Omega^n}^2 - \|\u_h^{n-1}\|_{\Omega^n}^2}
 \end{align}
yields for $n>1$
 \begin{eqnarray}\label{StabAnsatz}
 \begin{aligned}
  \|\u_h^n\|_{\Omega^n}^2  +  \|\u_h^n - \u_h^{n-1}\|_{\Omega^n}^2 - \|\u_h^{n-1}\|_{\Omega^n}^2  &+ \Delta t |||\u_h^n|||_{h,n}^2
 + \Delta t \gamma_p s_h(p_h^n, p_h^n)  \leq 2\Delta t (\f,\u_h^n)_{\Omega^n}.
  \end{aligned}
\end{eqnarray}

We bring the term $\|\u_h^{n-1}\|_{\Omega^n}^2$ to $\Omega^{n-1}$ by using Lemma~\ref{lem.domainExt}
\begin{eqnarray}\label{fromLemma}
\begin{aligned}
 \|\u_h^{n-1}\|_{\Omega^n}^2 
 &\leq \left(1 + {\sfrei c_1(w_{\max})} \Delta t \right)  \|\u_h^{n-1}\|_{\Omega^{n-1}}^2 
 + \frac{ \Delta t}{2} \|\nabla \u_h^{n-1}\|_{\Omega^{n-1}}^2\\
 &\qquad\qquad+  {\sfrei c_2(w_{\max})} {\cal K} \Delta t g_h^{n-1}(\u_h^{n-1}, \u_h^{n-1}).
\end{aligned}
\end{eqnarray}

Inserting \eqref{fromLemma} into \eqref{StabAnsatz} we have
 \begin{eqnarray}\label{StabFirstEstimate}
 \begin{aligned}
  \|\u_h^n&\|_{\Omega^n}^2   +\|\u_h^n-\u_h^{n-1}\|_{\Omega^n}^2 + \Delta t \left(||| \u_h^n|||_{h,n}^2  +  \gamma_p s_h(p_h^n, p_h^n)\right)  \\
 &\leq 2\Delta t (\f,\u_h^n)_{\Omega^n} + \left(1 + {\sfrei c_1(w_{\max})} \Delta t \right)  \|\u_h^{n-1}\|_{\Omega^{n-1}}^2
 + \frac{ \Delta t}{2} \|\nabla \u_h^{n-1}\|_{\Omega^{n-1}}^2\\
 &\qquad\qquad+  \Delta t \gamma_g g_h^{n-1}(\u_h^{n-1}, \u_h^{n-1})
  \end{aligned}
\end{eqnarray}
for $\gamma_g \geq {\sfrei c_2(w_{\max})} K$
and $\gamma_D$ sufficiently large.
For $n=1$, we have instead of \eqref{StabAnsatz}
 \begin{eqnarray}
 \begin{aligned}
  \|\u_h^1\|_{\Omega^1}^2  +  \|\u_h^1 - E\u^0\|_{\Omega^1}^2 - \|E\u^0\|_{\Omega^1}^2  + \Delta t \left(|||\u_h^1|||_{h,1}^2
  +\gamma_p s_h^1(p_h^1, p_h^1)\right)  = 2\Delta t (\f,\u_h^1)_{\Omega^1}.
  \end{aligned}
\end{eqnarray}
%
%
In both cases ($n\geq 1$) we use the Cauchy-Schwarz and Young's inequality for the first term on the right-hand side to get
\begin{align*}
 2\Delta t (\f,\u_h^n)_{\Omega^n} \leq  \Delta t \|\u_h^n\|_{\Omega^n}^2 + \Delta t \|\f\|_{\Omega^n}^2.
\end{align*}
%
%
%
Summing over $k=0,...,n$ in \eqref{StabFirstEstimate} and using the $L^2$-stability of the extension of the initial value yields
 \begin{eqnarray}
\begin{aligned}
 \|\u_h^n\|_{\Omega^n}^2 + \sum_{k=1}^n & \|\u_h^k-\u_h^{k-1}\|_{\Omega^k}^2 + \frac{\Delta t}{2} \left( |||\u_h^k|||_{h,k}^2 + 2 \gamma_p s_h^k(p_h^k, p_h^k) \right) \\
 &\leq c\|\u^0\|_{\Omega^0} + 2 t_n \|\f\|_{\infty,0}^2 + {\sfrei c_1(w_{\max})} \Delta t \sum_{k=0}^{n-1}   \|\u_h^k\|_{\Omega^k}^2.
\end{aligned}
\end{eqnarray}

Application of a discrete Gronwall lemma yields the statement.

\end{proof}

 
 \begin{remark}{(BDF(2))}
 For the BDF(2) variant, we get the 
 stability estimate \eqref{StabBDF1} with the weaker dissipation $\|\u_h^k-2\u_h^{k-1}+ \u_h^{k-2}\|_{\Omega^k}^2$ instead of $\|\u_h^k - \u_h^{k-1}\|_{\Omega^k}^2$. 
 To this end, one uses the relation 
 \begin{align*}
  (3\u_h^n - 4 \u_h^{n-1} + \u_h^{n-2}, \u_h^n)_{\Omega^n} = \frac{1}{2} \big( \|&\u_h^n\|_{\Omega^n}^2 -  \|\u_h^{n-1}\|_{\Omega^n}^2 + \|2 \u_h^n-\u_h^{n-1}\|_{\Omega^n}^2 \\
  &-  \|2\u_h^{n-1}-\u_h^{n-2}\|_{\Omega^n}^2 
  + \|\u_h^n - 2\u_h^{n-1} + \u_h^{n-1}\|_{\Omega^n}^2\big) 
 \end{align*}
 instead of \eqref{telescope}.
\end{remark}

\subsection{Stability estimate for the pressure}
\label{sec.pressstab}

We show the following stability estimates for the $L^2$- and $H^1$-semi-norm of pressure.

\begin{lemma}\label{lem.pressstab}
Let $(\u_h^n, p_h^n)$ be the discrete solution of \eqref{DiscSystem}. For $n\geq 1$ it holds that
 \begin{align}
  \|p_h^n\|_{\Omega^n}^2 &\leq c \left(\|D_t^{(s)} \u_h^n\|^2_{\Omega^n} + |||\u_h^n|||_{h,n}^2 + s_h^n(p_h^n,p_h^n) 
 + \|\f\|_{\Omega^n}^2\right),\label{pressL2stab}\\
 h^2\|\nabla p_h^n\|_{\Omega^n}^2 &\leq c \left(h^2 \|D_t^{(s)} \u_h^n\|_{\Omega^n}^2 + |||\u_h^n|||_{h,n}^2 + s_h^n(p_h^n,p_h^n) 
 + \|\f\|_{\Omega^n}^2\right),\label{pressH1stab}
\end{align}
where $\u_h^0 = \pi_h^1 E \u^0$.
\end{lemma}
\begin{proof}
First, we derive a bound for $h^2\|\nabla p_h^n\|^2_{\sfrei \Omega^n}$. To this end, we extend $\nabla p_h^n$ by zero to $\Omega_{\delta,h}^n\setminus \Omega_h^n$, using 
the same notation for the extended function.
We insert $\pm C_h^n \nabla p_h^n$, where $C_h^n: L^2(\Omega_{\delta,h}^n)^d \to {\cal V}_h^n$ is the interpolation
operator used in \eqref{stabass},
and integrate by parts
\begin{eqnarray}
\begin{aligned}\label{threeParts}
 h^2\|\nabla p_h^n\|^2_{\sfrei \Omega^n} &= h^2 (\nabla p_h^n - C_h^n\nabla p_h^n, \nabla p_h^n)_{\Omega^n} + h^2 (C_h^n\nabla p_h^n, \nabla p_h^n)_{\Omega^n} \\
 &= h^2 (\nabla p_h^n - C_h^n\nabla p_h^n, \nabla p_h^n)_{\Omega^n} - h^2 ({\rm div} (C_h^n\nabla p_h^n), p_h^n)_{\Omega^n}
 + h^2 (C_h^n\nabla p_h^n, p_h^n \n)_{\partial\Omega^n}.
\end{aligned}
\end{eqnarray}
For the first term, we have by means of \eqref{stabass} and Young's inequality
\begin{align}\label{PressH1firstterm}
 h^2 (\nabla p_h^n - C_h^n\nabla p_h^n, \nabla p_h^n)_{\Omega^n} \leq h^2 \|\nabla p_h^n - C_h^n\nabla p_h^n\|_{\Omega_h^n} \|\nabla p_h^n\|_{\Omega^n} \leq c s_h^n(p_h^n, p_h^n) + \frac{h^2}{4}\|\nabla p_h^n\|_{\Omega^n}^2.
\end{align}
{\sfrei The last term in \eqref{PressH1firstterm} will be absorbed into the left-hand side of \eqref{threeParts}.}
For the second term on the right-hand side of \eqref{threeParts}, we use that $(\u_h^n, p_h^n)$ solves the discrete system \eqref{DiscSystem}
\begin{eqnarray}
\begin{aligned}\label{h2divChn}
 - h^2 ({\rm div} (C_h^n\nabla p_h^n), p_h^n)_{\Omega^n} = &-h^2 \Big(  \left(D_t^{(s)} \u_h^n, C_h^n\nabla p_h^n\right)_{\Omega^n} + \left( \nabla \u_h^n, \nabla (C_h^n\nabla p_h^n)\right)_{\Omega^n} \\
 &+ a_D^n(\u_h^n, p_h^n; C_h^n\nabla p_h^n,0) +\gamma_g g_h^n(\u_h^n, C_h^n\nabla p_h^n) - \left(\f,C_h^n\nabla p_h^n\right)_{\Omega^n}
 \Big).
\end{aligned}
\end{eqnarray}
{\sfrei To estimate the first term on the right-hand side of \eqref{h2divChn}, we use the Cauchy-Schwarz inequality
and \eqref{stabass} to get 
\begin{align}\label{Dtuhn1}
-h^2 \left(D_t^{(s)} \u_h^n, C_h^n\nabla p_h^n\right)_{\Omega^n} \leq ch \|D_t^{(s)} \u_h^n\|_{\Omega^n} \, \left(h\|\nabla p_h^n\|_{\Omega^n} +  s_h^n(p_h^n, p_h^n)\right).
\end{align}
Similarly, we get for the second and the last term on the right-hand side of \eqref{h2divChn}
\begin{eqnarray}
\begin{aligned}
-h^2 \left( \nabla \u_h^n, \nabla (C_h^n\nabla p_h^n)\right)_{\Omega^n} &\leq c \|\nabla \u_h^n\|_{\Omega^n}
\left( h\|\nabla p_h^n\|_{\Omega^n} + s_h^n(p_h^n, p_h^n)\right), \\
h^2\left(\f,C_h^n\nabla p_h^n\right)_{\Omega^n} &\leq c \|\f\|_{\Omega^n} \left( h\|\nabla p_h^n\|_{\Omega^n} + s_h^n(p_h^n, p_h^n)^{1/2}\right).
\end{aligned}
\end{eqnarray}
}
For the Nitsche term $a_D^n$, we have as in \eqref{Nitsche4p} 
\begin{eqnarray}
\begin{aligned}\label{aDn}
 -h^2 a_D^n&(\u_h^n, p_h^n; C_h^n\nabla p_h^n,0) \\
 &=-h^2\left( \frac{\gamma_D}{h} \left(\u_h^n, C_h^n\nabla p_h^n\right)_{\partial\Omega^n} - \left(\partial_n \u_h^n - p_h^n \n, C_h^n\nabla p_h^n\right)_{\partial\Omega^n} 
 + \left(\u_h^n,  \partial_n (C_h^n\nabla p_h^n)\right)_{\partial\Omega^n}\right)  \\
 &\leq ch |||u_h^n|||_{h,n} { \|C_h^n \nabla p_h^n\|_{\Omega_h^n}}  - h^2\left(p_h^n \n, C_h^n\nabla p_h^n\right)_{\partial\Omega^n}\\
 &\leq c |||u_h^n|||_{h,n} { \left(h \|\nabla p_h^n\|_{\Omega^n} + s_h(p_h^n, p_h^n)^{1/2}\right)}  
 - h^2\left(p_h^n \n, C_h^n\nabla p_h^n\right)_{\partial\Omega^n}.
\end{aligned}
\end{eqnarray}
In the last step \eqref{stabass} has been used.
Note that the boundary term on the right-hand side will cancel out with the third term in \eqref{threeParts}. For the ghost penalty {\sfrei we have by means of  
an inverse inequality and \eqref{stabass}
\begin{eqnarray}
\begin{aligned}\label{ghostghn}
 h^2 g_h^n(\u_h^n, C_h^n\nabla p_h^n) &\leq ch g_h^n(\u_h^n, \u_h^n)^{1/2} 
 \|C_h^n \nabla p_h^n\|_{\Omega_{h,\delta}^n} \\
 &\leq c g_h^n(\u_h^n, \u_h^n)^{1/2} 
  \left( h\|\nabla p_h^n\|_{\Omega^n} + s_h^n(p_h^n, p_h^n)^{1/2}\right).
\end{aligned}
\end{eqnarray}
Altogether, \eqref{h2divChn}-\eqref{Dtuhn1} and \eqref{aDn}-\eqref{ghostghn} result in 
\begin{eqnarray}
\begin{aligned}\label{divTerm}
 - h^2 ({\rm div} &(C_h^n\nabla p_h^n), p_h^n)_{\Omega^n} + h^2\left(p_h^n \n, C_h^n\nabla p_h^n\right)_{\partial\Omega^n}\\
&\leq c\left(h\|\nabla p_h^n\|_{\Omega^n} + s_h^n(p_h^n, p_h^n)^{1/2}\right) \Big(h \big\|D_t^{(s)} \u_h^n\big\|_{\Omega^n} + |||\u_h^n|||_{h,n} + \| \f\|_{\Omega^n}\Big)\\
&\leq \frac{h^2}{4} \|\nabla p_h^n\|_{\Omega^n}^2 
 + c \Big(h^2 \big\|D_t^{(s)} \u_h^n\big\|_{\Omega^n}^2 
 + |||\u_h^n|||_{h,n}^2 + \| \f\|_{\Omega^n}^2 + s_h^n(p_h^n, p_h^n) \Big).
\end{aligned}
\end{eqnarray}
In the last step we have applied Young's inequality.
Combination of \eqref{threeParts}, \eqref{PressH1firstterm} an \eqref{divTerm} yields \eqref{pressH1stab}.}
To show \eqref{pressL2stab} we start using the modified inf-sup condition (Lemma~\ref{lem.modinfsup})
 \begin{align}\label{infsup3}
  \beta \|p_h^n\|_{\Omega^n} \leq \sup_{\v_h \in {\cal V}_h^n} \frac{({\rm div}\, \v_h, p_h^n)_{\Omega^n} - (\v_h \cdot \n, p_h^n)_{\partial\Omega^n}}{|||\v_h|||_{h,n}} + h\|\nabla p_h^n\|_{\Omega^n}.
 \end{align}
{\sfrei By \eqref{DiscSystem}, we have
\begin{eqnarray}
\begin{aligned}\label{StartPress}
({\rm div}\, \v_h, p_h^n)_{\Omega^n} - (\v_h \cdot n, p_h^n)_{\partial\Omega^n} &=
   (D_t^{(s)} \u_h^n, \v_h)_{\Omega^n} + {\cal A}_h^n(\u_h^n, 0; \v_h, 0)
      -(\f,\v_h)_{\Omega^n}
\end{aligned}
\end{eqnarray}
To estimate the right-hand side of \eqref{StartPress}, we use the continuity of the bilinear form ${\cal A}_h^n$ \eqref{contDisc} and the Cauchy-Schwarz inequality
 \begin{eqnarray}
 \begin{aligned}\label{finalPress}
  (D_t^{(s)} \u_h^n, \v_h)_{\Omega^n} &+ {\cal A}_h^n(\u_h^n, 0; \v_h, 0)
      -(\f,\v_h)_{\Omega^n}\\
     &\leq c \left( \|D_t^{(s)} \u_h^n\|_{\Omega^n} + \|\f\|_{\Omega^n} \right) \|\v_h\|_{\Omega^n}  + c |||\u_h^n|||_{h,n} |||\v_h|||_{h,n}  \\ 
  &\leq c \left( \|D_t^{(s)} \u_h^n\|_{\Omega^n} + |||\u_h^n|||_{h,n} + \|\f\|_{\Omega^n} \right) |||\v_h|||_{h,n}.
  \end{aligned}
  \end{eqnarray}
  In the last step, we have used that $\|\v_h^n\|_{\Omega^n} \leq c \left( \|\nabla \v_h\|_{\Omega^n} + \|\v_h\|_{\partial\Omega^n}\right) \leq c ||| \v_h|||_{h,n}$ by a Poincar\'e- type estimate.}
Combination of \eqref{infsup3}-\eqref{finalPress} and \eqref{pressH1stab} yields \eqref{pressL2stab}.
\end{proof}

Lemma~\ref{lem.pressstab} gives a stability result for $\|\nabla p_h^k\|_{\Omega^k}$, which results in the following corollary:
\begin{corollary}\label{cor.press}
 Under the assumptions of Theorem~\eqref{theo.stab}, it holds for $s=1$ that
 \begin{eqnarray}
  \begin{aligned}\label{corpress1}
  \|\u_h^n\|_{\Omega^n}^2  &+ \Delta t \sum_{k=1}^n \left( |||\u_h^k|||_{h,k}^2 +\frac{1}{\Delta t} \|\u_h^k-\u_h^{k-1}\|_{\Omega^k}^2+ \gamma_p s_h^k(p_h^k,p_h^k) + \min \{h^2,\Delta t\} \|\nabla p_h^k\|_{\Omega^k}^2\right) \\
  &\leq \exp( {\sfrei c_1(w_{\max})} t_n) \left(c \|\u^0\|_{\Omega^0}  +  2t_n\|\f\|_{\infty,0}^2 \right). 
 \end{aligned}
 \end{eqnarray}
 For $s= 2$, we have
 \begin{eqnarray}
 \begin{aligned}\label{stabslarger1}
    \|\u_h^n\|_{\Omega^n}^2  &+ \Delta t \sum_{k=1}^n \left( |||\u_h^k|||_{h,k}^2 + \gamma_p s_h^k(p_h^k,p_h^k) + \min \{h^2,\Delta t^2\} \|\nabla p_h^k\|_{\Omega^k}^2\right) \\
  &\leq \exp( {\sfrei c_1(w_{\max})} t_n) \left(c \|\u^0\|_{\Omega^0}  +  2t_n\|\f\|_{\infty,0}^2 \right).
 \end{aligned}
 \end{eqnarray}
\end{corollary}
\begin{proof}
{\sfrei We start by proving \eqref{corpress1} for $s=1$. To this end, we distinguish between the cases $\Delta t\geq h^2$ and $\Delta t<h^2$. In the first case, we note that, by \eqref{pressH1stab}
\begin{eqnarray}
 \begin{aligned}\label{cor1}
 \Delta t h^2\|\nabla p_h^n\|_{\Omega^n}^2 &\leq c \Delta t h^2 \|D_t^{(1)} \u_h^n\|_{\Omega^n}^2 + c\Delta t \left( |||\u_h^n|||_{h,n}^2 + s_h^n(p_h^n,p_h^n) 
 + \|\f\|_{\Omega^n}^2\right)\\
 &\leq c \Delta t^2 \|D_t^{(1)} \u_h^n\|_{\Omega^n}^2 + c\Delta t \left( |||\u_h^n|||_{h,n}^2 + s_h^n(p_h^n,p_h^n) 
 + \|\f\|_{\Omega^n}^2\right).
\end{aligned}
\end{eqnarray}
As $\Delta t^2 \| D_t^{(1)} \u_h^n\|_{\Omega^n}^2 = \|\u_h^n - \u_h^{n-1}\|_{\Omega^n}^2$ \eqref{pressH1stab} follows from Theorem~\ref{theo.stab}.} For $\Delta t < h^2$, we multiply \eqref{cor1} by $\frac{\Delta t}{h^2}$ to get
\begin{eqnarray*}
 \begin{aligned}
   \Delta t^2\|\nabla p_h^n\|_{\Omega^n}^2 &\leq c\Delta t^2 \|D_t^{(1)} \u_h^n\|_{\Omega^n}^2 + c\Delta t \frac{\Delta t}{h^2} \left( |||\u_h^n|||_{h,n}^2 + s_h^n(p_h^n,p_h^n) 
 + \|\f\|_{\Omega^n}^2\right)
 \end{aligned}
 \end{eqnarray*}
and use the same argumentation. For $s=2$, we do not have control over $\Delta t^2 \| D_t^{(2)} \u_h^n\|_{\Omega^n}^2$. Instead, we use the estimate
\begin{align*}
 \Delta t^3\| D_t^{(2)} \u_h^n\|_{\Omega^n}^2 \leq c\Delta t \sum_{k=0}^2 \left(\|\u_h^{n-k} \|_{\Omega^{n-k}}^2 + \Delta t \|\nabla \u_h^{n-k}\|_{\Omega^{n-k}}^2 + {\cal K} \Delta t g_h^{n-k}(\u_h^{n-k}, \u_h^{n-k})\right)
\end{align*}
that follows from the triangle inequality and \eqref{secondStatement}. The estimate \eqref{stabslarger1} follows by a similar argumentation by distinguishing between the cases $\Delta t \lessgtr h$. 
\end{proof}


Concerning the $L^2$-norm of the pressure, Lemma~\ref{lem.pressstab} gives a stability result only for 
\begin{align*}
 \Delta t^2 \sum_{k=1}^n \|p_h^k\|_{\Omega^k}^2,
\end{align*}
even in the case $s=1$. In the case of fixed domains and fixed discretisations, a stability estimate for $\|p_h^k\|$ can be derived by showing an upper bound for the right-hand side in \eqref{pressL2stab}, including the term $\frac{1}{\Delta t} \|\u_h^n-\u_h^{n-1}\|_{\Omega^n}^2$, see for 
example Besier \& Wollner~\cite{BesierWollner}.
The argumentation requires, however, that the term $({\rm div}\, \u_h^{n-1}, {\sfrei \xi_h^n})_{\Omega^n}$ vanishes for $\xi_h^n \in {\cal L}_h^n$.
This is not true in the case of time-dependent domains, as $\u_h^{n-1}$ is not discrete divergence-free with respect to $\Omega^n$
\begin{align*}
 ({\rm div}\, \u_h^{n-1}, {\sfrei \xi_h^n})_{\Omega^n} \neq 0
\end{align*}
for certain $\xi_h^n \in {\cal L}_h^n$.
Moreover, the domain mismatch $\Omega^{n-1} \neq \Omega^n$ causes additional problems in the transfer of the term $|||\u_h^{n-1}|||_{h,n} \neq |||\u_h^{n-1}|||_{h,n-1}$ from one time level to the previous one.
In the error analysis developed in the following section, we will therefore use the $H^1$-stability results in Corollary~\ref{cor.press} for the pressure variable.

\section{Error analysis}
\label{sec.error}

The energy error analysis for the velocities follows largely the argumentation of Lehrenfeld \& Olshanskii~\cite{LehrenfeldOlshanskii} and is based on Galerkin orthogonality and the stability result of Theorem~\ref{theo.stab}. 
We write $\u^n:= \u(t_n), p^n:=p(t_n)$ and introduce the notation 
\begin{align*}
 \e_u^n  := \u^n - \u_h^n, \quad \etab_u^n:= \u^n - I_h^n \u^n, \quad \xib_{h,u}^n := I_h^n \u^n - \u_h^n,\\
 e_p^n  := p^n - p_h^n, \quad \eta_p^n := p^n - i_h^n p^n, \quad \xi_{h,p}^n := i_h^n p^n - p_h^n
\end{align*}
for $n\geq 1$, where $I_h^n$ denotes the standard Lagrangian nodal interpolation to ${\cal T}_{h,\delta}^n$ and $i_h^n$ a generalised $L^2$-stable interpolation (for example the Cl\'ement interpolation) 
to ${\cal T}_{h}^n$. Moreover, we set 
\begin{align*}
 \e_u^0=\etab_u^0=\xib_{h,u}^0 = 0.
\end{align*}
This is possible, as $\u_h^0$ cancels out in the summed space-time system \eqref{completeSpaceTime}.
The following estimates for the interpolation errors are well-known
\begin{align}
 \|\etab_u^n\|_{H^l(\Omega)} &\leq ch^{k-l} \|\u^n\|_{H^k(\Omega)} \quad &&\text{for }\, 0\leq l \leq 1,\quad 2 \leq k\leq m+1,\label{interpol1}\\
 \|\eta_p^n\|_{H^l(\Omega)} &\leq ch^{k-l} \|p^n\|_{H^k(\Omega)} \quad &&\text{for } \, 0\leq l\leq 1, \quad 1\leq k\leq m+1,\label{interpol2}\\
 \|\eta_p^n\|_{H^l(\partial\Omega)} &\leq ch^{k-l-1/2} \|p^n\|_{H^k(\Omega)} \quad &&\text{for }\, 0\leq l\leq 1, \quad 1< k\leq m+1.\label{interpolTrace}
\end{align}
We will again make use of the extension operators $E^n$ introduced in Section~\ref{sec.ext}. For better readability, we will sometimes skip the operators $E^n$ assuming that quantities
that would be undefined on the domains of integration are extended smoothly. 

For the error analysis, we assume that the solution $(\u,p)$ to \eqref{Stokes} lies in $L^2(I, H^{m+1}(\Omega(t))^d) \times L^2(I,H^{m}(\Omega(t)))$ for $m\geq 1$. 
Then, we can incorporate the Nitsche terms in the variational formulation on the continuous level and see that $(\u,p)$ is the solution to
\begin{align}\label{StokesNitsche}
 (\partial_t \u, \v)_{\Omega(t)} + {\cal A}_S(\u,p;\v,q) + a_D^n(\u, p; \v, q) &= (\f, \v)_{\Omega(t)} \quad \forall \v \in \tilde{\cal V}(t), \; q \in {\cal L}(t) \quad \text{ a.e. in } t \in I,
\end{align}
where 
\begin{align*}
 \tilde{\cal V}(t) &:= H^1(\Omega(t))^d.
\end{align*}

\subsection{Energy error}

As a starting point for the error estimation, we subtract \eqref{DiscSystem} from \eqref{StokesNitsche} to obtain the orthogonality relation
\begin{eqnarray}\label{Galerkin}
 \begin{aligned}
   \big(D_t^{(s)} &\e_u^n, \v_h\big)_{\Omega^n}  + ({\cal A}_S^n+a_D^n)(\e_u^n,e_p^n;\v_h,q_h) 
 + \gamma_g g_h(\e_u^n, \v_h) + \gamma_p s_h(\xi_{h,p}^n, q_h)  \\
 &= \underbrace{ (D_t^{(s)} \u^n - \partial_t \u(t_n), \v_h)_{\Omega^n} + \gamma_g g_h(\u^n, \v_h) + \gamma_p s_h(i_h^n p^n, q_h)}_{ =:{\cal E}_c^n(\v_h, q_h)} 
 \qquad \forall \v_h \in {\cal V}_h^n, q_h \in {\cal L}_h^n, 
  \end{aligned}
\end{eqnarray}
for $n\geq s$ with the consistency error ${\cal E}_c^n(\v_h, q_h)$. 
Note that this relation holds in particular also for $n=s$, as we have defined $\e_u^0=0$.
We have used a different splitting in the pressure stabilisation $s_h^n$ compared to the other terms, in order to include 
the case $p\in H^1(\Omega)$ ($m=1$), where $s_h^n(p^n, q_h)$ would not be well-defined.

We further split \eqref{Galerkin} into interpolation and discrete error parts
 \begin{eqnarray}\label{Galerkin2}
 \begin{aligned}
  \left(D_t^{(s)} \xib_{h,u}^n, \v_h\right)_{\Omega^n}  + \left({\cal A}_S^n + a_D^n\right)&(\xib_{h,u}^n,\xi_{h,p}^n;\v_h,q_h) 
 + \gamma_g g_h^n(\xib_{h,u}^n, \v_h)
 + \gamma_p s_h(\xi_{h,p}^n, q_h) \\
 &= -{\cal E}_i^n(\v_h, q_h) + {\cal E}_c^n(\v_h, q_h) 
 \quad \forall \v_h \in {\cal V}_h^n, q_h \in {\cal L}_h^n, 
  \end{aligned}
\end{eqnarray}
where the interpolation error is defined by 
\begin{align}\label{interpolError}
{\cal E}_i^n(\v_h, q_h) :=  (D_t^{(s)}\etab_u^n , \v_h)_{\Omega^n}  + \left({\cal A}_S^n + a_D^n\right)(\etab_u^n,\eta_p^n;\v_h,q_h)
 + \gamma_g g_h(\etab_u^n, \v_h).
\end{align}
We will apply the stability result of Theorem~\ref{theo.stab} to \eqref{Galerkin2}, which will be the basis of the error estimate. For better readability, we will restrict restrict ourselves again to the case 
$s=1$ first. Let us first estimate the consistency and interpolation errors.

\begin{lemma}{(Consistency error)}\label{lem.consistency}
 Let $\u \in W^{2,\infty}(I_n, L^2(\Omega^n)^d) \cap L^{\infty}(I_n, H^{m+1}(\Omega^n)^d)$ and $p\in L^{\infty}(I_n, H^{m}(\Omega^n))$. Under the assumptions made in Section~\ref{sec.disc}, {\sfrei including Assumption~\ref{ass.T}}, it holds for $s=1$,
 $\v_h\in {\cal V}_h^n, q_h\in {\cal L}_h^n$ and $n\geq 1$ that
 \begin{align*}
  |{\cal E}_c^n(\v_h, q_h)| \leq \,&c {\sfrei \Delta t^{\frac{1}{2}} \|\partial_t^2 \u\|_{Q^n}}  \|\v_h\|_{\Omega^n}\\
  &+ ch^{m} \left(\|\u\|_{H^{m+1}(\Omega^n)} + \|p\|_{H^{m}(\Omega^n)}\right) \left(g_h^n(\v_h, \v_h)^{1/2} + s_h^n(q_h, q_h)^{1/2}\right).
 \end{align*}
\end{lemma}
\begin{proof}
For the first part of the consistency error, we have using integration by parts {\sfrei and a Cauchy-Schwarz inequality in time
\begin{align*}
 \frac{1}{\Delta t} \left(\u^n - E^n \u^{n-1}\right) - \partial_t \u(t_n) &= -\frac{1}{\Delta t} \int_{t_{n-1}}^{t_n} \partial_t (E^n \u(t)) - \partial_t (E^n \u(t_n)) \, d\text{t} \\
 &= -\frac{1}{\Delta t} \int_{t_{n-1}}^{t_n} (t-t_{n-1})\partial_t^2 (E^n \u(t)) \, d\text{t}\\
 &\leq \frac{1}{\Delta t} \left(\int_{t_{n-1}}^{t_n} (t-t_{n-1})^2 d\text{t}\right)^{1/2} \left(\int_{t_{n-1}}^{t_n} \partial_t^2 (E^n \u(t))^2 d\text{t}\right)^{1/2}\\
 &\leq \Delta t^{1/2} \left(\int_{t_{n-1}}^{t_n} \partial_t^2 (E^n \u(t))^2 d\text{t}\right)^{1/2}.
\end{align*}}
{\sfrei Using \eqref{Extdt2}} this implies
\begin{eqnarray}
\begin{aligned}\label{consTime1}
\Big|\frac{1}{\Delta t} \left(\u^n - E^n\u^{n-1}, \v_h\right)_{\Omega^n} - (\partial_t \u(t_n), \v_h)_{\Omega^n}\Big| &\leq c {\sfrei \Delta t^{1/2} \|\partial_t^2 (E^n \u)\|_{Q_\delta^n}} \|\v_h\|_{\Omega^n}\\
&\leq c {\sfrei \Delta t^{1/2} \|\partial_t^2 \u\|_{Q^n}} \|\v_h\|_{\Omega^n}.
\end{aligned}
\end{eqnarray}
{\sfrei The extension operator $E^n$ is needed, as the integration domain in the left-hand side of \eqref{consTime1} includes parts, that lie outside the physical domain $Q^n$. }
For the ghost penalty part, we have with Lemma~\ref{lem.ghost} and the $H^{m+1}$-stability of the extension \eqref{extension}
\begin{align*}
 g_h^n(\u^n, \v_h) \leq g_h^n(\u^n, \u^n)^{1/2} g_h^n(\v_h, \v_h)^{1/2} 
 &\leq ch^m \|{\sfrei \u^n}\|_{H^{m+1}(\Omega^n)} g_h^n(\v_h, \v_h)^{1/2}
\end{align*}
Concerning the pressure stabilisation, we note that for $p^n \in H^1(\Omega^n)$ the term $s_h^n(p^n, p^n)$ is not well-defined. 
For this reason we distinguish between the cases $m=1$ and $m\geq 2$. In the first case, we estimate using \eqref{stabass}
and the $H^1$-stability of the interpolation
 \begin{align*}
  s_h^n(i_h^n p^n, q_h) 
  &\leq ch \|i_h^n p^n\|_{H^1(\Omega^n)} s_h^n(q_h, q_h)^{1/2}
  \leq ch \|p^n\|_{H^1(\Omega^n)} s_h^n(q_h, q_h)^{1/2}.
 \end{align*}
For $m\geq 2$, we insert $\pm p^n$ and use \eqref{stabass} and the interpolation error estimate \eqref{interpolTrace}
\begin{align*}
 s_h^n(i_h^n p^n, q_h) 
 &\leq \left(s_h^n(\eta_p^n, \eta_p^n)^{1/2} + s_h^n(p^n, p^n)^{1/2}
 \right) s_h^n(q_h, q_h)^{1/2} 
 \leq ch^m \|{\sfrei  p^n}\|_{H^m(\Omega^n)} s_h^n(q_h, q_h)^{1/2}.
\end{align*}
\end{proof}

\begin{lemma}{(Interpolation error)}\label{lem.interpol}
Let $\u\in L^{\infty}(I_n, H^{m+1}(\Omega(t))^d), \partial_t \u \in L^{\infty}(I_n, H^{m}(\Omega(t))^d), p \in L^{\infty}(I_n,H^{m}(\Omega(t))$. Under 
the assumptions made in Section~\ref{sec.disc}, {\sfrei including Assumption~\ref{ass.T}}, it holds for $\v_h\in{\cal V}_h^n$ and $q_h\in {\cal L}_h^n$ that
\begin{align*}
  |{\cal E}_i^n(\v_h, q_h)| &\leq ch^m \left( \|\u\|_{\infty, m+1,I_n} + \|\partial_t \u\|_{\infty, m,I_n} 
  + \|p\|_{H^{m}(\Omega^n)}\right)  \left(|||\v_h|||_{h,n} + h \|\nabla q_h\|_{\Omega^n}\right).
 \end{align*}
\end{lemma}
\begin{proof}
 We estimate the interpolation error \eqref{interpolError} term by term. For the first term we use that we can exchange time derivative and 
 interpolation operator $\partial_t I_h \u^n = I_h \partial_t \u (t_n)$
 \begin{eqnarray}
 \begin{aligned}\label{interpoldteta}
  \big|\frac{1}{\Delta t} (\etab_u^n - \etab_u^{n-1}, \v_h)_{\Omega^n}\big| \leq \frac{1}{\Delta t} \|\etab_u^n - \etab_u^{n-1}\|_{\Omega^n} \|\v_h\|_{\Omega^n} 
  &= \frac{1}{\Delta t} \Big\| \int_{t_{n-1}}^{t_n} \partial_t (\u(t) - I_h \u(t)) \, d\text{t} \Big\|_{\Omega^n} \|\v_h\|_{\Omega^n}\\
  &\leq h^m \|\partial_t (E^n \u)\|_{\infty,m,I_n} \|\v_h\|_{\Omega^n}. 
 \end{aligned}
 \end{eqnarray}
 We note again that the integration domain in the first norm on the right-hand side includes parts, that might lie outside the physical domain $Q^n$. 
 By means of \eqref{Extdt} we conclude
 \begin{align*}
  \Big| \frac{1}{\Delta t} (\etab_u^n - \etab_u^{n-1}, \v_h)_{\Omega^n}\Big| \leq ch^m \left( \|\partial_t \u\|_{\infty,m} + \|\u\|_{\infty, m+1} \right) \|\v_h\|_{\Omega^n}
 \end{align*}
 
 For the second term in \eqref{interpolError}, we use Lemma~\ref{lem.contCont}
 \begin{eqnarray}
 \begin{aligned}\label{A4interpol}
 ({\cal A}_S^n &+ a_D^n)(\etab_u^n,\eta_p^n;\v_h,q_h)\\ 
 &\leq c \left( \|\nabla \etab_u^n\|_{\Omega^n} + h^{-1} \|\etab_u^n\|_{\Omega^n} + h^{-1/2} \|\etab_u^n\|_{\partial\Omega^n} + h^{1/2} \left(\|\etab_u^n\|_{\partial\Omega^n} +\|\eta_p^n\|_{\partial\Omega^n} \right)  +  \|\eta_p^n\|_{\Omega^n}\right)\\
 &\qquad\qquad\cdot \left(\|\nabla \v_h\|_{\Omega^n} +h^{-1/2} \|\v_h\|_{\partial\Omega^n} + h^{1/2}\|\partial_n \v_h\|_{\partial\Omega^n} + h\|\nabla q_h\|_{\Omega^n} \right) \\
 &\leq ch^m \left(\|{\sfrei \u^n}\|_{H^{m+1}(\Omega^n)} + \|{\sfrei p^n}\|_{H^m(\Omega^n)} \right) \left( |||\v_h|||_{h,n} + h\|\nabla q_h\|_{\Omega^n} \right) 
 \end{aligned}
 \end{eqnarray}
 Finally, we get for the ghost penalty part from \eqref{ghbound} and the $H^{m+1}$-stability of the extension
 \begin{align*}
  g_h^n(\etab_u^n, \v_h) \leq ch^m \|{\sfrei \u^n}\|_{H^{m+1}(\Omega_\delta^n)} g_h^n(\v_h, \v_h)^{1/2} \leq ch^m \|{\sfrei \u^n}\|_{H^{m+1}(\Omega^n)} g_h^n(\v_h, \v_h)^{1/2}.
 \end{align*}
\end{proof}

Now, we are ready to show an error estimate for the velocities.

\begin{theorem}\label{theo.energyerror}
 Let $\u_h = (\u_h^k)_{k=1}^n, p_h = (p_h^k)_{k=1}^n$ be the discrete solution of \eqref{DiscSystem} for $s=1$ and $(\u,p)$ the continuous solution of \eqref{Stokes}. Further, let  
 $\gamma_g \geq {\sfrei c_2(w_{\max})} K$ with ${\sfrei c_2(w_{\max})}$ defined in Lemma~\ref{lem.domainExt},
 $\gamma_D, \gamma_p$ sufficiently large and $\Delta t\geq ch^2$ for some $c>0$. 
 Under the assumptions stated in Section~\ref{sec.disc}, {\sfrei including Assumption~\ref{ass.T}}, it holds for the error $\e_u^k=\u^k-\u_h^k, e_p^k=p^k-p_h^k$ for $n\geq 1$
 \begin{align*}
 \|\e_u^n\|_{\Omega^n}^2 &+\sum_{k=1}^n \left\{ \|\e_u^k - \e_u^{k-1}\|_{\Omega^k}^2 + \Delta t \left( ||| \e_u^k|||_{h,k}^2 + h^2\|\nabla e_p^k\|_{\Omega^k}^2\right)\right\}\\
 &\leq c\exp({\sfrei c_1(w_{\max})} t_n) \Big( \Delta t^2 \|\partial_t^2 \u\|_{Q}^2 +h^{2m} \left( \|\u\|_{\infty,m+1}^2 + \|\partial_t \u\|_{\infty,m}^2
  + \|p\|_{\infty,m}^2\right) \Big),
\end{align*}
where $\e_u^0:=0$ and {\sfrei ${\sfrei c_1(w_{\max})}$ is defined in Lemma~\ref{lem.domainExt}.}
\end{theorem}
\begin{proof}
As in the stability proof (Theorem~\ref{theo.stab}, \eqref{StabFirstEstimate}), we obtain from \eqref{Galerkin2} for $n\geq 1$
  \begin{eqnarray}
 \begin{aligned}\label{startenergy}
  \|\xib_{h,u}^n&\|_{\Omega^n}^2 +\|\xib_{h,u}^n-\xib_{h,u}^{n-1}\|_{\Omega^n}^2 +\Delta t \left(|||\xib_{h,u}^n|||_{h,n}^2
+ \gamma_p s_h(\xi_{h,p}^n, \xi_{h,p}^n) \right)\\
 &\leq \left(1 + {\sfrei c_1(w_{\max})} \Delta t \right)  \|\xib_{h,u}^{n-1}\|_{\Omega^{n-1}}^2 
+ \frac{ \Delta t}{2} \|\nabla \xib_{h,u}^{n-1}\|_{\Omega^{n-1}}^2
 +  \Delta t \gamma_g g_h^{n-1}(\xib_{h,u}^{n-1}, \xib_{h,u}^{n-1})\\
 &\qquad\qquad\qquad+ 2 \Delta t \left( \big| {\cal E}_c^n(\xib_{h,u}^n, \xi_{h,p}^n)\big| + \big| {\cal E}_i^n(\xib_{h,u}^n, \xi_{h,p}^n)\big| \right) .
  \end{aligned}
\end{eqnarray}
for $\gamma_g \geq {\sfrei c_2(w_{\max})} {\cal K}$. {\sfrei A bound for $\|\nabla \xi_{h,p}^n\|_{\Omega^n}$ can be obtained from \eqref{Galerkin} as in the proof of Lemma~\ref{lem.pressstab} (compare \eqref{pressH1stab})
\begin{eqnarray}
\begin{aligned}\label{xihpn}
\Delta t h^2\|\nabla \xi_{h,p}^n\|_{\Omega^n}^2 &\leq c \Delta t \Big(h^2 \|D_t^{(1)} \xib_{h,u}^n\|_{\Omega^n}^2 + |||\xib_{h,u}^n|||_{h,n}^2 + s_h^n(\xi_{h,p}^n,\xi_{h,p}^n) \\
 &\qquad+ \big| {\cal E}_c^n(\xib_{h,u}^n, \xi_{h,p}^n)\big| + \big| {\cal E}_i^n(\xib_{h,u}^n, \xi_{h,p}^n)\big| \Big).
\end{aligned}
\end{eqnarray}
We multiply \eqref{xihpn} by $\epsilon>0$ and add it to \eqref{startenergy}. 
Due to the assumption $\Delta t \geq ch^2$ the first three terms on the right-hand side of \eqref{xihpn} can be absorbed into the left-hand side of \eqref{startenergy} for sufficiently small $\epsilon$
\begin{eqnarray}
 \begin{aligned}
  \|\xib_{h,u}^n&\|_{\Omega^n}^2 +\frac{3}{4}\|\xib_{h,u}^n-\xib_{h,u}^{n-1}\|_{\Omega^n}^2 +\frac{3\Delta t}{4} \left(|||\xib_{h,u}^n|||_{h,n}^2
+ \gamma_p s_h(\xi_{h,p}^n, \xi_{h,p}^n) 
+ \epsilon h^2\|\nabla \xi_{h,p}^n\|_{\Omega^n}^2
\right)\\
 &\leq \left(1 + {\sfrei c_1(w_{\max})} \Delta t \right)  \|\xib_{h,u}^{n-1}\|_{\Omega^{n-1}}^2 
+ \frac{ \Delta t}{2} \|\nabla \xib_{h,u}^{n-1}\|_{\Omega^{n-1}}^2
 +  \Delta t \gamma_g g_h^{n-1}(\xib_{h,u}^{n-1}, \xib_{h,u}^{n-1})\\
 &\qquad\qquad\qquad+ 2 \Delta t \left( \big| {\cal E}_c^n(\xib_{h,u}^n, \xi_{h,p}^n)\big| + \big| {\cal E}_i^n(\xib_{h,u}^n, \xi_{h,p}^n)\big| \right) .
  \end{aligned}
\end{eqnarray}
}
Next, we use Lemmata~\ref{lem.consistency} and~\ref{lem.interpol} in combination with Young's inequality to estimate ${\cal E}_c^n$ and ${\cal E}_i^n$
  \begin{eqnarray}
 \begin{aligned}\label{energyonestep}
  (1-\Delta &t) \|\xib_{h,u}^n\|_{\Omega^n}^2 + \frac{1}{2}\|\xib_{h,u}^n-\xib_{h,u}^{n-1}\|_{\Omega^n}^2 +\frac{\Delta t}{2}\left( |||\xib_{h,u}^n|||_{h,n}
 + \gamma_p s_h(\xi_{h,p}^n, \xi_{h,p}^n) + \epsilon h^2\|\nabla \xi_{h,p}^n\|_{\Omega^n}^2\right)  \\
 &\leq \left(1 + {\sfrei c_1(w_{\max})} \Delta t \right)  \|\xib_{h,u}^{n-1}\|_{\Omega^{n-1}}^2 
  + \frac{ \Delta t}{2} \|\nabla \xib_{h,u}^{n-1}\|_{\Omega^{n-1}}^2
 +  \Delta t \gamma_g g_h^{n-1}(\xib_{h,u}^{n-1}, \xib_{h,u}^{n-1})\\
 &\qquad\;\;+c \Delta t \Big( \Delta t \|\partial_t^2 \u\|_{Q^n}^2
 +h^{2m} \left( \|\u\|_{\infty,m+1,I_n}^2 + \|\partial_t \u\|_{\infty,m,I_n}^2
  + \|p\|_{H^{m}(\Omega^n)}^2\right) \Big).
  \end{aligned}
\end{eqnarray}
We sum over $k=1,\dots,n$ and apply a discrete Gronwall lemma to find
\begin{eqnarray}
\begin{aligned}\label{AfterGronwall}
 \|\xib_{h,u}^n\|_{\Omega^n}^2 &+\sum_{k=1}^n \left( \|\xib_{h,u}^k - \xib_{h,u}^{k-1}\|_{\Omega^k} + \Delta t\left( ||| \xib_{h,u}^k|||_{h,k}^2 + \gamma_p s_h^k(\xi_{h,p}^k, \xi_{h,p}^k) + h^2\|\nabla \xi_{h,p}^k\|_{\Omega^k}^2\right)\right)\\
 &\leq c\exp({\sfrei c_1(w_{\max})} t_n) \Big( \Delta t^2 \|\partial_t^2 \u\|_{Q}^2
 +h^{2m} \left( \|\u\|_{\infty,m+1}^2 + \|\partial_t \u\|_{\infty,m}^2
  + \|p\|_{\infty,m}^2\right) \Big).
\end{aligned}
\end{eqnarray}

Finally, the interpolation estimates \eqref{interpol1}-\eqref{interpolTrace} and the argumentation used in \eqref{interpoldteta} yield
\begin{eqnarray}
\begin{aligned}
 \|\etab_u^n\|_{\Omega^n}^2 + \sum_{k=1}^n &\left(\|\etab_u^k - \etab_u^{k-1}\|_{\Omega^k} + \Delta t\left( ||| \etab_u^k|||_{h,k}^2 + h^2\|\nabla \eta_p^k\|_{\Omega^k}^2\right)\right)\\
&\leq ch^{2m} \left( \|\u\|_{\infty,m+1}^2 + \|\partial_t \u\|_{\infty,m}^2
  + \|p\|_{\infty,m}^2\right).
 \end{aligned}\label{interpol4error}
 \end{eqnarray}
{\sfrei Addition of \eqref{AfterGronwall} and \eqref{interpol4error} proves the statement.}
\end{proof}


\begin{remark}{(Optimality)}
 The energy norm estimate is optimal 
 under the inverse CFL condition $\Delta t \geq ch^2$. This condition is needed to control 
 the pressure error $h\|\nabla \xi_{h,p}^n\|_{\Omega^n}$ using Lemma~\ref{lem.pressstab}, see Corollary~\ref{cor.press}.
{\sfrei If the Brezzi-Pitk\"aranta stabilisation would be used instead of the CIP pressure stabilisation, this term would be controlled by the pressure stabilisation in Theorem~\ref{theo.stab}, as 
$h\|\nabla \xi_{h,p}^n\|_{\Omega^n} = s_h^n(\xi_{h,p}^n,\xi_{h,p}^n)^{1/2}$. Hence, an unconditional error estimate of first order in space would result.} 
\end{remark}

\begin{remark}{(BDF(2))}\label{rem.energys}
 For $s=2$ we obtain a similar result under the stronger condition $\Delta t \geq ch$ 
 This is needed 
 to get control over $h\|\nabla \xi_{h,p}^k\|_{\Omega^k}$, see Corollary~\ref{cor.press} \eqref{stabslarger1}. Under this assumption, we can show the following
 result for $n\geq 2$:
 \begin{align*}
 \|\e_u^n\|_{\Omega^n}^2 &+\Delta t \sum_{k=1}^n \left( ||| \e_u^k|||_{h,k}^2 + h^2\|\nabla e_p^k\|_{\Omega^k}^2\right)\\
 &\leq c\exp({\sfrei c_1(w_{\max})} t_n) \Big( {\sfrei \Delta t^{4} \|\u\|_{H^3(I,L^2(\Omega^n))}^2}^2
 +h^{2m} \left( \|\u\|_{\infty,m+1}^2 + \|\partial_t \u\|_{\infty,m}^2
  + \|p\|_{\infty,m}^2\right) \Big)\\
  &\qquad\qquad{ +c\left(\|\e_u^1\|_{\Omega^1}^2 + \Delta t |||\e_u^1|||_{h,1}^2\right).}
 \end{align*}
 which is of second order in time $\Delta t$, if we 
 assume that the initial error is bounded by
 \begin{eqnarray}
 \begin{aligned}\label{eu1}
  \|\e_u^1\|_{\Omega^1}^2 &+  \Delta t |||\e_u^1|||_{h,1}^2\\ 
  &\leq c\left({\sfrei \Delta t^4 \|\u \|_{H^3(I,L^2(\Omega^n))}^2}
 +h^{2m} \left( \|\u\|_{\infty,m+1,I_1}^2 + \|\partial_t \u\|_{\infty,m,I_1}^2
  + \|p\|_{H^{m}(\Omega^1)}^2\right) \right).
 \end{aligned}
 \end{eqnarray}
 The initialisation will be discussed in the following remark.
 The main modifications in the proof concern the approximation of the time derivative in Lemmas~\ref{lem.consistency} and~\ref{lem.interpol}. In \eqref{consTime1} we estimate
 \begin{align*}
  \left(D_t^{(2)} \u^n -\partial_t \u(t_n), \v_h\right)_{\Omega^n} \leq c {\sfrei \Delta t^{3/2} \|\partial_t^3 \u\|_{Q^n}} \|\v_h\|_{\Omega^n},
 \end{align*}
see \cite{HairerNorsettWanner1991, BurmanFernandez2008}. In order to estimate the analogue of \eqref{interpoldteta}, we use
\begin{align}\label{Dtuformula}
 D_t^{(2)} \u_h^n = \frac{3}{2} D_t^{(1)} \u_h^n + \frac{1}{2} D_t^{(1)} \u_h^{n-1}.
\end{align}
Then the argumentation used in \eqref{interpoldteta} can be applied to both terms on the right-hand side of \eqref{Dtuformula}.
\end{remark}

 {\sfrei
\begin{remark}{(Initialisation of BDF(2))}\label{rem.initBDF2}
To initialise the BDF(2) scheme, the function $\u_h^1$ needs to be computed with sufficient accuracy. The simplest possibility is to use one BDF(1) step by solving 
 \begin{eqnarray*}
 \begin{aligned}
\frac{1}{\Delta t}  &\left(\u_h^1, \v_h^1\right)_{\Omega^k} + {\cal A}_h^1(\u_h^1,p_h^1;\v_h^1,q_h^1) 
  =  \frac{1}{\Delta t}\left(E\u^0, \v_h^1\right)_{\Omega^1} + (\f,\v_h^1)_{\Omega^1} \quad \forall \v_h^1 \in {\cal V}_h^1, q_h^1 \in {\cal L}_h^1 
  \end{aligned}
  \end{eqnarray*}
  for $(\u_h^1, p_h^1) \in ({\cal V}_h^1 \times {\cal L}_{h,0}^1)$.
 Similar to the proof of Theorem~\ref{theo.energyerror}, the error after one BDF(1) step can be estimated by
  \begin{eqnarray*}
   \begin{aligned}
  \|\e_u^1\|_{\Omega^1}^2 +  \Delta t |||\e_u^1|||_{h,1}^2
  &\leq c\left({\sfrei \Delta t^3 \|\partial_t^2 \u \|_{Q^1}^2}
 +\Delta t\, h^{2m} \left( \|\u\|_{\infty,m+1,I_1}^2 + \|\partial_t \u\|_{\infty,m,I_1}^2
  + \|p\|_{H^{m}(\Omega^1)}^2\right) \right)
  \\
  &\leq c\left({\sfrei \Delta t^4 \|\u \|_{H^3(I,L^2(\Omega^n))}^2}
 +\Delta t\, h^{2m} \left( \|\u\|_{\infty,m+1,I_1}^2 + \|\partial_t \u\|_{\infty,m,I_1}^2
  + \|p\|_{H^{m}(\Omega^1)}^2\right) \right),
 \end{aligned}
 \end{eqnarray*}
 where in the last step a Sobolev inequality has been applied in time to show $\|\partial_t^2 \u \|_{Q^1}^2 \leq \Delta t \|\partial_t^2 \u \|_{\infty,0,I_1}^2 \leq c\Delta t \|\u \|_{H^3(I,L^2(\Omega^n))}^2$.

\end{remark}}

\subsubsection{$L^2(L^2)$-norm error of pressure}
\label{sec.L2L2press}

The energy estimate in Theorem~\ref{theo.energyerror} includes an optimal bound for the $H^1$-norm of the pressure. To show an optimal bound in the $L^2$-norm seems to be non-trivial, due to the fact
that $\u_h^{n-1}$ is not discrete divergence-free with respect to $\Omega^n$ and ${\cal V}_h^n$, see the discussion in Section~\ref{sec.pressstab}. 
We show here only a sub-optimal bound for $s=1$.
An optimal estimate is subject to future work.

\begin{lemma}\label{cor.pressest}
Under the assumptions of Theorem~\ref{theo.energyerror} it holds for $s=1$ 
 \begin{align*}
 \left(\Delta t \sum_{k=1}^n \|e_p^k\|_{\Omega^k}^2\right)^{1/2} 
 &\leq c\exp({\sfrei c_1(w_{\max})} t_n) \Big( \Delta t^{1/2} \|\partial_t^2 \u\|_{Q} +\frac{h^{m}}{\Delta t^{1/2}} \left( \|\u\|_{\infty,m+1} + \|\partial_t \u\|_{\infty,m}
  + \|p\|_{\infty,m}\right) \Big),
\end{align*}
where $\e_u^0:=0$. 
\end{lemma}
\begin{proof}
 We use the modified inf-sup condition for the discrete part $\xi_{h,p}^n = i_h^n p^n - p_h^n$ and standard interpolation estimates
 \begin{eqnarray}
  \begin{aligned}\label{applinfsup}
  \beta \|\xi_{h,p}^n\|_{\Omega^n} &\leq \sup_{\v_h^n \in {\cal V}_h^n} \frac{({\rm div}\, \v_h^n, \xi_{h,p}^n)_{\Omega^n} - (\v_h^n \cdot n, \xi_{h,p}^n)_{\partial\Omega^n}}{|||\v_h^n|||_{h,n}} + h\|\nabla \xi_{h,p}^n\|_{\Omega^n}\\
  &\leq  \sup_{\v_h^n \in {\cal V}_h^n} \frac{({\rm div}\, \v_h^n, e_p^n)_{\Omega^n} - (\v_h^n \cdot n, e_p^n)_{\partial\Omega^n}}{|||\v_h^n|||_{h,n}} 
  + \sup_{\v_h^n \in {\cal V}_h^n} \frac{({\rm div}\, \v_h^n, \eta_p^n)_{\Omega^n} - (\v_h^n \cdot n, \eta_p^n)_{\partial\Omega^n}}{|||\v_h^n|||_{h,n}}\\
  &\qquad\quad+ h\left( \|\nabla e_p^n\|_{\Omega^n} + \|\nabla \eta_{p}^n\|_{\Omega^n}\right)\\
  &\leq  \sup_{\v_h^n \in {\cal V}_h^n} \frac{({\rm div}\, \v_h^n, e_p^n)_{\Omega^n} - (\v_h^n \cdot n, e_p^n)_{\partial\Omega^n}}{|||\v_h^n|||_{h,n}} +  h\|\nabla e_p^n\|_{\Omega^n} + ch^{m} \|p^n\|_{H^{m}(\Omega^n)}.
 \end{aligned}
 \end{eqnarray}
 The second term on the right-hand side is bounded by the energy estimate. For the first term, we use Galerkin orthogonality~\eqref{Galerkin}, {\sfrei followed by Cauchy-Schwarz and Poincar\'e inequalities}
 \begin{align*}
 ({\rm div}\, \v_h^n, e_p^n)_{\Omega^n} - (\v_h^n \cdot \n, e_p^n)_{\partial\Omega^n}
  &=-(D_t^{(1)} \e_u^n, \v_h)_{\Omega^n}  - ({\cal A}_S^n+a_D^n)(\e_u^n,0;\v_h,0) - \gamma_g g_h^n(\e_u^n, \v_h) \\
 &\qquad\quad+ \gamma_g g_h^n(\u^n, \v_h) + (D_t^{(1)} \u^n -\partial_t \u(t_n), \v_h)_{\Omega^n}\\
 &\hspace{-1.3cm}\leq c\left\{ \|D_t^{(1)} \e_u^n\|_{\Omega^n} + |||\e_u^n|||_{h,n} + h^{m} \|\u^n\|_{H^{m+1}(\Omega^n)} 
 + \Delta t \|\partial_t^2 \u^n\|_{\Omega^n}\right\} |||\v_h|||_{h,n}.
 \end{align*}
 {\sfrei After summation in \eqref{applinfsup}, we obtain}
 \begin{eqnarray}
 \begin{aligned}\label{summeddt}
  \Delta t \sum_{k=1}^n \|\xi_{h,p}^k\|_{\Omega^k}^2 &\leq c \sum_{k=1}^n \Big\{ \frac{1}{\Delta t} \|\e_u^k-\e_u^{k-1}\|_{\Omega^k}^2 + \Delta t \left(|||\e_u^k|||_{h,k}^2 + h^2\|\nabla e_p^k\|_{\Omega^k}^2\right)\\
   &\qquad+ \Delta t h^{2m} \left( \|\u^k\|_{H^{m+1}(\Omega^k)}^2 + \|p^k\|_{H^{m}(\Omega^k)}^2\right)
    + \Delta t^3 \|\partial_t^2 \u^k\|_{\Omega^k}^2\Big\}.
 \end{aligned}
 \end{eqnarray}
{\sfrei Using the standard interpolation estimate
{$\displaystyle
\|\eta_{p}^k\|_{\Omega^k}^2 \leq ch^{2m} \|p^k\|_{H^{m}(\Omega^k)}^2$}, we see that \eqref{summeddt} holds for $\xi_{h,p}^k$ replaced by $e_p^k$.} 
 Finally, Theorem~\ref{theo.energyerror} yields the statement. Unfortunately, the factor $\frac{1}{\Delta t}$ in front of the first term on the right-hand side of \eqref{summeddt} leads to a loss of $\Delta t^{-1/2}$ in the final estimate.
\end{proof}

\begin{remark}{(BDF(2))}
For $s=2$ we can only control $\Delta t^3 \|D_t^{(2)} \e_u^n\|_{\Omega^n}^2 = \frac{\Delta t}{2} \|3 \e_u^n - 4 \e_u^{n-1} +\e_u^{n-2}\|_{\Omega^n}^2$ (compared to $\Delta t^2 \|D_t^{(1)} \e_u^n\|_{\Omega^n}^2$ for 
$s=1$), which leads to a further loss of $\Delta t^{-1}$ in the above estimate:
 \begin{align*}
 \left(\Delta t \sum_{k=1}^n \|e_p^k\|_{\Omega^k}^2\right)^{1/2} 
 &\leq c\exp({\sfrei c_1(w_{\max})} t_n) \Big( \Delta t \|\partial_t^2 \u\|_{Q} +\frac{h^{m}}{\Delta t} \left( \|\u\|_{\infty,m+1} + \|\partial_t \u\|_{\infty,m}
  + \|p\|_{\infty,m}\right) \Big).
\end{align*}
 
\end{remark}

\begin{remark}
The estimate in Lemma~\ref{cor.pressest} is balanced, if we choose $\Delta t \sim h^{m}$, which yields a convergence order of {\sfrei ${\cal O}(\Delta t^{1/2}) = {\cal O}(h^{m/2})$. This means that the convergence order is reduced by ${\cal O}(h^{m/2})$ compared to the situation on a fixed domain $\Omega(t)=\Omega$.} 
For BDF(2) the estimate is balanced for $\Delta t^2 \sim h^{m}$ and we obtain a convergence order of {\sfrei ${\cal O}(\Delta t) = {\cal O}(h^{m/2})$}.
The inverse CFL conditions in Theorem~\ref{theo.energyerror}
and Remark~\ref{rem.energys} are automatically fulfilled for these choices, if $m\geq 2$ or $m=s=1$. 
\end{remark}

\subsection{$L^2(L^2)$-norm error of velocity}
\label{sec.errorL2}

To obtain an optimal bound for the velocity error in the $L^2$-norm, we introduce a dual problem. The argumentation of 
Burman \& Fern\'andez~\cite{BurmanFernandez2008}, that does not require a dual problem, but is based on a Stokes projection $P_h(\u,p)$
of the continuous solution, can not be transferred in a straight-forward way to the case of moving domains, 
as it requires an estimate for the time derivative $\partial_t (u -P_h^u u)$. Time derivative and Stokes projection do, however, not commute in the 
case of moving domains, as $P_h^u \u(t)$ depends on the domain $\Omega(t)$. For this reason an estimate for the time derivative is non-trivial.

We focus again on the case $s=1$ first and remark on how to transfer the argumentation to the case $s>1$ afterwards. {\sfrei The argumentation will be based on 
a semi-discretised (in time) dual problem. Before we introduce the dual problem, let us note that the semi-discretised primal problem is given by: 
\textit{Find $(\u^k, p^k)_{k=1}^{n}$ with $\u^k\in H^1_0(\Omega^k), p^k \in L^2_0(\Omega^k)$ such that}
\begin{eqnarray}
\begin{aligned}
\sum_{k=1}^{n} \big\{ (\u^k-E^{k-1} \u^{k-1}, \psib_u^k)_{\Omega^k} + \Delta t {\cal A}_S^k(\u^k, &p^k; \psib_u^k, \psi_p^k) \big\}
 + (E^0 \u^0, \psib_u^1)_{\Omega^1} 
  \\
 =\left(E^0 \u^0, \psib_u^1\right)_{\Omega^1} +  \Delta t \sum_{k=1}^n \left(\f, \psib_u^k\right)_{\Omega^k}   &\;\;\forall \psib_u^k \in H^1_0(\Omega^k)^d, \,\psi_p^k\in L^2(\Omega^k),\,\,k=1,\dots, n,
\end{aligned}
\end{eqnarray}
{\sfrei where $E^k$ denotes the smooth extension operator to $\Omega_\delta^k$ introduced in Section~\ref{sec.ext}.}
 
The corresponding semi-discretised dual problem, which will be needed in the following, reads:} \textit{Find $(\z_u^k, z_p^k)_{k=1}^{n}$ with $\z_u^k\in H^1_0(\Omega^k), z_p^k \in L^2_0(\Omega^k)$ such that}
\begin{eqnarray}
\begin{aligned}\label{semidiscDual}
 \Delta t \sum_{k=1}^n \left(\e_u^k, \phib_u^k\right)_{\Omega^k} = \sum_{k=1}^{n} &\left\{ (\phib_u^k-{\sfrei E^{k-1} \phib_u^{k-1}}, \z_u^k)_{\Omega^k} + \Delta t {\cal A}_S^k(\phib_u^k, \phi_p^k; \z_u^k, z_p^k) \right\}
 + ({\sfrei E^0} \phib_u^0, \z_u^1)_{\Omega^1} \\
 &\qquad\qquad\qquad\qquad \forall \phib_u^k \in H^1_0(\Omega^k)^d, \,\phi_p^k\in {\sfrei L_0^2(\Omega^k)},\quad k=1,\dots, n.
\end{aligned}
\end{eqnarray}
Note that the Dirichlet conditions are imposed strongly in this formulation and the bilinear form ${\cal A}_S^k$ does not include the Nitsche terms.

We start by showing the well-posedness of the problem \eqref{semidiscDual}.

\begin{lemma}\label{lem.dualEx}
 Let $s=1$, $\e_u^k \in L^2(\Omega^k)$ for $k =1,\dots,n$ and assume Assumption~\ref{ass.T}. The semi-discrete dual problem \eqref{semidiscDual} defines unique solutions $(\z_u^k, z_p^k)_{k=1}^n$ with regularity $\z_u^k\in H^2(\Omega^k), z_p^k\in H^1(\Omega^k)$. Moreover,  
 the following regularity estimates are valid, where $S_\delta^k := \Omega_\delta^k \setminus \Omega^k$
 \begin{align}\label{dualreg1}
  \|\z_u^k\|_{H^2(\Omega^k)} + \|z_p^k\|_{H^1(\Omega^k)} &\leq c  \left(\frac{1}{\Delta t}\left(\|\z_u^{k+1}-\z_u^k\|_{\Omega^{k}} + \|\z_u^{k+1}\|_{S_\delta^{k}}\right) + \|\e_u^k\|_{\Omega^k} \right) 
\quad \text{for } k<n,\\
  \|\z_u^n\|_{H^2(\Omega^n)} + \|z_p^n\|_{H^1(\Omega^n)} &\leq c \left(\frac{1}{\Delta t}\|\z_u^n\|_{\Omega^{n}} + \|\e_u^n\|_{\Omega^n} \right).\label{dualreg2}
 \end{align}
\end{lemma}
\begin{proof}
By testing \eqref{semidiscDual} with {\sfrei $\tilde \phib_u^l=\delta_{kl} \,\phib_u^k, \tilde \phi_p^l = \delta_{kl} \phi_p^k, l=1,...,n$,} 
where $\delta_{kl}$ is the Kronecker delta, we observe
that {\sfrei the system splits into $n$ separate time steps, where each step corresponds to a stationary Stokes system with an additional $L^2$-term coming from the discretisation of the time derivative. For $k<n$ we have
\begin{eqnarray}
\begin{aligned}\label{statStokes1}
 \frac{1}{\Delta t} (\phib_u^k, \z_u^k)_{\Omega^{k}} + {\cal A}_S^k(\phib_u^k, \phi_p^k; \z_u^k, z_p^k) = \frac{1}{\Delta t} (E^k \phib_u^k, \z_u^{k+1})_{\Omega^{k+1}} &+ \left(\e_u^k, \phib_u^k\right)_{\Omega^k}
 \\ &\forall \phib_u^k \in H^1_0(\Omega^k)^d, \phi_p^k\in L^2_0(\Omega^k),
\end{aligned}
\end{eqnarray}
 and for $k=n$
\begin{align}\label{statStokes2}
 \frac{1}{\Delta t}\left( \phib_u^n, \z_u^n\right)_{\Omega^n} + {\cal A}_S^n(\phib_u^n, \phi_p^n; \z_u^n, z_p^n) &=
  \left(\e_u^n, \phib_u^n\right)_{\Omega^n}   
  \quad\forall \phib_u^n \in H^1_0(\Omega^n)^d, \phi_p^k\in L^2_0(\Omega^n).
 \end{align}
As the corresponding reduced problems are coercive in the velocity space ${\cal V}_0(t_k)$ (cf. Section~\ref{sec.wellp}), existence and uniqueness of solutions $\z_u^k\in H^1_0(\Omega^k), z_p^k\in L^2_0(\Omega^k)$ follow inductively by standard arguments for $k=n,...,1$, see e.g.\,Temam~\cite{Temam2000}, Section I.2.
 
 To show the regularity estimates \eqref{dualreg1} and \eqref{dualreg2}, let us re-formulate the problems \eqref{statStokes1} and \eqref{statStokes2} in the following way: For $k<n$ we have
}
\begin{eqnarray}
\begin{aligned}\label{statStokes}
 {\cal A}_S^k(\phib_u^k, \phi_p^k; \z_u^k, z_p^k) &= \underbrace{\frac{1}{\Delta t}\left(  (E^k \phib_u^k, \z_u^{k+1})_{\Omega^{k+1}} -(\phib_u^k, \z_u^k)_{\Omega^{k}}\right) + \left(\e_u^k, \phib_u^k\right)_{\Omega^k}}_{=: F_k(\phib_u^k)}\\
 &\hspace{5.5cm} \forall \phib_u^k \in H^1_0(\Omega^k)^d, \phi_p^k\in L^2_0(\Omega^k), 
\end{aligned}
\end{eqnarray}
and for $k=n$
\begin{align*}
 {\cal A}_S^n(\phib_u^n, \phi_p^n; \z_u^n, z_p^n) &=
  \underbrace{\left(\e_u^n, \phib_u^n\right)_{\Omega^n}  
  - \frac{1}{\Delta t}\left( \phib_u^n, \z_u^n\right)_{\Omega^n}}_{=:F_n(\phib_u^n)} 
  \quad\forall \phib_u^n \in H^1_0(\Omega^n)^d, \phi_p^k\in L^2_0(\Omega^n).
 \end{align*}

If we can prove that $F_k$ lies in the dual space $[L^2(\Omega^k)^d]^*$, Proposition I.2.2 in Temam's book~\cite{Temam2000} guarantees 
the regularity estimate
\begin{align}\label{TemamH2}
 \|\z_u^k\|_{H^2(\Omega^k)} + \|z_p^k\|_{H^1(\Omega^k)} \leq c\sup_{\phib_u^k \in L^2(\Omega^k)} \frac{F_k(\phib_u^k)}{\|\phib_u^k\|_{\Omega^k}}.
\end{align}
We need to show that the right-hand side is bounded. Splitting the first integral on the right-hand side into an integral 
over $\Omega^k$ and $S_{\delta}^k$, we have for $k<n$
\begin{align*}
 F_k(\phib_u^k) 
 &\leq \frac{1}{\Delta t}\left(  \|\phib_u^k\|_{\Omega^k} \|\z_u^{k+1}-\z_u^k\|_{\Omega^{k}} +\|E^k\phib_u^k\|_{S_\delta^k} \|E^k \z_u^{k+1}\|_{S_\delta^k}\right) + \|\e_u^k\|_{\Omega^k} \|\phib_u^k\|_{\Omega^k}.
\end{align*}
and thus,
\begin{align}\label{Dualk}
 F_k(\phib_u^k) \leq c \left(\frac{1}{\Delta t}\left(\|\z_u^{k+1}-\z_u^k\|_{\Omega^{k}} + \|E^k \z_u^{k+1}\|_{S_\delta^{k}}\right) + \|\e_u^k\|_{\Omega^k} \right) \|\phib_u^k\|_{\Omega^k}.
\end{align}
For $k=n$, we obtain
\begin{align}\label{Dualkn}
 F_n(\phib_u^n) \leq c \left(\frac{1}{\Delta t}\|\z_u^n\|_{\Omega^{n}} + \|\e_u^n\|_{\Omega^n} \right) \|\phib_u^n\|_{\Omega^n}.
\end{align}
The boundedness of $F_k$ follows by induction for $k=n,\dots,1$ and by using the stability of the extension operator $E^k$.
Combination of \eqref{TemamH2} and \eqref{Dualk}, resp. \eqref{Dualkn}, yield the regularity estimates \eqref{statStokes1} and \eqref{statStokes2}.
\end{proof}

Next, we derive a stability estimate for the semi-discretised dual problem \eqref{semidiscDual}. We remark that a stability estimate for the continuous dual problem, including the first time derivative $\partial_t z$,
could be obtained as well. This is however not enough to bound the consistency error of the time derivative in a sufficient way for an 
optimal $L^2$-norm error estimate. 

\begin{lemma}\label{lem.StabDual}
{\sfrei Let the assumptions made in Section~\ref{sec.disc} be valid, including Assumption~\ref{ass.T}. For sufficiently small $\Delta t<\xi$, where $\xi$ depends only on $c_\delta$, $w_{\max}$ and the domains $\Omega^k, k=1,...,n$,} the solution $(\z_u^k, z_p^k)_{k=1}^{n}$ 
to the semi-discretised dual problem \eqref{semidiscDual} for $s=1$ fulfils the stability estimate
  \begin{align*}
 \|\nabla \z_u^1\|_{\Omega^1}^2  + \frac{1}{\Delta t} \|&\nabla \z_u^n\|_{\Omega^n}^2 
 + \sum_{k=1}^{n-1} \left\{ \|\nabla (\z_u^k-\z_u^{k+1})\|_{\Omega^k}^2 + \frac{1}{\Delta t} \|\z_u^k-\z_u^{k+1}\|_{\Omega^k}^2\right\} \\
 &\qquad+ \Delta t\sum_{k=1}^n \left\{  \|\z_u^k\|_{H^2(\Omega^k)}^2  + \|z_p^k\|_{H^1(\Omega^k)}^2\right\}  \leq c 
 {\sfrei w_{\max}^2} \Delta t \sum_{k=1}^n \|\e_u^k\|_{\Omega^k}^2.
\end{align*}
\end{lemma}

\begin{proof}
We show a stability estimate for the first derivatives $\nabla \z_u^k$ first. Diagonal testing in \eqref{semidiscDual} with $\phib_u^k = \z_u^k, \phi_p^k = z_p^k$ results in 
\begin{align*}
 \sum_{k=1}^{n} &\left\{ (\z_u^k- \z_u^{k-1}, \z_u^k)_{\Omega^{k}} +  \Delta t \|\nabla \z_u^k\|_{\Omega^{k}}^2\right\} 
 + (\z_u^0, \z_u^1)_{\Omega^1} = \Delta t \sum_{k=1}^{n} \left(\e_u^k, \z_u^k\right)_{\Omega^k},
\end{align*}
or equivalently
\begin{align}\label{dualcite}
 \sum_{k=1}^{n-1} &\left\{ \|\z_u^k\|_{\Omega^{k}}^2 - (\z_u^k, \z_u^{k+1})_{\Omega^{k+1}} +  \Delta t \|\nabla \z_u^k\|_{\Omega^{k}}^2\right\} 
 + \|\z_u^n\|_{\Omega^n}^2 +  \Delta t \|\nabla \z_u^n\|_{\Omega^n}^2 
 = \Delta t \sum_{k=1}^{n} \left(\e_u^k, \z_u^k\right)_{\Omega^k}.
\end{align}
As $\z_u^k$ vanishes on $\partial\Omega^k$, a Poincar\'e-like estimate gives in combination {\sfrei with \eqref{defDelta}} and the stability of the extension operator
\begin{align}\label{dualSdelta}
 \|\z_u^k\|_{S_\delta^k} \leq {\sfrei c_p \left( \delta^{1/2} \|\z_u^k\|_{\partial \Omega^k} + \delta \|\nabla \z_u^k\|_{S_\delta^k} \right) \leq c_p c_\delta w_{\rm max} \Delta t \|\nabla \z_u^k\|_{\Omega^k},}
\end{align}
{\sfrei where $c_p$ denotes a constant depending on the domain $\Omega^k$ and $c_\delta>1$ is the constant in \eqref{defDelta}.
Using Young's inequality, this implies for $\Delta t
\leq (4c_p^2 c_\delta^2 w_{\max}^2)^{-1}$
\begin{align*}
\|\z_u^k\|_{S_\delta^k} \|\z_u^{k+1}\|_{S_\delta^k} \, &\leq \, 
c_p^2 c_\delta^2 w_{\max}^2 \Delta t^2 \left(\|\nabla \z_u^k\|_{\Omega^k}^2 + \|\nabla \z_u^{k+1}\|_{\Omega^{k+1}}^2\right)\\
&\leq \frac{\Delta t}{4} \left(\|\nabla \z_u^k\|_{\Omega^k}^2 + \|\nabla \z_u^{k+1}\|_{\Omega^{k+1}}^2\right).
\end{align*}
}
We obtain
\begin{eqnarray}\label{Sdeltaarg}
\begin{aligned}
  \|\z_u^k\|_{\Omega^{k}}^2 - &(\z_u^k, \z_u^{k+1})_{\Omega^{k+1}} 
  \geq  (\z_u^k, \z_u^k - \z_u^{k+1})_{\Omega^k} -  \|\z_u^k\|_{S_\delta^k} \|\z_u^{k+1}\|_{S_\delta^k} \\
  &\geq \frac{1}{2} \left( \|\z_u^k\|_{\Omega^k}^2 + \|\z_u^k - \z_u^{k+1} \|_{\Omega^k}^2 - \|\z_u^{k+1} \|_{\Omega^{k+1}}^2\right)
  - {\sfrei \frac{\Delta t}{4}} \left(\|\nabla \z_u^k\|_{\Omega^k}^2 + \|\nabla \z_u^{k+1}\|_{\Omega^{k+1}}^2\right). 
\end{aligned}
\end{eqnarray}
{\sfrei For the right-hand side in \eqref{dualcite}, we apply 
the Cauchy-Schwarz, a Poincar\'e and Young's inequality to get
\begin{align}\label{dualrhs}
\Delta t \sum_{k=1}^{n} \left(\e_u^k, \z_u^k\right)_{\Omega^k} \,\leq \,\sum_{k=1}^n \frac{\Delta t}{4}
 \|\nabla \z_u^k\|_{\Omega^k}^2 + c\Delta t \|\e_u^k\|_{\Omega^k}^2.
\end{align}}
{\sfrei Using \eqref{Sdeltaarg} and \eqref{dualrhs}, \eqref{dualcite} writes } 
\begin{align}\label{H1stabdual}
 \|\z_u^1\|_{\Omega^1}^2 + \|\z_u^n\|_{\Omega^n}^2 +  \sum_{k=1}^{n-1} \|\z_u^k - \z_u^{k+1} \|_{\Omega^k}^2 + \sum_{k=1}^{n}  \Delta t \|\nabla \z_u^k\|_{\Omega^{k}}^2
  \leq {\sfrei c \Delta t \sum_{k=1}^n \|\e_u^k\|_{\Omega^k}^2.}
\end{align}

Next, we use the regularity estimates in Lemma~\ref{lem.dualEx} to get a bound for the second derivatives of $\z_u^k$. For $k=n$ we have
 \begin{align*}
  \|\z_u^n\|_{H^2(\Omega^n)} + \|z_p^n\|_{H^1(\Omega^n)} \leq c \left(\frac{1}{\Delta t}\|\z_u^n\|_{\Omega^{n}} + \|\e_u^n\|_{\Omega^n} \right).
 \end{align*}
{\sfrei For $k<n$ Lemma~\ref{lem.dualEx} gives us
  \begin{align*}
   \|\z_u^k\|_{H^2(\Omega^k)} + \|z_p^k\|_{H^1(\Omega^k)} \leq c  \left(\frac{1}{\Delta t}\left(\|\z_u^{k+1}-\z_u^k\|_{\Omega^{k}} + \|\z_u^{k+1}\|_{S_\delta^{k}}\right) + \|\e_u^k\|_{\Omega^k} \right).
  \end{align*}
We estimate the term on $S_\delta^k$ by using a Poincar\'e-type inequality with a domain-dependent constant $c_p>0$ as in \eqref{dualSdelta}, followed by \eqref{deltaeps} for $\epsilon=1$ and the stability of the extension
\begin{align*}
 \|\z_u^{k+1}\|_{S_\delta^k} &\leq c_p\delta \|\nabla \z_u^{k+1}\|_{S_\delta^k} \leq 
c c_p\delta^{3/2} \|\z_u^{k+1}\|_{H^2(\Omega_\delta^k)} \leq
 c c_p\delta^{3/2} \|\z_u^{k+1}\|_{H^2(\Omega^{k+1})}.
\end{align*}
Using \eqref{defDelta} we get for $\Delta t<(2 c^2 c_p^2 c_\delta^3 w_{\max}^3)^{-1}$
\begin{align}
\|\z_u^{k+1}\|_{S_\delta^k}\leq  c c_p (c_\delta w_{\max}\Delta t)^{3/2} \|\z_u^{k+1}\|_{H^2(\Omega^{k+1})} \leq 
 \frac{\Delta t}{2} \|\z_u^{k+1}\|_{H^2(\Omega^{k+1})},
\end{align} 
and hence
 \begin{align*}
  \|\z_u^k\|_{H^2(\Omega^k)} + \|z_p^k\|_{H^1(\Omega^k)} \leq \frac{c}{\Delta t}\|\z_u^{k+1}-\z_u^k\|_{\Omega^{k}} + \frac{1}{2}\|\z_u^{k+1}\|_{H^2(\Omega^{k+1})} + c\|\e_u^k\|_{\Omega^k}.
 \end{align*}
 Summation over $k=1,...,n$ results in}
\begin{eqnarray}
\begin{aligned}\label{sumzukH2}
 \Delta t \sum_{k=1}^n &\left\{\|\z_u^k\|_{H^2(\Omega^k)}^2 + \|z_p^k\|_{H^1(\Omega^k)}^2\right\} \\
 &\qquad\qquad\leq \frac{c_0}{\Delta t} \left( \|\z_u^n\|_{\Omega^{n}}^2 + \sum_{k=1}^{n-1} \|\z_u^{k+1}-\z_u^k\|_{\Omega^{k}}^2\right) 
 + c \Delta t\sum_{k=1}^n \|\e_u^k\|_{\Omega^k},
\end{aligned}
\end{eqnarray}
where $c_0$ denotes a constant.
It remains to derive a bound for the discrete time derivative on the right-hand side. Therefore, note that for $k<n$ 
we can write \eqref{statStokes} equivalently by using the density of $H^1(\Omega^k)$ in $L^2(\Omega^k)$ as
\begin{eqnarray}
\begin{aligned}\label{DualStab1}
 - (\Delta \z_u^k, &\phib_u^k)_{\Omega^k} + (\nabla z_p^k, \phib_u^k)_{\Omega^k} - (\z_u^k, \nabla \phi_p^k)_{\Omega^k}\\
 &= \frac{1}{\Delta t}\left(  (E\phib_u^k, \z_u^{k+1})_{\Omega^{k+1}} -(\phib_u^k, \z_u^k)_{\Omega^{k}}\right) + \left(\e_u^k, \phib_u^k\right)_{\Omega^k}
 \quad \forall \phib_u^k\in L^2(\Omega^k)^d, \phi_p^k\in H^1(\Omega^k).
\end{aligned}
\end{eqnarray}
For $k=n$ we have
\begin{eqnarray}
\begin{aligned}\label{Dualkeqn}
  - (\Delta \z_u^n, \phib_u^n)_{\Omega^n} + (\nabla z_p^n, &\phib_u^n)_{\Omega^n} - (\z_u^n, \nabla \phi_p^n)_{\Omega^n}\\
  &= \left(\e_u^n, \phib_u^n\right)_{\Omega^n}  - \frac{1}{\Delta t}\left( \phib_u^n, \z_u^n\right)_{\Omega^n}
  \quad\forall \phib_u^n\in L^2(\Omega^n)^d, \phi_p^n\in H^1(\Omega^n).
\end{aligned}
\end{eqnarray}
For $k<n$ we test \eqref{DualStab1} with $\phib_u^k=\z_u^k - \z_u^{k+1}, \phi_p^k=0$
\begin{eqnarray}
\begin{aligned}\label{StabDual2}
 - (\Delta \z_u^k, &\,\z_u^k - \z_u^{k+1})_{\Omega^k} + (\nabla z_p^k, \z_u^k - \z_u^{k+1})_{\Omega^k}
 + \frac{1}{\Delta t}\big(  \|\z_u^k - \z_u^{k+1}\|_{\Omega^{k+1}}^2 \\
 &+(\z_u^k - \z_u^{k+1}, \z_u^{k+1})_{\Omega^{k}\setminus\Omega^{k+1}} - (\z_u^k - \z_u^{k+1}, \z_u^{k+1})_{\Omega^{k+1}\setminus\Omega^k}\big) 
 =\left(\e_u^k, \z_u^k - \z_u^{k+1}\right)_{\Omega^k}.
\end{aligned}
\end{eqnarray}

{\sfrei 
Using integration by parts and the fact that $\z_u^k|_{\partial\Omega^k} =0$, the first term in 
\eqref{StabDual2} writes
\begin{align*}
- (\Delta \z_u^k, &\,\z_u^k - \z_u^{k+1})_{\Omega^k}
=  \left(\nabla \z_u^k, \,\nabla (\z_u^k - \z_u^{k+1})\right)_{\Omega^k} + \left(\partial_n \z_u^k, \z_u^{k+1}\right)_{\partial\Omega^k}.
\end{align*}
For the second term in \eqref{StabDual2} we note that
$(\nabla z_p^k, \z_u^l)_{\Omega^l}=0$ for $l=k, k+1$
\begin{align*}
(\nabla z_p^k, \z_u^k - \z_u^{k+1})_{\Omega^k}
= -(\nabla z_p^k, \z_u^{k+1})_{\Omega^{k+1}\setminus\Omega^k} + (\nabla z_p^k, \z_u^{k+1})_{\Omega^{k}\setminus\Omega^{k+1}}
\geq -\|\nabla z_p^k\|_{S_\delta^k} \|\z_u^{k+1}\|_{S_\delta^k}.
\end{align*}
Using the Cauchy-Schwarz and Young's inequality, we obtain from \eqref{StabDual2}
}
\begin{eqnarray}
\begin{aligned}\label{StabDual3}
 \frac{1}{2\Delta t}&\left( \left\|\z_u^k - \z_u^{k+1}\right\|_{\Omega^k}^2 -\left\|\z_u^k - \z_u^{k+1}\right\|_{S_\delta^k}\left\|\z_u^{k+1}\right\|_{S_\delta^k}\right)
 + \left(\nabla \z_u^k, \nabla (\z_u^k-\z_u^{k+1})\right)_{\Omega^k}\\
 &\qquad\qquad\qquad\qquad-  (\partial_n \z_u^k, \z_u^{k+1})_{\partial\Omega^k} 
 -\|\nabla z_p^k\|_{S_\delta^k} \|\z_u^{k+1}\|_{S_\delta^k} \,\leq\, c\Delta t \|\e_u^k\|_{\Omega^k}^2.
\end{aligned}
\end{eqnarray}
{\sfrei To estimate the second term on the left-hand side, we
apply the triangle inequality and Young's inequality first to get
\begin{align}\label{startsecterm}
\frac{1}{\Delta t} \left\|\z_u^k - \z_u^{k+1}\right\|_{S_\delta^k}\left\|\z_u^{k+1}\right\|_{S_\delta^k} \leq \frac{c}{\Delta t} \left( \left\|\z_u^k\right\|_{S_\delta^k}^2 + \left\|\z_u^{k+1}\right\|_{S_\delta^k}^2\right)
\end{align}
Next, we use a Poincar\'e-type estimate with a constant $c_p>0$, see~\eqref{dualSdelta}, and the fact that $\z_u^l=0$ on $\partial\Omega^l$ for $l=k,k+1$
\begin{eqnarray}
\begin{aligned}\label{delta2arg}
\frac{c}{\Delta t}\left\|\z_u^l\right\|_{S_\delta^k}^2
\leq cc_p^2\frac{\delta}{\Delta t} \left(\left\|\z_u^l\right\|_{\partial\Omega}^2 + \delta \left\|\nabla \z_u^l\right\|_{S_\delta^k}^2\right)
= cc_p^2\frac{\delta^2}{\Delta t}  \left\|\nabla \z_u^l\right\|_{S_\delta^k}^2 
\end{aligned}
\end{eqnarray}
Using \eqref{deltaeps} followed by \eqref{defDelta}, we obtain further
\begin{eqnarray}\label{endsecterm}
\begin{aligned}
cc_p^2\frac{\delta^2}{\Delta t}  \left\|\nabla \z_u^l\right\|_{S_\delta^k}^2 
&\leq  
cc_p^2\frac{\delta^3}{\Delta t}  \left(\epsilon^{-1}\left\|\nabla \z_u^l\right\|_{\Omega^l}^2 + \epsilon\left\|\nabla^2 \z_u^l\right\|_{\Omega^l}^2\right) \\
&\leq c c_p^2 c_\delta^3 w_{\max}^3\Delta t^2  \left(\epsilon^{-1}\left\|\nabla \z_u^l\right\|_{\Omega^l}^2 + \epsilon\left\|\nabla^2 \z_u^l\right\|_{\Omega^l}^2\right).
\end{aligned}
\end{eqnarray}
For sufficiently small $\Delta t < \frac{1}{c c_p^2 c_\delta^3 w_{\max}^3}$ we obtain from \eqref{startsecterm}-\eqref{endsecterm}
\begin{eqnarray}
\begin{aligned}
\frac{1}{\Delta t} \left\|\z_u^k - \z_u^{k+1}\right\|_{S_\delta^k}\left\|\z_u^{k+1}\right\|_{S_\delta^k}
 &\leq \Delta t \sum_{l=k}^{k+1} \left(\epsilon^{-1}\left\|\nabla \z_u^l\right\|_{\Omega^l}^2 + \epsilon\left\|\nabla^2 \z_u^l\right\|_{\Omega^l}^2\right).
\end{aligned}
\end{eqnarray}}
For the third term in \eqref{StabDual3}, we use a telescope argument
\begin{align*}
  \left(\nabla \z_u^k, \nabla (\z_u^k-\z_u^{k+1})\right)_{\Omega^k} = \frac{1}{2} \left( \|\nabla \z_u^k\|_{\Omega^k}^2 + \|\nabla (\z_u^k- \z_u^{k+1})\|_{\Omega^k}^2 - \|\nabla \z_u^{k+1}\|_{\Omega^{k}}^2\right).
\end{align*}
To bring the last term to $\Omega^{k+1}$, we estimate using \eqref{deltaeps}
\begin{align*}
 \|\nabla \z_u^{k+1}\|_{\Omega^{k}}^2\leq \|\nabla \z_u^{k+1}\|_{\Omega^{k+1}}^2 + \|\nabla \z_u^{k+1}\|_{S_\delta^{k}}^2 \leq (1+c c_\delta^2 w_{\max}^2\epsilon^{-1}\Delta t) \|\nabla \z_u^{k+1}\|_{\Omega^{k+1}}^2 +\epsilon\Delta t \|\nabla^2 \z_u^{k+1}\|_{\Omega^{k+1}}^2
\end{align*}
{\sfrei For the boundary term in \eqref{StabDual3}, we use Green's theorem on $S_\delta^k$ 
\begin{align}\label{dnzukzuk}
  (\partial_n \z_u^k, \z_u^{k+1})_{\partial\Omega^k} &\leq -(\partial_n \z_u^k, \underbrace{\z_u^{k+1}}_{=0})_{\partial\Omega^{k+1}} + \|\nabla \z_u^k\|_{S_\delta^k} \|\nabla \z_u^{k+1}\|_{S_\delta^k} 
 + \|\Delta \z_u^k\|_{S_\delta^k} \|\z_u^{k+1}\|_{S_\delta^k}
\end{align}
For the second term on the right-hand side in \eqref{dnzukzuk} we use \eqref{deltaeps} twice, followed by \eqref{defDelta} and 
Young's inequality
\begin{align*}
\|\nabla \z_u^k\|_{S_\delta^k} &\|\nabla \z_u^{k+1}\|_{S_\delta^k}\\
&\leq 
 c\delta \left(\epsilon^{-1/2}\|\nabla \z_u^k\|_{\Omega^k} + \epsilon^{1/2}\|\nabla^2 \z_u^k\|_{\Omega^k}\right)  \left(\epsilon^{-1/2}\|\nabla \z_u^{k+1}\|_{\Omega^{k+1}} +\epsilon^{1/2} \|\nabla^2 \z_u^{k+1}\|_{\Omega^{k+1}}\right)\\
 &\leq  \Delta t \sum_{l=k}^{k+1}\left(c c_\delta^2 w_{\max}^2 \epsilon^{-1}\|\nabla \z_u^l\|_{\Omega^l}^2 + \epsilon\|\nabla^2 \z_u^l\|_{\Omega^l}^2\right).
 \end{align*} 
For the last term in \eqref{dnzukzuk} we obtain as in \eqref{delta2arg}
\begin{align}\label{dnzukzukfinal}
\|\Delta \z_u^k\|_{S_\delta^k} \|\z_u^{k+1}\|_{S_\delta^k}
\,\leq\, c\delta \|\Delta \z_u^k\|_{S_\delta^k} \|\nabla \z_u^{k+1}\|_{S_\delta^k}
\,\leq\, \epsilon\Delta t \|\Delta \z_u^k\|_{\Omega^k}^2 + c c_\delta^2 w_{\max}^2 \epsilon^{-1}\Delta t \|\nabla \z_u^{k+1}\|_{\Omega^{k+1}}^2.
\end{align}  
In the last inequality, we have used \eqref{defDelta} and Young's inequality.
Together, \eqref{dnzukzuk}-\eqref{dnzukzukfinal} yield the estimate 
\begin{align*}
(\partial_n \z_u^k, \z_u^{k+1})_{\partial\Omega^k} &\leq 
\Delta t \sum_{l=k}^{k+1}\left(c c_\delta^2 w_{\max}^2 \epsilon^{-1}\|\nabla \z_u^l\|_{\Omega^l}^2 + \epsilon\|\nabla^2 \z_u^l\|_{\Omega^l}^2\right).
\end{align*}}
To estimate the pressure term in \eqref{StabDual3}, we obtain as in \eqref{dnzukzukfinal}
\begin{align*}
 \|\nabla z_p^k\|_{S_\delta^k} \|\z_u^{k+1}\|_{S_\delta^k} 
 &\leq \Delta t \left(\epsilon \|\nabla z_p^k\|_{\Omega^k}^2 +c c_\delta^2 w_{\max}^2 \epsilon^{-1} \|\nabla \z_u^{k+1}\|_{\Omega^{k+1}}^2\right).
\end{align*}
To summarise we have shown that
\begin{eqnarray}
\begin{aligned}\label{zuk}
  \frac{1}{2\Delta t} &\left\|\z_u^k - \z_u^{k+1}\right\|_{\Omega^k}^2 
  +\frac{1}{2} \left( (1-c\epsilon^{-1} {\sfrei w_{\max}^2} \Delta t) \|\nabla \z_u^k\|_{\Omega^k}^2 + \|\nabla (\z_u^k- \z_u^{k+1})\|_{\Omega^k}^2 \right)\\  
  &\qquad\leq {\sfrei \frac{1}{2}} (1+c \epsilon^{-1}{\sfrei w_{\max}^2} \Delta t) \|\nabla \z_u^{k+1}\|_{\Omega^{k+1}}^2 \\
  &\qquad\qquad\qquad\qquad+ \epsilon\Delta t \left( \|\nabla z_p^k\|_{\Omega^k}^2 +  \|\nabla^2 \z_u^{k}\|_{\Omega^{k}}^2 + \|\nabla^2 \z_u^{k+1}\|_{\Omega^{k+1}}^2\right)
  + c\Delta t \|\e_u^k\|_{\Omega^k}^2.
\end{aligned}
\end{eqnarray}
For $k=n$ we obtain from \eqref{Dualkeqn} tested with $\phib_u^n=\z_u^n$ and $\phi_p^n=z_p^n$ that
\begin{align}\label{zun}
  \frac{1}{2\Delta t}& \left\|\z_u^n\right\|_{\Omega^n}^2 
  +\frac{1}{2} \|\nabla \z_u^n\|_{\Omega^n}^2
  \leq \frac{\Delta t}{2} \|\e_u^n\|_{\Omega^n}^2.
\end{align}

Summation in \eqref{zuk} over $k=1,\dots,n-1$ and addition of \eqref{zun} and \eqref{sumzukH2} multiplied by a factor of $3\epsilon$ yields for $\epsilon<\frac{1}{12c_0}$
\begin{eqnarray}
\begin{aligned}\label{DualVelRes}
 \|\nabla \z_u^1\|_{\Omega^1}^2 &+ \|\nabla \z_u^n\|_{\Omega^n}^2 + \frac{1}{\Delta t} \|\z_u^n\|_{\Omega^n}^2 
 + \sum_{k=1}^{n-1} \left\{ \frac{1}{\Delta t} \left\|\z_u^k - \z_u^{k+1}\right\|_{\Omega^k}^2 +  \|\nabla (\z_u^k-\z_u^{k+1})\|_{\Omega^k}^2  \right\}\\
 &+ \epsilon \Delta t\sum_{k=1}^n  \left\{  \|\z_u^k\|_{H^2(\Omega^k)}^2 + \|z_p^k\|_{H^1(\Omega^k)}^2 \right\}
 \leq c\Delta t \sum_{k=1}^n \left\{\|\e_u^k\|_{\Omega^k}^2 + {\sfrei w_{\max}^2}\|\nabla \z_u^k\|_{\Omega^k}^2\right\}.
\end{aligned}
\end{eqnarray}
{\sfrei Using \eqref{H1stabdual} we can estimate the last term by
\begin{align*}
w_{\max}^2 \Delta t \sum_{k=1}^n \|\nabla \z_u^k\|_{\Omega^k}^2
\leq c w_{\max}^2 \Delta t \sum_{k=1}^n \|\e_u^k\|_{\Omega^k}^2,
\end{align*}
which completes the proof.
}

\end{proof}

Now we are ready to prove an error estimate for the $L^2(L^2)$-norm of the velocities. First, we note that,
{\sfrei due to the regularity proven in Lemma~\ref{lem.dualEx}},
 the solution $(\z_u^k, z_p^k)_{k=1}^{n}$ of \eqref{semidiscDual} is
also the unique solution to the Nitsche formulation: \textit{Find $(\z_u^k, z_p^k)_{k=1}^{n}$, where $\z_u^k\in H^2(\Omega^k)^d, z_p^k \in H^1(\Omega) \cap L^2_0(\Omega^k)$ such that}
\begin{eqnarray}
\begin{aligned}\label{semidiscDualNitsche}
 \Delta t \sum_{k=1}^n \left(\e_u^k, \phib_u^k\right)_{\Omega^k} = \sum_{k=1}^{n} &\left\{ (\phib_u^k - \phib_u^{k-1}, \z_u^k)_{\Omega^k} + \Delta t 
 \left( {\cal A}_S^k +a_D^k\right)(\phib_u^k, \phi_p^k; \z_u^k, z_p^k) \right\} \, + (\phib_u^0, \z_u^1)_{\Omega^1} \\
 &\hspace{4cm} \forall \phib_u^k \in H^1(\Omega^k), \phi_p^k\in L^2(\Omega^k), \quad k=1,\dots, n.
\end{aligned}
\end{eqnarray}

\begin{theorem}\label{theo.L2error}
We assume that the solution $(\u,p)$ of \eqref{Stokes} fulfils the regularity assumptions $\u(t_k) \in H^{m+1}(\Omega^k)^d$ and $p(t_k) \in H^m(\Omega^k)$ for $k=1,\dots\,n$ 
and $s=1$.
 Under the assumptions of Theorem~\ref{theo.energyerror} and the inverse CFL condition $\Delta t \geq ch^2$ for some $c>0$, it holds that
 \begin{align*}
  &\left(\Delta t \sum_{k=1}^n \|\e_u^k\|_{\Omega^k}^2\right)^{1/2} \\
  &\qquad\qquad\leq c {\sfrei w_{\max}} \exp({\sfrei c_1(w_{\max})} t_n) \Big( \Delta t \|\partial_t^2 \u\|_{Q}
  +h^{m+1} \left( \|\u\|_{\infty,m+1} + \|\partial_t \u\|_{\infty,m}
  + \|p\|_{\infty,m} \right) \Big),
 \end{align*}
 {\sfrei with $c_1(w_{\max})$ specified in Lemma~\ref{lem.domainExt}.}
 \end{theorem}
\begin{proof}
 We test \eqref{semidiscDualNitsche} with $\phib_u^k = \e_u^k, \phi_p^k = e_p^k, k=0,\dots,n$ to get
 \begin{align*}
   \Delta t \sum_{k=1}^n \|\e_u^k\|_{\Omega^k}^2 = \sum_{k=1}^{n} &\left\{ (\e_u^k - \e_u^{k-1}), \z_u^k)_{\Omega^k} + \Delta t \left( {\cal A}_S^k + a_D^k\right)(\e_u^k, e_p^k; \z_u^k, z_p^k) \right\}
 + (\e_u^0, \z_u^1)_{\Omega^1}.
 \end{align*}
 We define 
 \begin{align*}
  \etab_{z,u}^k := \z_u^k - I_h^k \z_u^k, \quad  \eta_{z,p}^k := z_p^k - i_h^k z_p^k
 \end{align*}
  and use Galerkin orthogonality to insert the interpolants $I_h^k \z_u^k$ and $i_h^k z_p^k$
 \begin{eqnarray}
 \begin{aligned}\label{startL2}
  \Delta t \sum_{k=1}^n \|\e_u^k\|_{\Omega^k}^2 &= \sum_{k=1}^{n} \big\{ (\e_u^k - \e_u^{k-1}), \etab_{z,u}^k)_{\Omega^k} + \Delta t \left( {\cal A}_S^k + a_D^k\right)(\e_u^k, e_p^k; \etab_{z,u}^k, \eta_{z,p}^k)\big\}
 \\
  &\quad+\Delta t 
  \sum_{k=1}^n  \big\{D_t^{(1)} \u(t_k) - \partial_t \u(t_k), i_h^k \z_u^k)_{\Omega^k} + \gamma_g g_h^k(\u_h^k, i_h^k \z_u^k)+ \gamma_p s_h^k(p_h^k, i_h^k z_p^k)\big\}
  \end{aligned}
  \end{eqnarray}
  We use the continuity of the bilinear form ${\cal A}_S^k+a_D^k$ (Lemma~\ref{lem.contCont}) and standard interpolation estimates
  \begin{align*}
  ({\cal A}_S^k&+a_D^k)(\e_u^k, e_p^k; \etab_{z,u}^k, \eta_{z,p}^k) 
  \leq c \left(\|\nabla \e_u^k\|_{\Omega^k} +h^{-1/2} \|\e_u^k\|_{\partial\Omega^k} + h^{1/2}\|\partial_n \e_u^k\|_{\partial\Omega^k} 
  + h\|\nabla e_p^k\|_{\Omega^k} \right) \\
 &\quad\cdot\left( \|\nabla \etab_{z,u}^k\|_{\Omega^k} + h^{-1} \|\etab_{z,u}^k\|_{\Omega^k} +  \|\eta_{z,p}^k\|_{\Omega^k} 
 + h^{-1/2} \|\etab_{z,u}^k\|_{\partial\Omega^k} + h^{1/2} \left(\|\partial_n \etab_{z,u}^k\|_{\partial\Omega^k} +\|\eta_{z,p}^k\|_{\partial\Omega^k} \right)  
 \right)\\
 &\leq ch \left(\|\nabla \e_u^k\|_{\Omega^k} +h^{-1/2} \|\e_u^k\|_{\partial\Omega^k} + h^{1/2}\|\nabla \e_u^k\|_{\partial\Omega^k} + h\|\nabla e_p^k\|_{\Omega^k} \right) \left(\|\nabla^2 \z_u^k\|_{\Omega^k} + \|\nabla z_p^k\|_{\Omega^k}\right). 
  \end{align*}
  To estimate $h^{1/2}\|\nabla \e_u^k\|_{\partial\Omega^k}$ we split into a discrete and an interpolatory part and use an inverse inequality and Lemma~\ref{lem.ghost}
  \begin{eqnarray}
  \begin{aligned}\label{inverseNitsche}
   h^{1/2}\|\nabla \e_u^k\|_{\partial\Omega^k} \leq h^{1/2}\left(\|\nabla \etab_u^k\|_{\partial\Omega^k} + \|\nabla \xib_{h,u}^k\|_{\partial\Omega^k}\right)
   &\leq ch^{m} \|\u\|_{H^{m+1}(\Omega^k)} + c \|\nabla \xib_{h,u}^k\|_{\Omega_h^k}\\
   &\leq ch^{m} \|\u\|_{H^{m+1}(\Omega^k)} + c  \|\nabla \e_u^k\|_{\Omega_h^k}\\
    &\leq ch^{m} \|\u\|_{H^{m+1}(\Omega^k)} + c  |||\e_u^k|||_{h,k}.
  \end{aligned}
  \end{eqnarray}
  This yields
  \begin{align*}
   ({\cal A}_S^k&+a_D^k)(\e_u^k, e_p^k; \etab_{z,u}^k, \eta_{z,p}^k) \\
   &\leq ch\left( |||\e_u^k|||_{h,k}  + h\|\nabla e_p^k\|_{\Omega^k} + h^{m} \|\u^k\|_{H^{m+1}(\Omega^k)}\right) 
   \left(\|\nabla^2 \z_u^k\|_{\Omega^k} + \|\nabla z_p^k\|_{\Omega^k}\right).
  \end{align*}
    For the consistency error of the time derivative on the right-hand side of \eqref{startL2}, we obtain as in Lemma~\ref{lem.consistency}
  \begin{align*}
   \Big|  \left(D_t^{(1)} \u(t_k) - \partial_t \u(t_k), i_h^k \z_u^k\right)_{\Omega^n}\Big| \leq c \Delta t \|\partial_t^2 \u \|_{Q^k} \|\z_u^k\|_{\Omega^k}.
  \end{align*}
 For the ghost penalty we insert $\pm \z_u^k$ and $\pm \u(t_k)$ and use Lemma~\ref{lem.ghost} as well as standard estimates for the interpolation
 \begin{align*}
  g_h(\u_h^k, i_h^k \z_u^k) &= g_h(\e_u^k, \etab_{z,u}^k) - g_h(\u(t_k), \etab_{z,u}^k) - g_h(\e_u^k, \z_u^k) + g_h(u(t_k), \z_u^k) \\
   &\leq  ch \left( g_h^k (\e_u^k, \e_u^k)^{1/2} + h^{m} \|\u\|_{H^{m+1}(\Omega^k)}\right) \|\z_u^k\|_{H^2(\Omega^k)}.
 \end{align*}
 For the pressure stabilisation we distinguish between the cases $m=1$ and $m>1$, the latter implying by assumption that $p^k\in H^2(\Omega^k)$. For $m=1$, the following estimate is optimal
 \begin{align*}
  s_h^k(p_h^k, i_h^k z_p^k) \leq c h^2 \|\nabla p_h^k\|_{\Omega^k} \|\nabla z_p^k\|_{\Omega^k} \leq c h^2 \left(\|\nabla p^k\|_{\Omega^k} + \|\nabla e_p^k\|\right) \|\nabla z_p^k\|_{\Omega^k}  .
 \end{align*}
 For $m>1$ we insert $\pm p^k$ and use \eqref{stabass}
 \begin{align*}
  s_h^k(p_h^k, i_h^k z_p^k) = -s_h^k(e_p^k, i_h^k z_p^k) + s_h^k(p^k, i_h^k z_p^k) \leq ch \left( h\|\nabla e_p^k\|_{\Omega^k} + h^{m} \| p^k\|_{H^{m}(\Omega^k)}\right) \|\nabla z_p^k\|_{\Omega^k}.
 \end{align*}  
  It remains to estimate the terms corresponding to the discrete time derivative in \eqref{startL2}. We use a standard interpolation estimate and the inverse
  CFL condition $h^2 \leq c\Delta t$ to get
  \begin{align}\label{L2DiscTime}
    (\e_u^k - \e_u^{k-1}, \etab_{z,u}^k)_{\Omega^k}  &\leq ch^2 \|\e_u^k - \e_u^{k-1}\|_{\Omega^k} \|\nabla^2 \z_u^k\|_{\Omega^k}
    \leq c h \|\e_u^k - \e_u^{k-1}\|_{\Omega^k} \Delta t^{1/2} \|\nabla^2 \z_u^k\|_{\Omega^k}.
  \end{align}
By combining the above estimates, we have from \eqref{startL2}
\begin{eqnarray}
\begin{aligned}\label{final}
 \Delta t &\sum_{k=1}^n \|\e_u^k\|_{\Omega^k}^2
 \leq c\Delta t h \sum_{k=1}^n \big\{\big( |||\e_u^k|||_{k,h} + h\|\nabla e_p^k\|_{\Omega^k}
 + h^{m} \left( \|\u^k\|_{H^{m+1}(\Omega^k)} + \|p^k\|_{H^{m}(\Omega^k)} \right) \big)  \\
 &\qquad\qquad\qquad\qquad\cdot\left(\|\z_u^k\|_{H^2(\Omega^k)} + \|\nabla z_p^k\|_{\Omega^k} \right)\big\}\\
 &\qquad\qquad\qquad+c\Delta t^2 \sum_{k=1}^n \left\{ \|\partial_t^2 \u \|_{Q^k} \|\z_u^k\|_{\Omega^k}\right\}
 +ch \sum_{k=1}^{n-1} \left\{\|\e_u^k - \e_u^{k-1}\|_{\Omega^k} \Delta t^{1/2}\|\nabla^2 \z_u^k\|_{\Omega^k}\right\}\\
 &\leq ch^2 \Bigg( \sum_{k=1}^n \|\e_u^k-\e_u^{k-1}\|_{\Omega^k}^2 + \Delta t \Big\{ |||\e_u^k|||_{k,h}^2 + h^2\|\nabla e_p^k\|_{\Omega^k}^2 + \Delta t^2\|\partial_t^2 \u \|_{Q^k}^2\\
 &\quad\;\;+ h^{2m} \left( \|\u^k\|_{H^{m+1}(\Omega^k)}^2 + \|p^k\|_{H^{m}(\Omega^k)}^2 
 \right) \Big\}\Bigg)^{1/2} \left( \Delta t \sum_{k=1}^n \|\z_u^k\|_{H^2(\Omega^k)}^2 + \|\nabla z_p^k\|_{\Omega^k}^2\right)^{1/2}. 
\end{aligned}
\end{eqnarray}
The last inequality follows by the Cauchy-Schwarz inequality. Now, the statement follows from Theorem~\ref{theo.energyerror} and Lemma~\ref{lem.StabDual}.

\end{proof}

\begin{remark}
 An analogous result can be shown for the BDF(2) variant under slightly stronger conditions. For $\Delta t\geq ch$, which is needed for the energy estimate,
the following estimate can be shown
 \begin{eqnarray}
\begin{aligned}\label{L2s2}
  &\left(\Delta t \sum_{k=1}^n \|\e_u^k\|_{\Omega^k}^2\right)^{1/2} \\
  &\hspace{1cm}\leq c {\sfrei w_{\max}} \exp({\sfrei c_1(w_{\max})} t_n) \Big( \Delta t^2 \|\partial_t^{3} \u\|_{Q} +h^{m+1} \left( \|\u\|_{\infty,m+1} + \|\partial_t \u\|_{\infty,m}
  + \|p\|_{\infty,m} \right) \Big).
\end{aligned}
\end{eqnarray}
 The main difference in the proof is that the energy norm estimate does not give a bound for $\Delta t \|D_t^{(2)} \e_u^k\|_{\Omega^k}$, see Remark~\ref{rem.energys}. 
 We have using \eqref{deltaeps} {\sfrei with $\epsilon=1$}
 \begin{align*}
   \Delta t (D_t^{(2)} \e_u^k, \etab_{z,u}^k)_{\Omega^k} &\leq ch^2 \sum_{i=0}^{2} \|\e_u^{k+i}\|_{\Omega^k} \|\nabla^2 \z_u^k\|_{\Omega^k}
   \leq ch^2 \sum_{i=0}^{2} \left( \|\e_u^{k+i}\|_{\Omega^{k+i}}  + \|\e_u^{k+i}\|_{S_\delta^k}\right) \|\nabla^2 \z_u^k\|_{\Omega^k} \\
   &\leq c{\sfrei \Delta t^2} \sum_{i=0}^{2} \left( \|\e_u^{k+i}\|_{\Omega^{k+i}} + h \|\nabla \e_u^{k+i}\|_{\Omega^{k+i}}\right) \|\nabla^2 \z_u^k\|_{\Omega^k}.
 \end{align*}
The $L^2$-term on the right-hand side can then be absorbed into the left-hand side of \eqref{final} to obtain \eqref{L2s2}.
 \end{remark}

\section{Numerical example}
\label{sec.num}

To substantiate the theoretical findings, we present numerical results for polynomial degrees $m=1,2$ and BDF formulas of order $s=1,2$. 
The results have been obtained using the 
CutFEM library~\cite{CutFEM2015}, which is based on FEniCS~\cite{FeNiCS}.

We consider flow through a 3-dimensional rectangular channel with a moving upper and lower wall {\sfrei in the time interval is $I=[0,2]$.}
The moving domain is given by
\begin{align*}
 \Omega(t) = (0,4) \times \left(-1+\frac{\sin(t)}{10}, 1-\frac{\sin(t)}{10}\right) \times (-1,1).
 \end{align*}
{\sfrei Due to the simple polygonal structure of the domain $\Omega(t)$, the integrals in~\eqref{completeSpaceTime} are evaluated exactly within the CutFEM library~\cite{CutFEM2015} and we can expect higher-order convergence in space for $m\geq 2$.}  
 
The data $\f$ and $\u^D$ is chosen in such a way that the manufactured solution
\begin{align*}
 \u(x,y,z;t) &= \left(\sin(t)\cdot\Big((1-\frac{\sin(t)}{10})^2-y^2\Big)(1-z^2), 0, 0\right),\quad
 p(x,y,z;t) = \sin(t)\cdot(8 - 2x)
\end{align*}
solves the system \eqref{Stokes}. We impose the corresponding Dirichlet boundary conditions $\u^D$ on the left \textit{inflow} boundary (given by $x=0$), 
a \textit{do-nothing} boundary condition $\partial_n \u - p \n =0$ on the right \textit{outflow} boundary (given by $x=4$) and
no-slip boundary conditions on the remaining boundary parts, including the moving upper and lower boundary. The initial value is homogeneous $\u^0(x) = 0$. 
We choose a Nitsche parameter $\gamma_D=500$, stabilisation parameters $\gamma_g=\gamma_p=10^{-3}$ and 
$\delta=w_{\rm max} s \Delta t$, where {\sfrei $w_{\rm max}=\underset{t\in I, x \in \partial\Omega(t)}{\max} \|\partial_t \T\cdot \n\| = 0.1$}. 
{\sfrei The background triangulations ${\cal T}_h$ are constructed from}
a uniform subdivision of the 
box $[0,4]\times[-1.1, 1.1] \times [-1, 1]$ {\sfrei into hexahedra and a subsequent split of each of the hexahedral elements into 6 tetrahedra. These background triangulations 
are then reduced in each time-step by eliminating} those elements that lie outside of $\Omega_{\delta}^n$.

\subsection{$P_1$ - BDF(1)}

First, we use $P_1$ finite elements ($m=1$) and the BDF(1) variant ($s=1$). 
The computed errors $\| \u-\u_h\|_{\Omega}, \|\nabla (\u-\u_h)\|_{\Omega}, \|p^k-p_h^k\|_{\Omega}$ 
and $\|\nabla (p^k-p_h^k)\|_{\Omega}$ are 
plotted over time in Figure~\ref{fig:stokes} for {\sfrei $\Delta t=0.8 h$}, where each of the norms has been normalised by 
the $L^{\infty}(L^2)$-norm of the respective continuous functions, e.g.$\,\| \u-\u_h\|_{\Omega}/\|\u\|_{\infty,0,I}$. 
We observe convergence in all norms for all times as 
${\sfrei \Delta t=0.8 h\to 0}$. Moreover, no oscillations are visible in any of the norms. While the error bounds shown in 
the previous sections include an exponential growth in time, coming from the application of Gronwall's lemma, the error 
does not accumulate significantly over time in the numerical results presented here.
 \begin{figure}[t]
\centering
\includegraphics[width=0.48\textwidth]{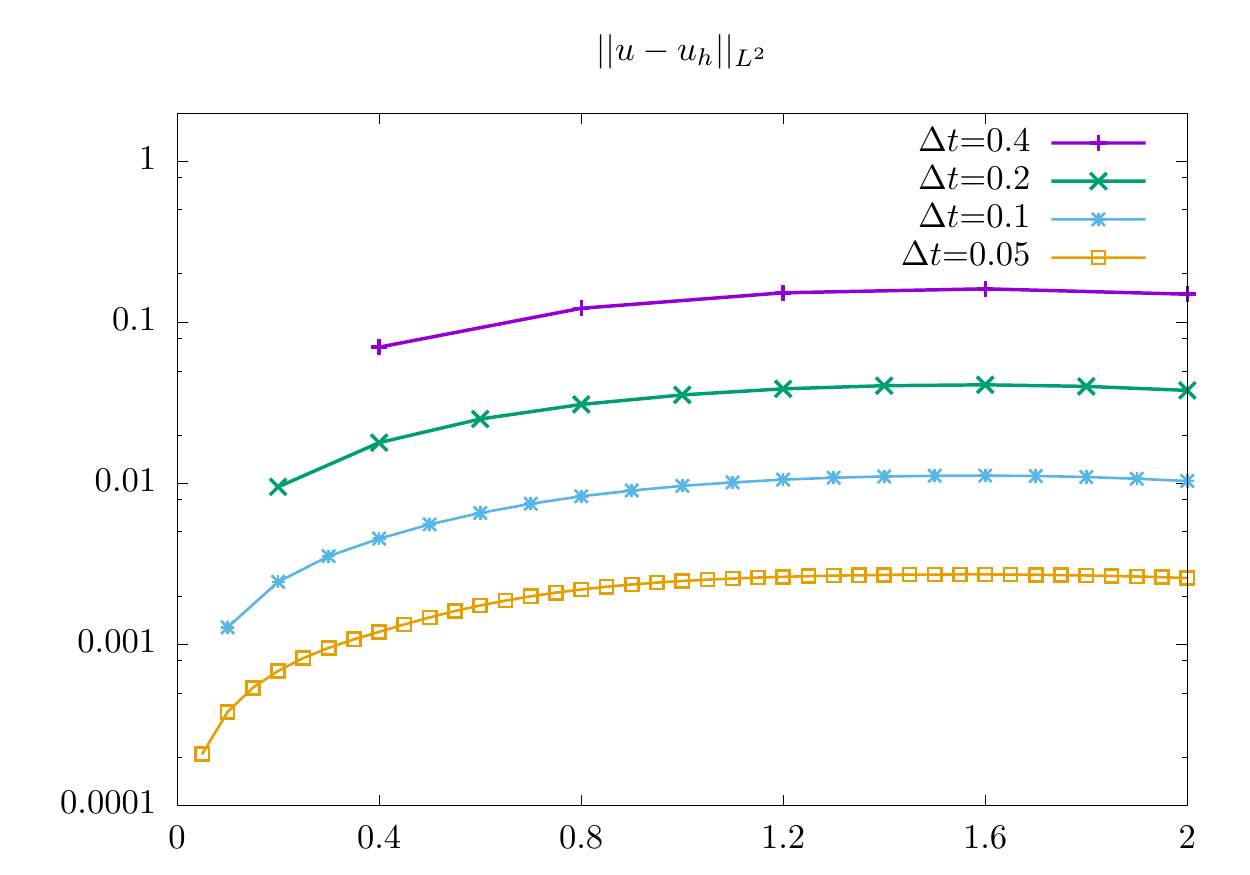}
\includegraphics[width=0.48\textwidth]{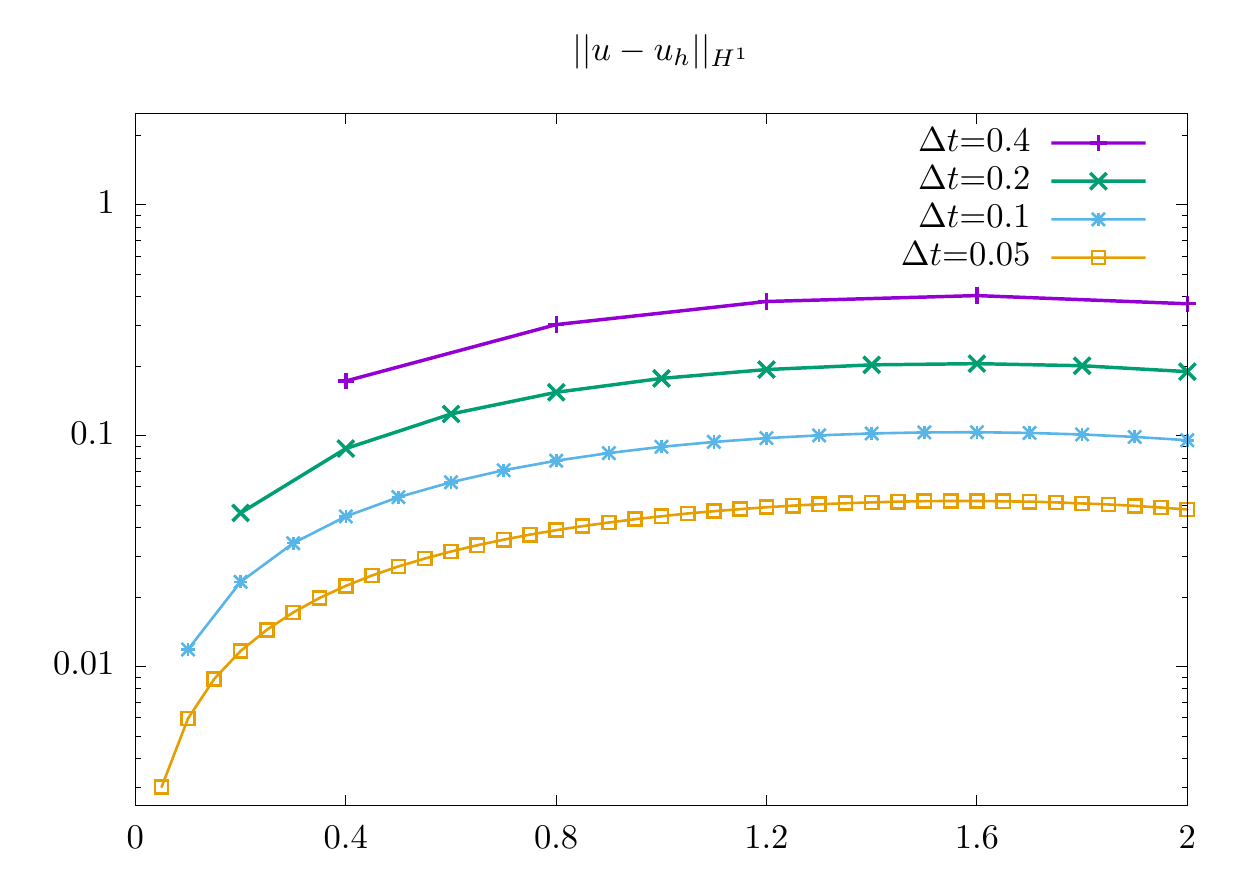}\\
 \includegraphics[width=0.48\textwidth]{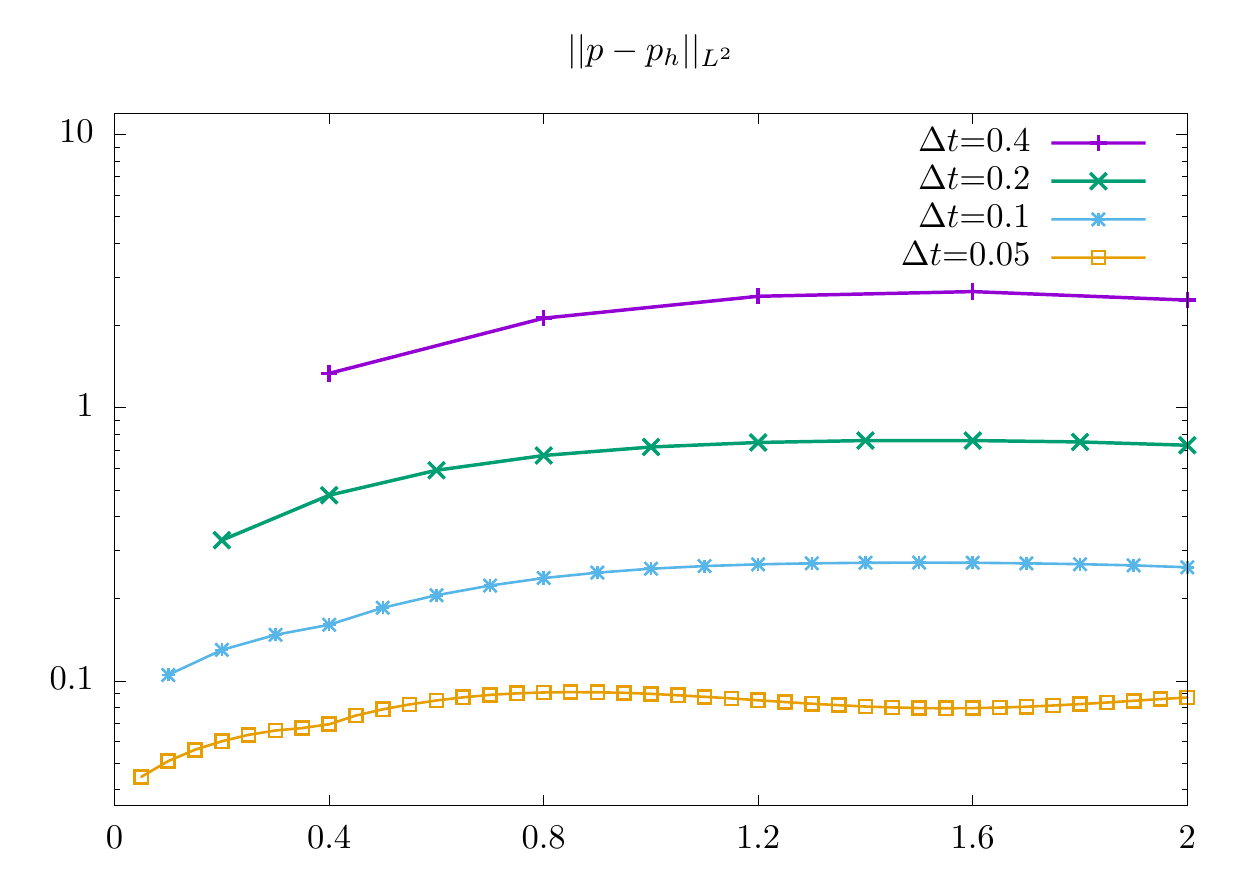} 
 \includegraphics[width=0.48\textwidth]{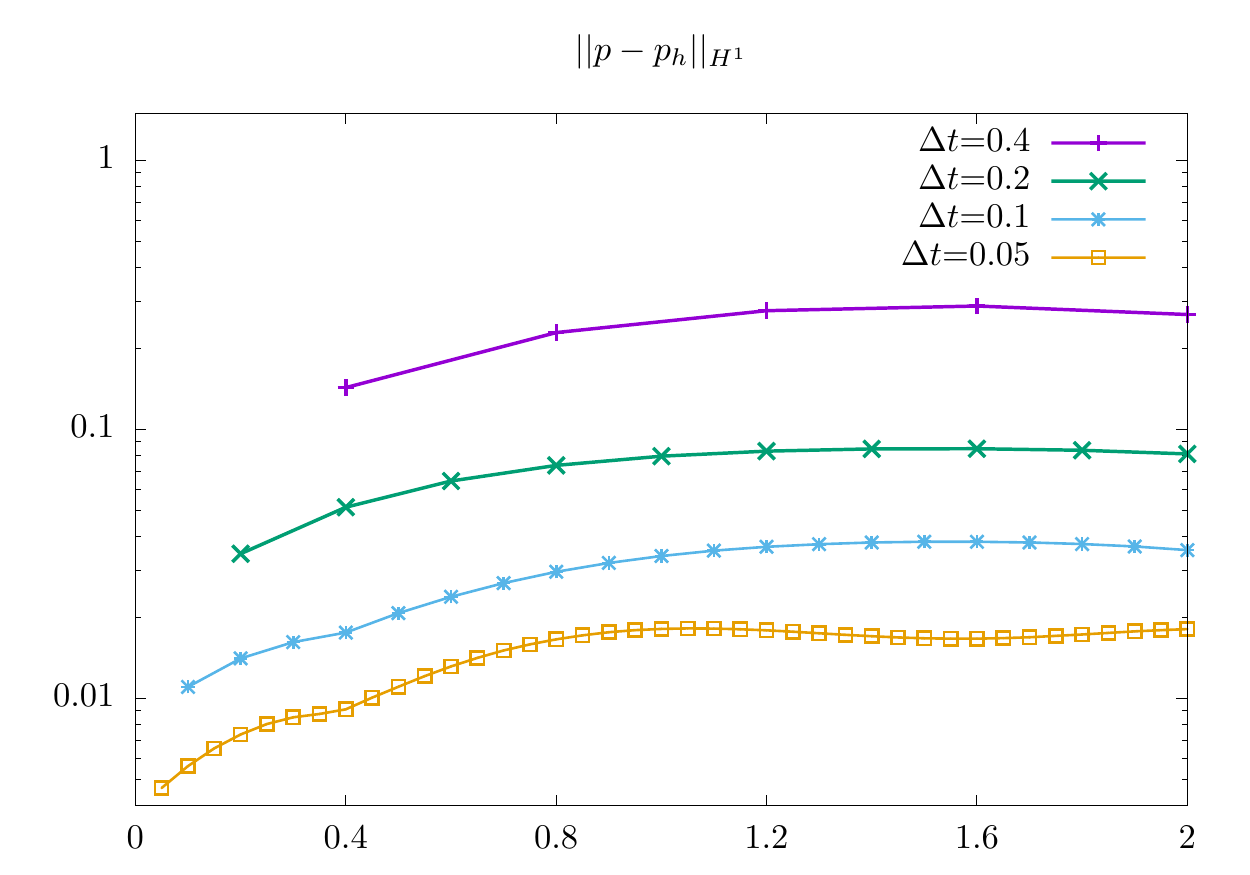}
 \caption{\label{fig:stokes} Top: $L^2$- and $H^1$-norm error of velocity, bottom: $L^2$-and $H^1$ norm error of pressure over time for 
 different time-step sizes and mesh levels, where ${\sfrei \Delta t=0.8 h}$. All the norms are normalised by the maximum (in time) of the respective norm 
 of the continuous functions.}
\end{figure}
To study the convergence orders in space and time, we show values for four different time-step and four different mesh sizes
in Table~\ref{tab:Stokes}. For $P_1$ finite elements, the finest mesh 
contains approximately 143.000 degrees of freedom. We observe that 
the temporal error is barely visible in the $L^2(L^2)$-norm and $L^2(H^1)$-semi-norm of velocities, as the spatial error is dominant. 
The spatial component of the velocities converges as expected by 
the theory (Theorems~\ref{theo.energyerror} and Theorem~\ref{theo.L2error}) with orders 2 and 1.

On the other hand, the temporal error shows up clearly in the pressure norms. To compute 
an \textit{estimated order of convergence} (eoc), let us assume that the overall error can be separated into 
a temporal and a spatial component
\begin{align}\label{splitError}
 g(\Delta t,h) = g_{\Delta t} (\Delta t) + g_h(h) = c_h h^{\rm {\rm eoc}_h} + c_{\Delta t} \Delta t^{{\rm eoc}_{\Delta t}}.
\end{align}
To estimate for instance the temporal order of convergence ${\rm eoc}_{\Delta t}$, we fit the three parameters $g_h, c_{\Delta t}$ and ${\rm eoc}_{\Delta t}$ 
of the function  
\begin{align*}
 g(\Delta t,\cdot) = g_h + c_{\Delta t} \Delta t^{\rm eoc_{\Delta t}}
\end{align*}
for a fixed mesh size {\sfrei $h\in \{\frac{1}{2}, \frac{1}{4}, \frac{1}{8}, \frac{1}{16}\}$} against the computed values. This is done by means of a least-squares fit using gnuplot~\cite{gnuplot}. {\sfrei The values for $g_h$ and eoc$_{\Delta t}$ in the first row are for example computed by fitting the previous values in the same row (i.e. those obtained with $h=\frac{1}{2}$ for different time-step sizes).}
A spatial order of convergence ${\rm eoc}_h$ is estimated similarly using the values 
for a fixed time-step size $\Delta t \,{\sfrei \in \{0.4, 0.2, 0.1, 0.05\}}${\sfrei , i.e. those in the same column.}

For the pressure norms the estimated temporal order of convergence is very close to 1 in both the $L^2$- and the $H^1$ semi-norm. 
This is expected for the $L^2(H^1)$-semi-norm by Theorem~\ref{theo.energyerror},
but better than proven in Lemma~\ref{cor.pressest} for the $L^2(L^2)$-norm. The spatial component of the error 
converges much faster than expected with ${\rm eoc}_h$ around 2 for both norms (compared to ${\cal O}(1)$,
which has been shown for the $H^1$-semi-norm, and ${\cal O}(h)$ for the $L^2$-norm). This might be due to superconvergence effects, 
as frequently observed for CIP stabilisations (see e.g.~\cite{Frei2019}), 
and possibly due to the sub-optimality of the pressure estimates.

The convergence orders of both pressure norms 
are very similar, especially for larger $h$ and $\Delta t$. Here it seems that  
due to the superconvergence of the $L^2(H^1)$-semi-norm the simple Poincar\'e estimate
\begin{align*}
  \|e_p^k\|_{\Omega^k} \leq c_P \|\nabla e_p^k\|_{\Omega^k}
 \end{align*}
is optimal for the $L^2(L^2)$-norm. Only for smaller $\Delta t$ and $h$, the convergence of the $L^2(L^2)$-norm seems to be slightly faster 
compared to the $L^2(H^1)$-semi-norm.


\begin{table}[t]
\small
  \begin{center}
    \begin{tabular}{lcccc|cc}
      \multicolumn{7}{c}{$\|\u-\u_{kh}\|_Q / \|\u\|_Q $} \\[0.1cm]
      \multicolumn{1}{l}{$h\shortdownarrow \hspace{0.01cm}\backslash \hspace{0.13cm} \Delta t\shortrightarrow\hspace{-0.7cm}$} & $0.4$ & $0.2$ & $0.1$& \multicolumn{1}{c}{$0.05$} &$g_h$ &${\red {\rm eoc}_{\Delta t}}$\\
      \midrule
      $1/2$ &$1.60\cdot {\bf 10^{-1}}$ &$1.60\cdot {\bf 10^{-1}}$ &$1.60\cdot {\bf 10^{-1}}$ &$1.57\cdot {\bf 10^{-1}}$&$1.57\cdot {\bf 10^{-1}}$ &- \\
      $1/4$ &$4.07\cdot {\bf 10^{-2}}$ &$4.05\cdot {\bf 10^{-2}}$ &$4.05\cdot {\bf 10^{-2}}$ &$4.05\cdot {\bf 10^{-2}}$&$4.04\cdot {\bf 10^{-2}}$ &{\red 1.16}  \\
      $1/8$ &$1.14\cdot {\bf 10^{-2}}$ &$1.11\cdot {\bf 10^{-2}}$ &$1.10\cdot {\bf 10^{-2}}$ &$1.10\cdot {\bf 10^{-2}}$&$1.09\cdot {\bf 10^{-2}}$ &{\red 1.57} \\
      $1/16$&$4.25\cdot {\bf 10^{-3}}$ &$3.24\cdot {\bf 10^{-3}}$ &$2.88\cdot {\bf 10^{-3}}$ &$2.76\cdot {\bf 10^{-3}}$&$2.69\cdot {\bf 10^{-3}}$ &{\red 1.52} \\
      \midrule
      $g_{\Delta t}$ &$1.90\cdot {\bf 10^{-3}}$ &$9.52\cdot {\bf 10^{-4}}$ &$5.87\cdot {\bf 10^{-4}}$ &\multicolumn{1}{c|}{$3.02\cdot {\bf 10^{-4}}$} &0 &{\red 0.90}\\
      ${\bl {\rm eoc}_h}$ &{\bl 2.03} &{\bl 2.00} &{\bl 1.99} &\multicolumn{1}{c|}{\bl 1.96} &{\bl 1.95} &\\
      \midrule\\
      \multicolumn{7}{c}{$\|\nabla (\u-\u_{kh})\|_Q / \|\nabla \u\|_Q$} \\[0.1cm]
       \multicolumn{1}{l}{$h\shortdownarrow \hspace{0.01cm}\backslash \hspace{0.13cm} \Delta t\shortrightarrow\hspace{-0.7cm}$} & $0.4$ & $0.2$ & $0.1$& \multicolumn{1}{c}{$0.05$} &$g_h$ &${\red {\rm eoc}_{\Delta t}}$\\
       \midrule
      $1/2$  &$4.13\cdot {\bf 10^{-1}}$ &$4.14\cdot {\bf 10^{-1}}$ &$4.13\cdot {\bf 10^{-1}}$ &$4.05\cdot {\bf 10^{-1}}$ &$4.05\cdot {\bf 10^{-1}}$ &-\\
      $1/4$  &$2.05\cdot {\bf 10^{-1}}$ &$2.05\cdot {\bf 10^{-1}}$ &$2.05\cdot {\bf 10^{-1}}$ &$2.05\cdot {\bf 10^{-1}}$ &$2.05\cdot {\bf 10^{-1}}$ &{\red 1.07}\\
      $1/8$  &$1.03\cdot {\bf 10^{-1}}$ &$1.03\cdot {\bf 10^{-1}}$ &$1.03\cdot {\bf 10^{-1}}$ &$1.03\cdot {\bf 10^{-1}}$ &$1.03\cdot {\bf 10^{-1}}$ &{\red 1.75}\\
      $1/16$ &$5.20\cdot {\bf 10^{-2}}$ &$5.17\cdot {\bf 10^{-2}}$ &$5.16\cdot {\bf 10^{-2}}$ &$5.16\cdot {\bf 10^{-2}}$ &$5.16\cdot {\bf 10^{-2}}$ &{\red 1.90}\\
       \midrule
       $g_{\Delta t}$ &$2.95\cdot {\bf 10^{-3}}$ &$2.94\cdot {\bf 10^{-3}}$ &$-1.88\cdot {\bf 10^{-3}}$ &\multicolumn{1}{c|}{$-1.88\cdot {\bf 10^{-3}}$} &0 &-\\
      ${\bl {\rm eoc}_h}$ &{\bl 1.02} &{\bl 1.02} &{\bl 0.98} &\multicolumn{1}{c|}{\bl 0.98} &{\bl 0.99} &\\
      \midrule\\
       \multicolumn{7}{c}{$\|p-p_{kh}\|_Q / \|p\|_Q$} \\[0.1cm]
        \multicolumn{1}{l}{$h\shortdownarrow \hspace{0.01cm}\backslash \hspace{0.13cm} \Delta t\shortrightarrow\hspace{-0.7cm}$} & $0.4$ & $0.2$ & $0.1$& \multicolumn{1}{c}{$0.05$} &$g_h$ &${\red {\rm eoc}_{\Delta t}}$\\
       \midrule
      $1/2$ &$2.83\cdot {\bf 10^{-1}}$ &$2.61\cdot {\bf 10^{-1}}$ &$2.49\cdot {\bf 10^{-1}}$ &$2.46\cdot {\bf 10^{-1}}$ &$2.42\cdot {\bf 10^{-1}}$ &{\red 1.14}\\ 
      $1/4$ &$1.08\cdot {\bf 10^{-1}}$ &$8.36\cdot {\bf 10^{-2}}$ &$7.14\cdot {\bf 10^{-2}}$ &$6.53\cdot {\bf 10^{-2}}$ &$5.92\cdot {\bf 10^{-2}}$ &{\red 1.00}\\
      $1/8$ &$6.65\cdot {\bf 10^{-2}}$ &$4.22\cdot {\bf 10^{-2}}$ &$2.98\cdot {\bf 10^{-2}}$ &$2.36\cdot {\bf 10^{-2}}$ &$1.71\cdot {\bf 10^{-2}}$ &{\red 0.98}\\
      $1/16$&$5.36\cdot {\bf 10^{-2}}$ &$2.93\cdot {\bf 10^{-2}}$ &$1.67\cdot {\bf 10^{-2}}$ &$1.02\cdot {\bf 10^{-2}}$ &$3.20\cdot {\bf 10^{-3}}$ &{\red 0.95}\\
      \midrule
      $g_{\Delta t}$ &$5.09\cdot {\bf 10^{-2}}$ &$2.08\cdot {\bf 10^{-2}}$ &$1.42\cdot {\bf 10^{-2}}$ &\multicolumn{1}{c|}{$6.83\cdot {\bf 10^{-3}}$} &0 &{\red 1.05}\\
      ${\bl {\rm eoc}_h}$ &{\bl 2.01} &{\bl 2.03} &{\bl 2.02} &\multicolumn{1}{c|}{\bl 2.18} &{\bl 2.01} &  \\
      \midrule\\
       \multicolumn{7}{c}{$\|\nabla (p-p_{kh})\|_Q / \|\nabla p\|_Q$} \\[0.1cm]
        \multicolumn{1}{l}{$h\shortdownarrow \hspace{0.01cm}\backslash \hspace{0.13cm} \Delta t\shortrightarrow\hspace{-0.7cm}$} & $0.4$ & $0.2$ & $0.1$& \multicolumn{1}{c}{$0.05$} &$g_h$ &${\red {\rm eoc}_{\Delta t}}$\\
        \midrule
       $1/2$ &$3.05\cdot {\bf 10^{-1}}$ &$2.82\cdot {\bf 10^{-1}}$ &$2.71\cdot {\bf 10^{-1}}$ & $2.68\cdot {\bf 10^{-1}}$ &$2.65\cdot {\bf 10^{-1}}$ &{\red 1.25}\\
       $1/4$ &$1.15\cdot {\bf 10^{-1}}$ &$9.25\cdot {\bf 10^{-2}}$ &$8.13\cdot {\bf 10^{-2}}$ & $7.58\cdot {\bf 10^{-2}}$ &$7.03\cdot {\bf 10^{-2}}$ &{\red 1.01}\\
       $1/8$ &$7.16\cdot {\bf 10^{-2}}$ &$4.97\cdot {\bf 10^{-2}}$ &$3.96\cdot {\bf 10^{-2}}$ & $3.51\cdot {\bf 10^{-2}}$ &$3.12\cdot {\bf 10^{-2}}$ &{\red 1.13}\\
       $1/16$ &$5.66\cdot {\bf 10^{-2}}$ &$3.39\cdot {\bf 10^{-2}}$ &$2.36\cdot {\bf 10^{-2}}$ & $1.94\cdot {\bf 10^{-2}}$&$1.58\cdot {\bf 10^{-2}}$ &{\red 1.18}\\
       \midrule
       $g_{\Delta t}$ &$5.44\cdot {\bf 10^{-2}}$ &$3.19\cdot {\bf 10^{-2}}$ &$2.20\cdot {\bf 10^{-2}}$ &\multicolumn{1}{c|}{$1.82\cdot {\bf 10^{-2}}$} &0 &{\red 0.61}\\
       ${\bl {\rm eoc}_h}$ &{\bl 2.03} &{\bl 2.02} &{\bl 2.04} &\multicolumn{1}{c|}{\bl 2.23} &{\bl 1.77}\\
      \bottomrule
    \end{tabular}   
  \end{center}
  \caption{\label{tab:Stokes}Errors for the fully discrete solutions for $P_1$ finite elements and BDF(1) for different mesh and time-step sizes. 
  The mesh size $h$ stands here for the minimal
  edge of the triangles and $Q=\{(x,t) \in \mathbb{R}^{d+1}, t\in I, x\in \Omega(t)\}$. Moreover, we
  show extrapolated values $g_h$ and $g_{\Delta t}$ and \textit{experimental order of convergences} ${\rm eoc}_h$ and ${\rm eoc}_{\Delta t}$
  based on a least squares fit of the four available values for fixed $h$ or $\Delta t$ against the three parameters of the 
  function $g(h)=g_{\Delta t} +c_h h^{\rm eoc_h}$ and $g(\Delta t)=g_h +c_{\Delta t} h^{{\rm eoc}_{\Delta t}}$, respectively. 
  On the bottom right of each table, we show estimated convergence orders of the extrapolated values $g_h$ and $g_\Delta t$ 
  by fitting the two parameters of the functions $g(\Delta t) = 0 + c_{\Delta t} \Delta t^{{\rm eoc}_{\Delta t}}$ and $g(h) = 0 + c_h h^{{\rm eoc}_h}$, respectively. {\sfrei A "-" in the eoc column means that the asymptotic standard error of the fit (computed by gnuplot) was $>20$\%, which means that the parameter could not be determined with a reasonable accuracy from the given values.} }
  \end{table}

\subsection{$P_2$-BDF(1)}
\label{sec.P2BDF1}

In order to increase the visibility of the temporal error component, we increase the order of the spatial discretisation first.
In Table~\ref{tab:StokesP2rect} we show results for $P_2$ finite elements and BDF(1) ($m=2, s=1$) on three different mesh levels. 
For $P_2$ the finest mesh level has again 
around 143.000 degrees of freedom, which is similar to $P_1$ elements on the next-finer mesh level.
Again the spatial error is dominant in the velocity norms on coarser meshes and 
shows convergence orders of approximately 3 in the $L^2(L^2)$-norm and 2 in the $L^2(H^1)$-semi-norm, 
as shown in Theorems~\ref{theo.L2error} and \ref{theo.energyerror}.
In contrast to $P_1$ elements, the temporal error is however visible on the finest mesh level, where ${\rm eoc}_{\Delta t}$ is close to 1, as expected.

In the $L^2(L^2)$-norm of pressure, the temporal error is dominant and shows again a convergence order of ${\cal O}(\Delta t)$. 
Due to the dominance
of the temporal component, it is less clear to deduce the spatial error contribution. From the values and the ${\rm eoc}_h$ 
it seems to converge again faster as predicted. 
{\sfrei Concerning the $L^2(H^1)$-norm of pressure, 
the assumption \eqref{splitError} that the spatial 
and temporal error are separated,
which was assumed in order to compute eoc$_{\Delta t}$ and eoc$_h$, is not valid, as the extrapolated values $g_h$ and $g_{\Delta t}$ do not
or converge only very slowly towards zero. For this reason, the computed convergence orders eoc$_{\Delta t}$ and eoc$_h$ are not meaningful in this case.} This does not contradict the theory, as
Theorem~\eqref{theo.energyerror} guarantees only the bound
\begin{align*}
 \left( \sum_{k=1}^n \Delta t \|\nabla e_p^k\|_{\Omega^k}^2\right)^{1/2} \leq {\cal O}\left(\frac{\Delta t}{h}\right) + {\cal O}(h).
\end{align*}

 \begin{table}[t]
\small
  \begin{center}
    \begin{tabular}{lcccc|cc}
      \multicolumn{7}{c}{$\|\u-\u_{kh}\|_Q / \|\u\|_Q $}\\[0.1cm]
      \multicolumn{1}{l}{$h\shortdownarrow \hspace{0.01cm}\backslash \hspace{0.13cm} \Delta t\shortrightarrow\hspace{-0.7cm}$} & $0.4$ & $0.2$ & $0.1$& \multicolumn{1}{c}{$0.05$} &$g_h$  &${\red {\rm eoc}_{\Delta t}}$ \\
      \midrule
      $1/2$ &$9.15\cdot {\bf 10^{-3}}$ &$8.92\cdot {\bf 10^{-3}}$ &$8.88\cdot {\bf 10^{-3}}$ &$8.88\cdot {\bf 10^{-3}}$ &$8.87\cdot {\bf 10^{-3}}$ &{\red 2.40}\\
      $1/4$ &$3.20\cdot {\bf 10^{-3}}$ &$1.97\cdot {\bf 10^{-3}}$ &$1.50\cdot {\bf 10^{-3}}$ &$1.38\cdot {\bf 10^{-3}}$ &$1.28\cdot {\bf 10^{-3}}$ &{\red 1.50}\\
      $1/8$ &$3.12\cdot {\bf 10^{-3}}$ &$1.62\cdot {\bf 10^{-3}}$ &$8.52\cdot {\bf 10^{-4}}$ &$4.78\cdot {\bf 10^{-4}}$ &$8.59\cdot {\bf 10^{-5}}$ &{\red 0.99}\\
      \midrule
      $g_{\Delta t}$ &$3.12\cdot {\bf 10^{-3}}$ &$1.60\cdot {\bf 10^{-3}}$ &$7.90\cdot {\bf 10^{-4}}$ & \multicolumn{1}{c|}{$3.55\cdot {\bf 10^{-4}}$} &0 &{\red 0.99} \\
      ${\bl {\rm eoc}_h}$ &{\bl 6.08} &{\bl 4.31} &{\bl 3.51} & \multicolumn{1}{c|}{\bl 3.06} &{\bl 2.82}& \\
      \midrule\\
      \multicolumn{7}{c}{$\|\nabla (\u-\u_{kh})\|_Q / \|\nabla \u\|_Q$}\\[0.1cm]
      \multicolumn{1}{l}{$h\shortdownarrow \hspace{0.01cm}\backslash \hspace{0.13cm} \Delta t\shortrightarrow\hspace{-0.7cm}$} & $0.4$ & $0.2$ & $0.1$& \multicolumn{1}{c}{$0.05$} &$g_h$ &${\red {\rm eoc}_{\Delta t}}$ \\
      \midrule
      $1/2$ &$5.55\cdot {\bf 10^{-2}}$ &$5.55\cdot {\bf 10^{-2}}$ &$5.55\cdot {\bf 10^{-2}}$ &$5.55\cdot {\bf 10^{-2}}$ &$5.55\cdot {\bf 10^{-2}}$ &-\\
      $1/4$ &$1.67\cdot {\bf 10^{-2}}$ &$1.58\cdot {\bf 10^{-2}}$ &$1.56\cdot {\bf 10^{-2}}$ &$1.56\cdot {\bf 10^{-2}}$ &$1.56\cdot {\bf 10^{-2}}$ &{\red 1.91}\\
      $1/8$ &$7.51\cdot {\bf 10^{-3}}$ &$5.18\cdot {\bf 10^{-3}}$ &$4.36\cdot {\bf 10^{-3}}$ &$4.12\cdot {\bf 10^{-3}}$ &$3.97\cdot {\bf 10^{-3}}$ &{\red 1.56}\\
       \midrule
      $g_{\Delta t}$ &$4.66\cdot {\bf 10^{-3}}$ &$1.30\cdot {\bf 10^{-3}}$ &$-4.82\cdot {\bf 10^{-5}}$ & \multicolumn{1}{c|}{$-5.17\cdot {\bf 10^{-4}}$} &0 &{\red 2.13}\\
      ${\bl {\rm eoc}_h}$ &{\bl 2.08} &{\bl 1.90} &{\bl 1.83} & \multicolumn{1}{c|}{\bl 1.80} &{\bl 1.85}&\\
      \midrule\\
       \multicolumn{7}{c}{$\|p-p_{kh}\|_Q / \|p\|_Q$}\\[0.1cm]
       \multicolumn{1}{l}{$h\shortdownarrow \hspace{0.01cm}\backslash \hspace{0.13cm} \Delta t\shortrightarrow\hspace{-0.7cm}$} & $0.4$ & $0.2$ & $0.1$& \multicolumn{1}{c}{$0.05$}  &$g_h$ &${\red {\rm eoc}_{\Delta t}}$ \\
       \midrule
      $1/2$ &$5.35\cdot {\bf 10^{-2}}$ &$2.87\cdot {\bf 10^{-2}}$ &$1.57\cdot {\bf 10^{-2}}$ &$9.20\cdot {\bf 10^{-3}}$ &$2.06\cdot {\bf 10^{-3}}$ &{\red 0.95} \\ 
      $1/4$ &$5.08\cdot {\bf 10^{-2}}$ &$2.64\cdot {\bf 10^{-2}}$ &$1.37\cdot {\bf 10^{-2}}$ &$7.18\cdot {\bf 10^{-3}}$ &$1.10\cdot {\bf 10^{-3}}$ &{\red 0.95}\\
      $1/8$ &$5.00\cdot {\bf 10^{-2}}$ &$2.57\cdot {\bf 10^{-2}}$ &$1.30\cdot {\bf 10^{-2}}$ &$6.53\cdot {\bf 10^{-3}}$ &$-5.40\cdot {\bf 10^{-4}}$ &{\red 0.95}\\
      \midrule
      $g_{\Delta t}$ &$4.97\cdot {\bf 10^{-2}}$ &$2.54\cdot {\bf 10^{-2}}$ &$1.26\cdot {\bf 10^{-2}}$ & \multicolumn{1}{c|}{$6.22\cdot {\bf 10^{-3}}$} &0 &{\red 0.99}\\
      ${\bl {\rm eoc}_h}$ &{\bl 1.75} &{\bl 1.72} &{\bl 1.51} & \multicolumn{1}{c|}{\bl 1.64} &{\bl 1.71}&\\
      \midrule\\
      \multicolumn{7}{c}{$\|\nabla (p-p_{kh})\|_Q / \|\nabla p\|_Q$}\\[0.1cm]
      \multicolumn{1}{l}{$h\shortdownarrow \hspace{0.01cm}\backslash \hspace{0.13cm} \Delta t\shortrightarrow\hspace{-0.7cm}$} & $0.4$ & $0.2$ & $0.1$& \multicolumn{1}{c}{$0.05$}  &$g_h$ &${\red {\rm eoc}_{\Delta t}}$ \\
      \midrule
      $1/2$ &$6.29\cdot {\bf 10^{-2}}$ &$4.13\cdot {\bf 10^{-2}}$ &$3.20\cdot {\bf 10^{-2}}$ & $2.83\cdot {\bf 10^{-2}}$&$2.54\cdot {\bf 10^{-2}}$ &{\red 1.24}\\
      $1/4$ &$5.43\cdot {\bf 10^{-2}}$ &$3.18\cdot {\bf 10^{-2}}$ &$2.17\cdot {\bf 10^{-2}}$ & $1.78\cdot {\bf 10^{-2}}$&$1.45\cdot {\bf 10^{-2}}$ &{\red 1.21}\\
      $1/8$ &$5.25\cdot {\bf 10^{-2}}$ &$2.99\cdot {\bf 10^{-2}}$ &$1.97\cdot {\bf 10^{-2}}$ & $1.60\cdot {\bf 10^{-2}}$&$1.27\cdot {\bf 10^{-2}}$ &{\red 1.23}\\
      \midrule
      $g_{\Delta t}$ &$5.20\cdot {\bf 10^{-2}}$ &$2.94\cdot {\bf 10^{-2}}$ &$1.92\cdot {\bf 10^{-2}}$ & \multicolumn{1}{c|}{$1.56\cdot {\bf 10^{-2}}$} & 0 &{\red 0.67}\\
      ${\bl {\rm eoc}_h}$ &{\bl 2.26} &{\bl 2.32} &{\bl 2.36} &\multicolumn{1}{c|}{\bl 2.54} & {\bl 0.57} &\\
      \bottomrule
    \end{tabular}   
  \end{center}
  \caption{\label{tab:StokesP2rect}Errors for the fully discrete solutions for $P_2$ finite elements and BDF(1) for different mesh and time-step sizes. The \textit{experimental orders of 
  convergence} (eoc) have been computed as in Table~\ref{tab:Stokes}.}
\end{table}

\subsection{$P_2$-BDF(2)}

Finally, we show results for $m=2$ and $s=2$ in Table~\ref{tab:StokesP2BDF2rect}. In order to simplify the initialisation, 
we use that the analytically given solution $\u(x,t)$ can be extended to $t<0$ and use the starting values $\u^0=0$ and $\u^{-1} := \u(-\Delta t)$ 
in the first time step. 
Due to the (expected) second-order convergence in time, 
the temporal error is barely visible in the velocity norms on the finer mesh levels, in contrast to the results for BDF(1). 
The estimated order of convergence of the spatial component lies slightly below the orders 3 and 2 in the $L^2(L^2)$-norm and 
$L^2(H^1)$-semi-norm, respectively, that have been shown analytically.

In the $L^2(L^2)$-norm of pressure both temporal and spatial errors are visible.
Both ${\rm eoc}_h$ and ${\rm eoc}_{\Delta t}$ are around 2, which has been shown in Section~\ref{sec.L2L2press} for the spatial part. For the 
temporal part only a reduced order of convergence of ${\cal O}(\Delta t)$ has been shown theoretically. This bound seems not to be sharp
in the numerical example studied here. 
In the $L^2(H^1)$-semi-norm of pressure the spatial error is dominant, which is in contrast to the BDF(1) results.  
However, the assumption \eqref{splitError} that the error allows for a separation into spatial and temporal 
error components is again not valid, {\sfrei which makes the computed values of eoc$_\Delta t$ and eoc$_h$ meaningless.}

 \begin{table}[t]
\small
  \begin{center}
    \begin{tabular}{lcccc|cc}
      \multicolumn{7}{c}{$\|\u-\u_{kh}\|_Q / \|\u\|_Q $} \\[0.1cm]
      \multicolumn{1}{l}{$h\shortdownarrow \hspace{0.01cm}\backslash \hspace{0.13cm} \Delta t\shortrightarrow\hspace{-1.5cm}$} & $0.4$ & $0.2$ & $0.1$& \multicolumn{1}{c}{$0.05$} &$g_h$ &${\red {\red {\rm eoc}_{\Delta t}}}$ \\
      \midrule
      $1/2$ &$8.91\cdot {\bf 10^{-3}}$ &$8.89\cdot {\bf 10^{-3}}$ &$8.89\cdot {\bf 10^{-3}}$ &$8.89\cdot {\bf 10^{-3}}$ &$8.89\cdot {\bf 10^{-3}}$ &{\red 3.47}\\
      $1/4$ &$1.44\cdot {\bf 10^{-3}}$ &$1.35\cdot {\bf 10^{-3}}$ &$1.35\cdot {\bf 10^{-3}}$ &$1.35\cdot {\bf 10^{-3}}$ &$1.35\cdot {\bf 10^{-3}}$ &{\red 5.04} \\
      $1/8$ &$6.38\cdot {\bf 10^{-4}}$ &$2.77\cdot {\bf 10^{-4}}$ &$2.33\cdot {\bf 10^{-4}}$ &$2.29\cdot {\bf 10^{-4}}$ &$2.28\cdot {\bf 10^{-4}}$ &{\red 3.06}\\
      \midrule
      $g_{\Delta t}$ &$5.41\cdot {\bf 10^{-4}}$ &$9.90\cdot {\bf 10^{-5}}$ &$3.87\cdot {\bf 10^{-5}}$ &\multicolumn{1}{c|}{$3.32\cdot {\bf 10^{-5}}$} &0 &{\red 2.31}\\
      {\bl \text{eoc}$_h$} &{\bl 3.22} &{\bl 2.81} &{\bl 2.75} &\multicolumn{1}{c|}{\bl 2.75} &{\bl 2.71} \\
      \midrule
      \multicolumn{7}{c}{}\\ 
      \multicolumn{7}{c}{$\|\nabla (\u-\u_{kh})\|_Q / \|\nabla \u\|_Q$}\\[0.1cm]
      \multicolumn{1}{l}{$h\shortdownarrow \hspace{0.01cm}\backslash \hspace{0.13cm} \Delta t\shortrightarrow\hspace{-0.7cm}$} & $0.4$ & $0.2$ & $0.1$& \multicolumn{1}{c}{$0.05$} &$g_h$ &${\red{\rm eoc}_{\Delta t}}$ \\
      \midrule
      $1/2$ &$5.60\cdot {\bf 10^{-2}}$ &$5.60\cdot {\bf 10^{-2}}$ &$5.60\cdot {\bf 10^{-2}}$ &$5.60\cdot {\bf 10^{-2}}$ &$5.60\cdot {\bf 10^{-2}}$ &{\red 2.74}\\ 
      $1/4$ &$1.56\cdot {\bf 10^{-2}}$ &$1.56\cdot {\bf 10^{-2}}$ &$1.56\cdot {\bf 10^{-2}}$ &$1.56\cdot {\bf 10^{-2}}$ &$1.56\cdot {\bf 10^{-2}}$ &{\red 3.60}\\
      $1/8$ &$4.23\cdot {\bf 10^{-3}}$ &$4.06\cdot {\bf 10^{-3}}$ &$4.05\cdot {\bf 10^{-3}}$ &$4.05\cdot {\bf 10^{-3}}$ &$4.05\cdot {\bf 10^{-3}}$ &{\red 3.91}\\
       \midrule
       $g_{\Delta t}$ &$-2.23\cdot {\bf 10^{-4}}$ &$-5.54\cdot {\bf 10^{-4}}$ &$-5.74\cdot {\bf 10^{-4}}$ &\multicolumn{1}{c|}{$-5.74\cdot {\bf 10^{-4}}$} &0 &-\\
       ${\bl {\rm eoc}_h}$ &{\bl 1.83} &{\bl 1.81} &{\bl 1.81} &\multicolumn{1}{c|}{\bl 1.81} &{\bl 1.86} & \\
      \midrule\\ 
       \multicolumn{7}{c}{$\|p-p_{kh}\|_Q / \|p\|_Q$}\\[0.1cm]
       \multicolumn{1}{l}{$h\shortdownarrow \hspace{0.01cm}\backslash \hspace{0.13cm} \Delta t\shortrightarrow\hspace{-0.7cm}$} & $0.4$ & $0.2$ & $0.1$& \multicolumn{1}{c}{$0.05$} &$g_h$ &${\red {\rm eoc}_{\Delta t}}$ \\
       \midrule
      $1/2$ &$7.78\cdot {\bf 10^{-3}}$ &$3.32\cdot {\bf 10^{-3}}$ &$2.96\cdot {\bf 10^{-3}}$ &$2.95\cdot {\bf 10^{-3}}$ &$2.94\cdot {\bf 10^{-3}}$ &{\red 3.67}\\ 
      $1/4$ &$7.27\cdot {\bf 10^{-3}}$ &$1.86\cdot {\bf 10^{-3}}$ &$9.41\cdot {\bf 10^{-4}}$ &$8.91\cdot {\bf 10^{-4}}$ &$8.29\cdot {\bf 10^{-4}}$ &{\red 2.66}\\
      $1/8$ &$7.27\cdot {\bf 10^{-3}}$ &$1.79\cdot {\bf 10^{-3}}$ &$5.34\cdot {\bf 10^{-4}}$ &$3.28\cdot {\bf 10^{-4}}$ &$2.28\cdot {\bf 10^{-4}}$ &{\red 2.18}\\
      \midrule
      $g_{\Delta t}$ &$7.27\cdot {\bf 10^{-3}}$ &$1.79\cdot {\bf 10^{-3}}$ &$4.31\cdot {\bf 10^{-4}}$ &\multicolumn{1}{c|}{$1.16\cdot {\bf 10^{-4}}$} &0 &{\red 1.97}\\
      ${\bl {\rm eoc}_h}$ &{\bl 10.01} &{\bl 4.38} &{\bl 2.31} &\multicolumn{1}{c|}{\bl 1.87} &{\bl 1.83} & \\
      \midrule\\ 
      \multicolumn{7}{c}{$\|\nabla (p-p_{kh})\|_Q / \|\nabla p\|_Q$}\\[0.1cm]
      \multicolumn{1}{l}{$h\shortdownarrow \hspace{0.01cm}\backslash \hspace{0.13cm} \Delta t\shortrightarrow\hspace{-0.7cm}$} & $0.4$ & $0.2$ & $0.1$& \multicolumn{1}{c}{$0.05$} &$g_h$ &${\red {\rm eoc}_{\Delta t}}$ \\
      \midrule
      $1/2$ &$3.43\cdot {\bf 10^{-2}}$ &$3.31\cdot {\bf 10^{-2}}$ &$3.29\cdot {\bf 10^{-2}}$ & $3.29\cdot {\bf 10^{-2}}$ &$3.29\cdot {\bf 10^{-2}}$&{\red 3.35}\\
      $1/4$ &$1.96\cdot {\bf 10^{-2}}$ &$1.81\cdot {\bf 10^{-2}}$ &$1.80\cdot {\bf 10^{-2}}$ & $1.80\cdot {\bf 10^{-2}}$ &$1.80\cdot {\bf 10^{-2}}$ &{\red 4.04}\\
      $1/8$ &$1.86\cdot {\bf 10^{-2}}$ &$1.71\cdot {\bf 10^{-2}}$ &$1.70\cdot {\bf 10^{-2}}$ & $1.70\cdot {\bf 10^{-2}}$ &$1.70\cdot {\bf 10^{-2}}$ &{\red 4.05}\\
      \midrule
      $g_{\Delta t}$ &$1.85\cdot {\bf 10^{-2}}$ &$1.70\cdot {\bf 10^{-2}}$ &$1.69\cdot {\bf 10^{-2}}$ &\multicolumn{1}{c|}{$1.69\cdot {\bf 10^{-2}}$} &0 &-\\
      ${\bl {\rm eoc}_h}$ &{\bl 3.88} &{\bl 3.91} &{\bl 3.90} &\multicolumn{1}{c|}{\bl 3.90} &{\bl 0.57} &\\
      \bottomrule
    \end{tabular}   
  \end{center}
  \caption{\label{tab:StokesP2BDF2rect}Errors for the fully discrete solutions for $P_2$ finite elements and BDF(2) for different mesh and time-step sizes. The \textit{experimental orders of 
  convergence} (eoc) have been computed as in Table~\ref{tab:Stokes}.}
\end{table}

\section{Conclusion}
\label{sec.concl}

We have derived a detailed a priori error analysis for two Eulerian time-stepping schemes based on backward difference formulas applied to the non-stationary Stokes equations on 
time-dependent domains. Following Schott~\cite{SchottDiss} and Lehrenfeld \& Olshanskii~\cite{LehrenfeldOlshanskii} discrete quantities are extended implicitly by means 
of ghost penalty terms to a larger domain, which is needed in the following step of the time-stepping scheme.

In particular, we have shown optimal-order error estimates for the $L^2(H^1)$-semi-norm and the $L^2(L^2)$-norm error for the velocities. 
The main difficulties herein consisted in the transfer 
of quantities between domains $\Omega^n$ and $\Omega^{n-1}$ at different time-steps and in the estimation of the pressure error. Optimal $L^2(H^1)$-norm errors for the pressure can be derived {\sfrei under the inverse CFL conditions $\Delta t \geq ch^2$ for the CIP pressure stabilisation and BDF(1) ($\Delta t \geq ch$ for BDF(2)), or unconditionally, 
when the Brezzi-Pitk\"aranta pressure stabilisation is used}. Fortunately, these estimates 
are sufficient to show optimal bounds for the velocities in both the $L^2(H^1)$- and the $L^2(L^2)$-norms. All these estimates are in good agreement with the numerical results presented. 

For the $L^2(L^2)$-norm error of the pressure, we have shown suboptimal bounds in terms of the time step $\Delta t$. The derivation of optimal bounds seems to be non-trivial and needs to be investigated 
in future work. {\sfrei Moreover, it would be interesting to further investigate if the exponential growth in the stability and error estimates can indeed be observed in numerical computations, for example by considering more complex domain motions.}

Further directions of research are the application of the approach to the non-linear Navier-Stokes equations, multi-phase flows and 
fluid-structure interactions, as well as the investigation of 
different time-stepping schemes, such as Crank-Nicolson or the fractional-step $\theta$ scheme within the framework presented and investigated in the 
present work.

\begin{acknowledgement}
The first author acknowledges support by the EPSRC grant EP/P01576X/1.
The second author was supported by the DFG Research Scholarship FR3935/1-1.
The work of the third author was funded by the Swedish Research Council under Starting Grant 2017-05038.
\end{acknowledgement}

\clearpage

\section*{Appendix: Proof of Lemma~\ref{lem.wellposed}}

\begin{proof}
Our proof is similar to the one given in~\cite{Bock1977} for the non-linear Navier-Stokes equations. As usual, we start by showing 
existence and uniqueness for the velocities $\u$ by considering a reduced problem in the space of divergence-free trial and test functions
\begin{eqnarray}
\begin{aligned}\label{RedSpaces}
 {\cal V}_0(t) &:= \{\u \in {\cal V}(t), \; {\div\, \u} = 0 \; \text{a.e. in } \Omega(t)\}, \\
 {\cal V}_{0,I} &:= \{ \u\in L^2(I, {\cal V}_0(t)), \; \partial_t \u \in L^2(I, {\cal L}(t)^d)\}.
\end{aligned}
\end{eqnarray}
The reduced problem is given by: \textit{Find } $\u\in {\cal V}_{0,I}$ such that
\begin{align}\label{RedStokes}
 (\partial_t \u, \v)_{\Omega(t)} +  (\nabla \u, \nabla \v)_{\Omega(t)}  &= (\f, \v)_{\Omega(t)} \quad \forall \v \in {\cal V}_0(t) &&\text{a.e. in } t \in I,\\
 \u(x,0) &= \u^0(x) &&\text{a.e. in } \Omega(0).
\end{align}
It can be easily seen that $\u\in {\cal V}_{0,I}$ is a solution to \eqref{RedStokes} if and only if it is the velocity part of a solution to \eqref{Stokes}.

(i) \textit{Transformation}: By means of the map $\T$ in \eqref{map}, we can transform the system of equations to an equivalent system on $\Omega(0)$: \textit{Find } $\hat \u\in \hat{\cal V}_{0,I}$ such that
\begin{eqnarray}
\begin{aligned}\label{RedStokesALE}
 \left(J(t) (\partial_t \hat \u -F^{-1}(t) \partial_t \T(t) \cdot \hat\nabla \hat \u ), \hat v\right)_{\Omega(0)} 
 &+  (J(t) \hat\nabla \hat \u F(t)^{-1}, \hat\nabla \hat \v F(t)^{-1})_{\Omega(0)} \\
 &= (J(t) \hat \f, \hat \v)_{\Omega(0)} \quad \forall \hat \v \in \hat V_0 &&\text{a.e. in } t \in I,\\
 \hat{\u}(x,0) &= \hat{\u}^0(x) \quad &&\text{a.e. in } \Omega(0),
\end{aligned}
\end{eqnarray}
where $F=\hat\nabla \T, J={\rm det}\, F$, $\hat\nabla$ denotes derivatives with respect to $\Omega(0)$ and quantities with a ``hat'' correspond to their counterparts without a hat by the relation
\begin{align*}
 \hat \u(x, t) = \u(\T(x,t),t) \quad \text{for } x\in \Omega(0).
\end{align*}
Test and trial spaces are defined as
\begin{align*}
 \hat {\cal V}_0(t) &:= \{\hat \u \in {\cal V}(0), \; {\hat\div (J(t) F(t)^{-1} \hat \u)} = 0 \; \text{a.e. in } \Omega(0)\}, \\
 \hat{\cal V}_{0,I} &:= \{ \hat \u\in L^2(I, \hat {\cal V}_0(t)), \; \partial_t \hat\u \in L^2(I, {\cal L}(0)^d)\}.
\end{align*}
Given that $\T$ is a $W^{1,\infty}$-diffeomorphism, it can be shown that~\cite{FailerDiss}
\begin{align*}
 \u \in {\cal V}_0(t) \; \Leftrightarrow \;  \hat \u \in \hat {\cal V}_0(t), \quad \u \in {\cal V}_{0,I} \; \Leftrightarrow \;  \hat \u \in \hat {\cal V}_{0,I}.
\end{align*}

We will show the well-posedness of \eqref{RedStokesALE} by a Galerkin argumentation. A basis $\left\{\hat{\w}_j\right\}_{j\in\mathbb{N}}$ of the time-dependent space $\hat{\cal V}_0(t)$ is given by the inverse Piola transform of an $L^2$-orthonormal basis $\left\{\hat{\phi}_j\right\}_{j\in\mathbb{N}}$ of the space $\hat{\cal V}_0(0)$
\begin{align*}
\hat{\w}_j(t) = J(t)^{-1}F(t) \hat{\phi}_j, \quad j\in\mathbb{N}.
\end{align*}
Under the given regularity assumptions on the domain movement $T$, the basis functions lie in $W^{1,\infty}(I, H^1(\Omega(0))^d)$. 

(ii) \textit{Galerkin approximation:} The ansatz
\begin{align*}
\hat{\u}_l = \sum_{j=1}^l \alpha_j(t) \hat{\w}_j(t)
\end{align*}
with coefficients $\alpha_j(t)\in\mathbb{R}$
leads to the Galerkin problem
\begin{eqnarray}
\begin{aligned}\label{GalerkinALE}
 \left(J(t) (\partial_t \hat \u_l -F^{-1}(t) \partial_t \T(t) \cdot \hat\nabla \hat \u_l ), \hat \w_k\right)_{\Omega(0)} 
 &+  (J(t) \hat\nabla \hat \u_l F(t)^{-1}, \hat\nabla \hat \w_k F(t)^{-1})_{\Omega(0)} \\
 &= (J(t) \hat \f, \hat \w_k)_{\Omega(0)} \quad k=1,...,l,\\
 \hat{\u}_l(x,0) &= \hat{\u}_l^0(x) \quad \text{a.e. in } \Omega(0),
\end{aligned}
\end{eqnarray}
where $\hat{\u}_l^0$ is an $L^2$-orthogonal projection of $\hat{\u}^0$ onto span$\{\hat{\w}_1,...,\hat{\w}_l\}$.
This is a system of ordinary differential equation for the coefficients $\alpha_j(t), j=1,...,l$
\begin{eqnarray}
\begin{aligned}\label{GalerkinALE2}
 \sum_{j=1}^l \alpha_j'(t) &\underbrace{\left(J(t) \hat \w_j, \hat \w_k\right)_{\Omega(0)}}_{M(t)} + \alpha_j(t) \underbrace{\left(J(t) (\partial_t \hat \w_j
   -F^{-1}(t) \partial_t \T(t) \cdot \hat\nabla \hat \w_j ), \hat \w_k\right)_{\Omega(0)}}_{B(t)} \\
 &+  \alpha_j(t) \underbrace{(J(t) \hat\nabla \hat \w_j F(t)^{-1}, \hat\nabla \hat \w_k F(t)^{-1})_{\Omega(0)}}_{A(t)} 
 = \underbrace{(J(t) \hat \f, \hat \w_k)_{\Omega(0)}}_{b(t)} \quad k=1,...,l. 
\end{aligned}
\end{eqnarray}
The assumption that $T$ describes a $W^{1,\infty}(\Omega(0))$ diffeomorphism implies that
\begin{align*}
0 < J_{\min} < J(t) < J_{\max} < \infty, \qquad J_{\min}, J_{\max}\in\mathbb{R}.
\end{align*}
It follows that the matrix $M(t)$ is invertible for all $t\in I$ and we can write \eqref{GalerkinALE2} as
\begin{align}\label{Galerkin3}
\alpha' = -M(t)^{-1} (A(t) + B(t)) \alpha + M(t)^{-1} b(t). 
\end{align}
Due to the time regularity of the basis functions $\partial_t \hat{\w}_j = \partial_t (J^{-1}F) \hat{\phi}_j \in L^{\infty}(I, H^1(\Omega(0))^d)$ the right-hand side in \eqref{Galerkin3} is Lipschitz. Hence, the Picard-Lindel\"of theorem guarantees a unique solution to \eqref{GalerkinALE}.

(iii) \textit{A priori estimate}: We test \eqref{RedStokesALE} with $\hat{\w} = \hat{\u}_l$. After some basic calculus, we obtain the system
\begin{eqnarray*}
\begin{aligned}
 \big(\partial_t (J^{1/2}\hat{\u}_l), J^{1/2}&\hat{\u}_l\big)_{\Omega(0)}
  -\left(\partial_t (J^{1/2})\hat{\u}_l + J^{1/2}F^{-1} \partial_t \T \cdot \hat\nabla \hat \u_l, J^{1/2} \hat \u_l\right)_{\Omega(0)}\\ 
 &+  (J^{1/2} \hat\nabla \hat \u_l F^{-1}, J^{1/2}\hat\nabla \hat \u_l F^{-1})_{\Omega(0)}
 = (J^{1/2} \hat \f, J^{1/2}\hat \u_l)_{\Omega(0)}, 
\end{aligned}
\end{eqnarray*}
where we have skipped the dependencies of $J$ and $F$ on time for better readability. Integration in time gives the estimate
 \begin{eqnarray*}
\begin{aligned}
 \big\| J^{1/2}(t_{\rm fin})&\hat{\u}_l(t_{\rm fin})\big\|_{\Omega(0)}^2
 +  \int_0^{t_{\rm fin}} \left\| J^{1/2} \hat\nabla \hat \u_l F^{-1}\right\|_{\Omega(0)}^2\, dt\\
 &\leq \left\|\hat{\u}_l(0)\right\|_{\Omega(0)}^2
 +  c\|\partial_t \T\|_{W^{1,\infty}(\Omega(0))} \int_0^{t_{\rm fin}} \|J^{1/2} \hat{\u}_l \|_{\Omega(0)}^2 \, dt 
 + c\int_0^{t_{\rm fin}} \left\|J^{1/2} \hat \f\right\|_{\Omega(0)}^2\, dt. 
\end{aligned}
\end{eqnarray*}
Using Gronwall's lemma, we obtain the first a priori estimate
 \begin{eqnarray}\label{firstAPriori}
\begin{aligned}
 \big\| J^{1/2}({t_{\rm fin}})&\hat{\u}_l({t_{\rm fin}})\big\|_{\Omega(0)}^2
 +  \int_0^{t_{\rm fin}} \left\| J^{1/2} \hat\nabla \hat \u_l F^{-1}\right\|_{\Omega(0)}^2\, dt\\
 &\leq c\exp\left(c\|\partial_t \T\|_{W^{1,\infty}(\Omega(0))} t_{\rm fin}\right) \left(\left\|\hat{\u}_l(0)\right\|_{\Omega(0)}^2
 + c\int_0^{t_{\rm fin}} \left\|J^{1/2} \hat \f\right\|_{\Omega(0)}^2\, dt. \right)
\end{aligned}
\end{eqnarray}
This implies that $\hat{\u}_l$ is bounded in $L^{\infty}(I,L^2(\Omega(0))^d)$ and $L^2(I,\hat{\cal V}_0)$. This implies the existence of convergent subsequences and limit functions $\hat{u},\hat{u}^*$ in the following sense
\begin{eqnarray}
\begin{aligned}\label{uconv}
\hat{\u}_{l'} &\to \hat{\u}^* \quad \text{weak star in } L^{\infty}(I,L^2(\Omega(0))^d),\\
\hat{\u}_{l'} &\to \hat{\u} \quad \text{ weakly in } L^2(I,\hat{\cal V}_0)) \text{ and strongly in } L^2(I,L^2(\Omega(0))^d).  
\end{aligned}
\end{eqnarray}
It is not difficult to prove that $\hat{\u}^* = \hat{\u}$, see~\cite{Temam2000}, Section III.1.3.\\
 (iv) \textit{A priori estimate for the time derivative}: In principle, we would like to test~\eqref{GalerkinALE} with $\partial_t \hat{u}_l$. Unfortunately, this is not possible, as in general $\partial_t \hat{u}_l\not\in \text{span}(\hat{\w}_1,...,\hat{\w}_l)$ due to the time-dependence of the basis functions. Instead, we can test with $J^{-1}F^T \partial_t (JF^{-T}\hat{\u}_l)$, as
 \begin{align*}
 J^{-1}F \partial_t (JF^{-1}\hat{\u}_l) 
 = \sum_{j=1}^l J^{-1}F \partial_t (JF^{-1}\alpha_j \hat{\w}_j)
 &= \sum_{j=1}^l J^{-1}F \partial_t (\alpha_j \hat{\phi}_j)
 = \sum_{j=1}^l \alpha_j'J^{-1}F \hat{\phi}_j
 = \sum_{j=1}^l \alpha_j' \hat{\w}_j.
 \end{align*} 
 We obtain
 \begin{eqnarray}
\begin{aligned}
 \big(\partial_t \hat \u_l, F \partial_t (JF^{-1}&\hat{\u}_l)\big)_{\Omega(0)} -
 \left(\partial_t \T \cdot \hat\nabla \hat \u_l , \partial_t (JF^{-1}\hat{\u}_l)\right)_{\Omega(0)} \\
 &+  (J \hat\nabla \hat \u_l F^{-1}, \hat\nabla (J^{-1}F \partial_t (JF^{-1}\hat{\u}_l) F^{-1})_{\Omega(0)} 
 = ( \hat \f, F \partial_t (JF^{-1}\hat{\u}_l))_{\Omega(0)}.
 \end{aligned}
\end{eqnarray}
The third term on the left-hand side is well-defined under the regularity assumptions stated, as $JF^{-1}$ is the cofactor matrix to $F$, which
can be written in terms of $T$. Using the product rule, we see that the first term on the left-hand side is bounded below by
\begin{align*}
\big(\partial_t \hat \u_l, F \partial_t (JF^{-1}\hat{\u}_l)\big)_{\Omega(0)} &= \big(\partial_t \hat \u_l, J \partial_t \hat{\u}_l\big)_{\Omega(0)} + \big(\partial_t \hat \u_l, F \partial_t (JF^{-1})\hat{\u}_l\big)_{\Omega(0)}\\
&\geq \left\|J^{1/2}\partial_t \hat \u_l\right\|_{\Omega(0)}^2 -c(\T) \left\|J^{1/2}\partial_t \hat \u_l\right\|_{\Omega(0)} \left\|J^{1/2}\hat \u_l\right\|_{\Omega(0)}
\end{align*}
For the third term on the left-hand side, we have
\begin{align*}
(J \hat\nabla &\hat \u_l F^{-1}, \hat\nabla (J^{-1}F \partial_t (JF^{-1}\hat{\u}_l) F^{-1})_{\Omega(0)} \\
&= (J \hat\nabla \hat \u_l F^{-1}, \hat\nabla \partial_t \hat{\u}_l F^{-1})_{\Omega(0)} 
+ (J \hat\nabla \hat \u_l F^{-1}, \hat\nabla\left( J^{-1}F \partial_t (JF^{-1})\hat{\u}_l\right) F^{-1})_{\Omega(0)} \\
&= \left(J^{1/2} \hat\nabla \hat \u_l F^{-1}, \hat \partial_t \left(J^{1/2}\hat{\nabla} \hat{\u}_l F^{-1}\right)\right)_{\Omega(0)} 
- \left(J^{1/2} \hat\nabla \hat \u_l F^{-1}, \hat\nabla\left( \hat{\u}_l \partial_t (J^{1/2}F^{-1})\right)\right)_{\Omega(0)} 
\\
&\qquad\qquad+ (J \hat\nabla \hat \u_l F^{-1}, \hat\nabla\left( J^{-1}F \partial_t (JF^{-1})\hat{\u}_l\right) F^{-1})_{\Omega(0)}\\
&\geq \frac{1}{2} \partial_t \left\|J^{1/2} \hat\nabla \hat \u_l F^{-1}\right\|_{\Omega(0)}^2
- c(\T) \left\|J^{1/2} \hat\nabla \hat \u_l F^{-1}\right\|_{\Omega(0)}^2 
\end{align*}
Using a similar argumentation and Young's inequality, we can show the bounds
\begin{align*}
\left(\partial_t \T \cdot \hat\nabla \hat \u_l , \partial_t (JF^{-1}\hat{\u}_l)\right)_{\Omega(0)} &\leq c(\T) \left\|J^{1/2} \hat\nabla \hat \u_l F^{-1}\right\|_{\Omega(0)}^2 + 
\frac{1}{4} \left\|J^{1/2}\partial_t \hat \u_l\right\|_{\Omega(0)}^2 \\
\left( \hat \f, F \partial_t (JF^{-1}\hat{\u}_l)\right)_{\Omega(0)} &\leq c(\T) \left(\left\|J^{1/2}\hat \f\right\|_{\Omega(0)}^2  + \left\|J^{1/2} \hat\nabla \hat \u_l F^{-1}\right\|_{\Omega(0)}^2\right) +  \frac{1}{4}\left\|J^{1/2}\partial_t \hat \u_l\right\|_{\Omega(0)}^2.
\end{align*}
Integration over $t\in I$ in \eqref{GalerkinALE} gives the estimate
\begin{align*}
&\left\|J^{1/2}(t_{\rm fin}) \hat\nabla \hat \u_l(t_{\rm fin}) F^{-1}(t_{\rm fin})\right\|_{\Omega(0)}^2 
+ \int_0^{t_{\rm fin}}
\left\|J^{1/2}\partial_t \hat \u_l\right\|_{\Omega(0)}^2 \, dt \\
&\qquad\qquad\qquad\leq  \left\|\hat\nabla \hat \u_l(0)\right\|_{\Omega(0)}^2 
+c(\T) \int_0^{t_{\rm fin}} \left\|J^{1/2}\hat \f\right\|_{\Omega(0)}^2 + \left\|J^{1/2} \hat\nabla \hat \u_l F^{-1}\right\|_{\Omega(0)}^2\,  dt.
\end{align*}
Using Gronwall's lemma we obtain
\begin{align*}
&\left\|J^{1/2}(t_{\rm fin}) \hat\nabla \hat \u_l(t_{\rm fin}) F^{-1}(t_{\rm fin})\right\|_{\Omega(0)}^2 
+ \int_0^{t_{\rm fin}}
\left\|J^{1/2}\partial_t \hat \u_l\right\|_{\Omega(0)}^2 \, dt \\
&\qquad\qquad\qquad\qquad\leq  c \exp(c(\T)t_{\rm fin}) \left( \left\|\hat\nabla \hat \u_l(0)\right\|_{\Omega(0)}^2 + \int_0^{t_{\rm fin}} \left\|J^{1/2}\hat \f\right\|_{\Omega(0)}^2 \,dt\right).
\end{align*}
This shows the boundedness of $\partial_t \hat{\u}_l$ in $L^2(I,L^2(\Omega(0))^d)$
and the convergence of a subsequence (see Temam~\cite{Temam2000}, Proposition III.1.2, for the details)
\begin{align}\label{dtuconv}
\partial_t \hat{\u}_{l'} \to \partial_t \hat{\u} \qquad \text{ weakly in } L^2(I,L^2(\Omega(0))^d).
\end{align}
(v) \textit{Conclusion}: The a priori bounds shown in (ii) and (iii) and the resulting convergence behaviour allows us to pass to the limit $l\to\infty$ in \eqref{GalerkinALE}. The convergences \eqref{uconv} and \eqref{dtuconv} imply that $\hat{\u}_l(x,0) \to \hat{\u}^0\, (l\to\infty)$. We find that the limit $\hat{u}$ is a solution to~\eqref{RedStokesALE}. Uniqueness is easily proven by testing \eqref{RedStokesALE} with $\hat{v}=\hat{u}$ and the a priori estimate~\eqref{firstAPriori}. Due to the equivalence of \eqref{RedStokesALE} and \eqref{RedStokes}, the pullback $u=\hat{u}\circ T^{-1}\in {\cal V}_{0,I}$ is the unique solution to \eqref{RedStokes}.

(vi) \textit{Pressure}: Finally, the unique existence of a pressure for a.e. $t\in I$ follows
by {\sfrei showing the existence of a weak pressure gradient that fulfils}
\begin{align*}
 \text{grad } p(t) = \f(t) +  \Delta \u(t) - \partial_t \u(t) \quad \text{in } \Omega(t)
\end{align*}
{\sfrei using} the de Rham theorem. We refer to~\cite{Temam2000}, Proposition III.1.2, for the details.
\end{proof}

\bibliographystyle{plainnat}

\end{document}